\documentclass[reqno]{amsart}

\usepackage{amssymb,amsmath,stmaryrd,txfonts,mathrsfs,amsthm}
\usepackage{comment}
\usepackage[neveradjust]{paralist}
\usepackage{adjustbox}
\usepackage{mathtools}
\usepackage{multirow}
\usepackage[outline]{contour}
\contourlength{1.2pt}
\usepackage{tikz}
\usetikzlibrary{intersections,arrows.meta,calc,quotes,math,decorations.pathreplacing,decorations.markings,cd,arrows,positioning,fit,matrix,
shapes.geometric,external,decorations.pathmorphing,backgrounds,circuits,circuits.ee.IEC,shapes}
\usepackage{circuitikz}
\usepackage[foot]{amsaddr}
\usepackage{graphicx}
\usepackage{stackrel}

\usepackage{xcolor}
\usepackage{framed,color}

\definecolor{shadecolor}{rgb}{1,0.8,0.3}
\definecolor{myurlcolor}{rgb}{0.3,0,0}
\definecolor{mycitecolor}{rgb}{0,0,0.7}
\definecolor{myrefcolor}{rgb}{0,0,0.7}
\definecolor{hyperrefcolor}{rgb}{0.3,0,0}

\usepackage[pagebackref]{hyperref}
\hypersetup{
	colorlinks,
	linkcolor={hyperrefcolor},
	citecolor={hyperrefcolor},
	urlcolor={hyperrefcolor}
}

\definecolor{rewritecolor}{rgb}{0,.9,1}
\tikzset{rewritenode/.style={shape=circle,fill=rewritecolor,scale=0.25,font=\Huge}}
\tikzset{RWopen/.style={shape=circle,draw=black,fill=white,scale=0.5,font=\Huge}}
\tikzset{RWclosed/.style={shape=circle,fill=black,scale=0.5,font=\Huge}}
\tikzset{CDnode/.style={shape=circle,fill=white,scale=.5}}
\makeatletter
\let\ea\expandafter

\pgfdeclarelayer{edgelayer}
\pgfdeclarelayer{nodelayer}
\pgfsetlayers{edgelayer,nodelayer,main}

\tikzset{tick/.style={postaction={decorate,decoration={markings,
mark=at position 0.4 with {\draw[-] (0,.4ex) -- (0,-.4ex);}}}}}

\definecolor{lblue}{rgb}{0,250,255}
\tikzstyle{species}=[circle,fill=yellow,draw=black,scale=1.15]
\tikzstyle{transition}=[rectangle,fill=lblue,draw=black,scale=1.15]
\tikzstyle{inarrow}=[->, >=stealth, shorten >=.03cm,line width=1.5]
\tikzstyle{empty}=[circle,fill=none, draw=none]
\tikzstyle{inputdot}=[circle,fill=purple,draw=purple, scale=.25]
\tikzstyle{inputarrow}=[->,draw=purple, shorten >=.05cm]
\tikzstyle{simple}=[-,draw=purple,line width=1.000]
\tikzstyle{none}=[inner sep=0pt]

\tikzset{->-/.style={decoration={
  markings,
  mark=at position .5 with {\arrow{>}}},postaction={decorate}}}

\tikzcdset{every label/.append style = {font = \normalsize}}

\pgfdeclarelayer{edgelayer}
\pgfdeclarelayer{nodelayer}
\pgfsetlayers{edgelayer,nodelayer,main}

\tikzset{font=\footnotesize}
\tikzstyle{none}=[inner sep=0pt]
\tikzstyle{connection}=[circuit symbol open,
    circuit symbol size=width .5 height .5,
    shape=circle ee,
    inner sep=0.25\pgflinewidth]
\tikzstyle{bdot}=[circle, fill=black, draw, inner sep=1.5pt, anchor=center]
\tikzstyle{circ}=[circle,fill=black,draw,inner sep=3pt]

\newcommand{\mult}[1]
{
\begin{aligned}
    \resizebox{#1}{!}{
\begin{tikzpicture}
	\begin{pgfonlayer}{nodelayer}
		\node [style=none] (0) at (1, -0) {};
		\node [style=circ] (1) at (0.125, -0) {};
		\node [style=none] (2) at (-1, 0.5) {};
		\node [style=none] (3) at (-1, -0.5) {};
	\end{pgfonlayer}
	\begin{pgfonlayer}{edgelayer}
		\draw[line width=2pt] (0.center) to (1.center);
		\draw[line width=2pt] [in=0, out=120, looseness=1.20] (1.center) to (2.center);
		\draw[line width=2pt] [in=0, out=-120, looseness=1.20] (1.center) to (3.center);
	\end{pgfonlayer}
      \end{tikzpicture}}
\end{aligned}
}

\newcommand{\unit}[1]
{
  \begin{aligned}
    \resizebox{#1}{!}{
\begin{tikzpicture}
	\begin{pgfonlayer}{nodelayer}
		\node [style=none] (spaceup) at (0, 0.5) {};
		\node [style=none] (spacedown) at (0, -0.5) {};
		\node [style=none] (0) at (1, -0) {};
		\node [style=none] (1) at (-1, -0) {};
		\node [style=circ] (2) at (0, -0) {};
	\end{pgfonlayer}
	\begin{pgfonlayer}{edgelayer}
		\draw[line width=2pt] (0.center) to (2);
	\end{pgfonlayer}
      \end{tikzpicture}}
  \end{aligned}
}

\newcommand{\comult}[1]
{
\begin{aligned}
    \resizebox{#1}{!}{
\begin{tikzpicture}
	\begin{pgfonlayer}{nodelayer}
		\node [style=none] (0) at (-1, -0) {};
		\node [style=circ] (1) at (-0.125, -0) {};
		\node [style=none] (2) at (1, 0.5) {};
		\node [style=none] (3) at (1, -0.5) {};
	\end{pgfonlayer}
	\begin{pgfonlayer}{edgelayer}
		\draw[line width=2pt] (0.center) to (1.center);
		\draw[line width=2pt] [in=180, out=60, looseness=1.20] (1.center) to (2.center);
		\draw[line width=2pt] [in=180, out=-60, looseness=1.20] (1.center) to (3.center);
	\end{pgfonlayer}
      \end{tikzpicture}}
\end{aligned}
}

\newcommand{\counit}[1]
{
  \begin{aligned}
    \resizebox{#1}{!}{
\begin{tikzpicture}
	\begin{pgfonlayer}{nodelayer}
		\node [style=none] (spaceup) at (0, 0.5) {};
		\node [style=none] (spacedown) at (0, -0.5) {};
		\node [style=none] (0) at (-1, -0) {};
		\node [style=none] (1) at (1, -0) {};
		\node [style=circ] (2) at (0, -0) {};
	\end{pgfonlayer}
	\begin{pgfonlayer}{edgelayer}
		\draw[line width=2pt] (0.center) to (2);
	\end{pgfonlayer}
      \end{tikzpicture}}
  \end{aligned}
}

\newcommand{\idone}[1]
{
  \begin{aligned}
    \resizebox{#1}{!}{
     \begin{tikzpicture}
	\begin{pgfonlayer}{nodelayer}
		\node [style=none] (0) at (-1, -0) {};
		\node [style=none] (1) at (1, -0) {};
		\node [style=none] (2) at (0, 0.5) {};
		\node [style=none] (3) at (0, -0.5) {};
	\end{pgfonlayer}
	\begin{pgfonlayer}{edgelayer}
		\draw[line width=2pt] (1.center) to (0.center);
	\end{pgfonlayer}
\end{tikzpicture} 
    }
  \end{aligned}
}

\newcommand{\swap}[1]
{
  \begin{aligned}
    \resizebox{#1}{!}{
\begin{tikzpicture}
	\begin{pgfonlayer}{nodelayer}
		\node [style=none] (2) at (-0.5, -0.5) {};
		\node [style=none] (3) at (-2, 0.5) {};
		\node [style=none] (4) at (-0.5, 0.5) {};
		\node [style=none] (5) at (-2, -0.5) {};
	\end{pgfonlayer}
	\begin{pgfonlayer}{edgelayer}
		\draw[line width=2pt] [in=180, out=0, looseness=1.00] (3.center) to (2.center);
		\draw[line width=2pt] [in=0, out=180, looseness=1.00] (4.center) to (5.center);
	\end{pgfonlayer}
\end{tikzpicture}
    }
  \end{aligned}
}
\newcommand{\assocl}[1]
{
  \begin{aligned}
    \resizebox{#1}{!}{
\begin{tikzpicture}
	\begin{pgfonlayer}{nodelayer}
		\node [style=circ] (0) at (0.125, -0) {};
		\node [style=none] (1) at (-1, 0.5) {};
		\node [style=none] (2) at (-1, -0.5) {};
		\node [style=none] (3) at (0, -1) {};
		\node [style=none] (4) at (2.25, -0.5) {};
		\node [style=none] (5) at (0.25, -0) {};
		\node [style=circ] (6) at (1.25, -0.5) {};
		\node [style=none] (7) at (-1, -1) {};
	\end{pgfonlayer}
	\begin{pgfonlayer}{edgelayer}
		\draw[line width=2pt] [in=0, out=120, looseness=1.20] (0.center) to (1.center);
		\draw[line width=2pt] [in=0, out=-120, looseness=1.20] (0.center) to (2.center);
		\draw[line width=2pt] (4.center) to (6);
		\draw[line width=2pt] [in=0, out=120, looseness=1.20] (6) to (5.center);
		\draw[line width=2pt] [in=0, out=-120, looseness=1.20] (6) to (3.center);
		\draw[line width=2pt] (3.center) to (7.center);
	\end{pgfonlayer}
      \end{tikzpicture}}
  \end{aligned}
}

\newcommand{\assocr}[1]
{
  \begin{aligned}
    \resizebox{#1}{!}{
\begin{tikzpicture}
	\begin{pgfonlayer}{nodelayer}
		\node [style=circ] (0) at (0.125, -0.5) {};
		\node [style=none] (1) at (-1, -1) {};
		\node [style=none] (2) at (-1, 0) {};
		\node [style=none] (3) at (0, 0.5) {};
		\node [style=none] (4) at (2.25, 0) {};
		\node [style=none] (5) at (0.25, -0.5) {};
		\node [style=circ] (6) at (1.25, 0) {};
		\node [style=none] (7) at (-1, 0.5) {};
	\end{pgfonlayer}
	\begin{pgfonlayer}{edgelayer}
		\draw[line width=2pt] [in=0, out=-120, looseness=1.20] (0.center) to (1.center);
		\draw[line width=2pt] [in=0, out=120, looseness=1.20] (0.center) to (2.center);
		\draw[line width=2pt] (4.center) to (6);
		\draw[line width=2pt] [in=0, out=-120, looseness=1.20] (6) to (5.center);
		\draw[line width=2pt] [in=0, out=120, looseness=1.20] (6) to (3.center);
		\draw[line width=2pt] (3.center) to (7.center);
	\end{pgfonlayer}
      \end{tikzpicture}}
  \end{aligned}
}

\newcommand{\unitl}[1]
{
  \begin{aligned}
    \resizebox{#1}{!}{
\begin{tikzpicture}
	\begin{pgfonlayer}{nodelayer}
		\node [style=none] (0) at (1, -0) {};
		\node [style=circ] (1) at (0.125, -0) {};
		\node [style=circ] (2) at (-1, 0.5) {};
		\node [style=none] (3) at (-1, -0.5) {};
		\node [style=none] (4) at (-2, -0.5) {};
	\end{pgfonlayer}
	\begin{pgfonlayer}{edgelayer}
		\draw[line width=2pt] (0.center) to (1.center);
		\draw[line width=2pt] [in=0, out=120, looseness=1.20] (1.center) to (2.center);
		\draw[line width=2pt] [in=0, out=-120, looseness=1.20] (1.center) to (3.center);
		\draw[line width=2pt] (4.center) to (3.center);
	\end{pgfonlayer}
\end{tikzpicture}
    }
  \end{aligned}
}

\newcommand{\unitr}[1]
{
  \begin{aligned}
    \resizebox{#1}{!}{
\begin{tikzpicture}
	\begin{pgfonlayer}{nodelayer}
		\node [style=none] (0) at (1, -0) {};
		\node [style=circ] (1) at (0.125, -0) {};
		\node [style=circ] (2) at (-1, -0.5) {};
		\node [style=none] (3) at (-1, 0.5) {};
		\node [style=none] (4) at (-2, 0.5) {};
	\end{pgfonlayer}
	\begin{pgfonlayer}{edgelayer}
		\draw[line width=2pt] (0.center) to (1.center);
		\draw[line width=2pt] [in=0, out=-120, looseness=1.20] (1.center) to (2.center);
		\draw[line width=2pt] [in=0, out=120, looseness=1.20] (1.center) to (3.center);
		\draw[line width=2pt] (4.center) to (3.center);
	\end{pgfonlayer}
\end{tikzpicture}
    }
  \end{aligned}
}

\newcommand{\commute}[1]
{
  \begin{aligned}
    \resizebox{#1}{!}{
\begin{tikzpicture}
	\begin{pgfonlayer}{nodelayer}
		\node [style=none] (0) at (1.25, -0) {};
		\node [style=circ] (1) at (0.375, -0) {};
		\node [style=none] (2) at (-0.5, -0.5) {};
		\node [style=none] (3) at (-2, 0.5) {};
		\node [style=none] (4) at (-0.5, 0.5) {};
		\node [style=none] (5) at (-2, -0.5) {};
	\end{pgfonlayer}
	\begin{pgfonlayer}{edgelayer}
		\draw[line width=2pt] (0.center) to (1.center);
		\draw[line width=2pt] [in=0, out=-120, looseness=1.20] (1.center) to (2.center);
		\draw[line width=2pt] [in=180, out=0, looseness=1.00] (3.center) to (2.center);
		\draw[line width=2pt] [in=0, out=120, looseness=1.20] (1.center) to (4.center);
		\draw[line width=2pt] [in=0, out=180, looseness=1.00] (4.center) to (5.center);
	\end{pgfonlayer}
\end{tikzpicture}
    }
  \end{aligned}
}

\newcommand{\cocommute}[1]
{
  \begin{aligned}
    \resizebox{#1}{!}{
\begin{tikzpicture}
	\begin{pgfonlayer}{nodelayer}
		\node [style=none] (0) at (-2, -0) {};
		\node [style=circ] (1) at (-1.125, -0) {};
		\node [style=none] (2) at (-0.25, -0.5) {};
		\node [style=none] (3) at (1.25, 0.5) {};
		\node [style=none] (4) at (-0.25, 0.5) {};
		\node [style=none] (5) at (1.25, -0.5) {};
	\end{pgfonlayer}
	\begin{pgfonlayer}{edgelayer}
		\draw[line width=2pt] (0.center) to (1.center);
		\draw[line width=2pt] [in=180, out=-60, looseness=1.20] (1.center) to (2.center);
		\draw[line width=2pt] [in=0, out=180, looseness=1.00] (3.center) to (2.center);
		\draw[line width=2pt] [in=180, out=60, looseness=1.20] (1.center) to (4.center);
		\draw[line width=2pt] [in=180, out=0, looseness=1.00] (4.center) to (5.center);
	\end{pgfonlayer}
\end{tikzpicture}
    }
  \end{aligned}
}
\newcommand{\frobs}[1]
{
  \begin{aligned}
    \resizebox{#1}{!}{
\begin{tikzpicture}
	\begin{pgfonlayer}{nodelayer}
		\node [style=none] (0) at (-1.5, 0.5) {};
		\node [style=circ] (1) at (-0.75, 0.5) {};
		\node [style=none] (2) at (0.25, -0) {};
		\node [style=none] (3) at (0.25, 1) {};
		\node [style=circ] (4) at (1, -0.5) {};
		\node [style=none] (5) at (0, -0) {};
		\node [style=none] (6) at (1.75, -0.5) {};
		\node [style=none] (7) at (0, -1) {};
		\node [style=none] (8) at (1.75, 1) {};
		\node [style=none] (9) at (-1.5, -1) {};
	\end{pgfonlayer}
	\begin{pgfonlayer}{edgelayer}
		\draw[line width=2pt] [in=180, out=-60, looseness=1.20] (1) to (2.center);
		\draw[line width=2pt] [in=180, out=60, looseness=1.20] (1) to (3.center);
		\draw[line width=2pt] (0.center) to (1);
		\draw[line width=2pt] (6.center) to (4);
		\draw[line width=2pt] [in=0, out=120, looseness=1.20] (4) to (5.center);
		\draw[line width=2pt] [in=0, out=-120, looseness=1.20] (4) to (7.center);
		\draw[line width=2pt] (3.center) to (8.center);
		\draw[line width=2pt] (7.center) to (9.center);
	\end{pgfonlayer}
\end{tikzpicture}
    }
  \end{aligned}
}

\newcommand{\frobx}[1]
{
  \begin{aligned}
    \resizebox{#1}{!}{
\begin{tikzpicture}
	\begin{pgfonlayer}{nodelayer}
		\node [style=circ] (0) at (-0.5, -0) {};
		\node [style=none] (1) at (-1.5, -0.5) {};
		\node [style=none] (2) at (-1.5, 0.5) {};
		\node [style=circ] (3) at (0.5, -0) {};
		\node [style=none] (4) at (1.5, -0.5) {};
		\node [style=none] (5) at (1.5, 0.5) {};
	\end{pgfonlayer}
	\begin{pgfonlayer}{edgelayer}
		\draw[line width=2pt] [in=0, out=-120, looseness=1.20] (0.center) to (1.center);
		\draw[line width=2pt] [in=0, out=120, looseness=1.20] (0.center) to (2.center);
		\draw[line width=2pt] [in=180, out=-60, looseness=1.20] (3) to (4.center);
		\draw[line width=2pt] [in=180, out=60, looseness=1.20] (3) to (5.center);
		\draw[line width=2pt] (0) to (3);
	\end{pgfonlayer}
\end{tikzpicture}
    }
  \end{aligned}
}

\newcommand{\frobz}[1]
{
  \begin{aligned}
    \resizebox{#1}{!}{
\begin{tikzpicture}
	\begin{pgfonlayer}{nodelayer}
		\node [style=none] (0) at (1.75, 0.5) {};
		\node [style=circ] (1) at (1, 0.5) {};
		\node [style=none] (2) at (0, -0) {};
		\node [style=none] (3) at (0, 1) {};
		\node [style=circ] (4) at (-0.75, -0.5) {};
		\node [style=none] (5) at (0.25, -0) {};
		\node [style=none] (6) at (-1.5, -0.5) {};
		\node [style=none] (7) at (0.25, -1) {};
		\node [style=none] (8) at (-1.5, 1) {};
		\node [style=none] (9) at (1.75, -1) {};
	\end{pgfonlayer}
	\begin{pgfonlayer}{edgelayer}
		\draw[line width=2pt] [in=0, out=-120, looseness=1.20] (1) to (2.center);
		\draw[line width=2pt] [in=0, out=120, looseness=1.20] (1) to (3.center);
		\draw[line width=2pt] (0.center) to (1);
		\draw[line width=2pt] (6.center) to (4);
		\draw[line width=2pt] [in=180, out=60, looseness=1.20] (4) to (5.center);
		\draw[line width=2pt] [in=180, out=-60, looseness=1.20] (4) to (7.center);
		\draw[line width=2pt] (3.center) to (8.center);
		\draw[line width=2pt] (7.center) to (9.center);
	\end{pgfonlayer}
\end{tikzpicture}
    }
  \end{aligned}
}

\newcommand{\spec}[1]
{
  \begin{aligned}
    \resizebox{#1}{!}{
\begin{tikzpicture}
	\begin{pgfonlayer}{nodelayer}
		\node [style=none] (0) at (1.75, -0) {};
		\node [style=circ] (1) at (0.75, -0) {};
		\node [style=none] (2) at (0, -0.5) {};
		\node [style=none] (3) at (0, 0.5) {};
		\node [style=circ] (4) at (-0.75, -0) {};
		\node [style=none] (5) at (0, -0.5) {};
		\node [style=none] (6) at (-1.75, -0) {};
		\node [style=none] (7) at (0, 0.5) {};
	\end{pgfonlayer}
	\begin{pgfonlayer}{edgelayer}
		\draw[line width=2pt] (0.center) to (1.center);
		\draw[line width=2pt] [in=0, out=-120, looseness=1.20] (1.center) to (2.center);
		\draw[line width=2pt] [in=0, out=120, looseness=1.20] (1.center) to (3.center);
		\draw[line width=2pt] (6.center) to (4);
		\draw[line width=2pt] [in=180, out=-60, looseness=1.20] (4) to (5.center);
		\draw[line width=2pt] [in=180, out=60, looseness=1.20] (4) to (7.center);
	\end{pgfonlayer}
\end{tikzpicture}
    }
  \end{aligned}
}

\newcommand{\cuptwo}[1]
{
\begin{aligned}
\begin{tikzpicture}[circuit ee IEC, set resistor graphic=var resistor IEC
      graphic,scale=.5]
\scalebox{1}{
		\node [style=none] (0) at (1, -0) {};
		\node [style=none] (1) at (0.125, -0) {};
		\node [style=none] (2) at (-1, 0.5) {};
		\node [style=none] (3) at (-1, -0.5) {};
		\node [style=none] (4) at (1, -0) {};
	\draw[line width = 1.5pt] (2) to [in=0, out=0,looseness = 4] (3);
}
      \end{tikzpicture}
\end{aligned}
}

\newcommand{\captwo}[1]
{
\begin{aligned}
\begin{tikzpicture}
[circuit ee IEC, set resistor graphic=var resistor IEC
      graphic,scale=.5]
\scalebox{1}{
		\node [style=none] (0) at (-.75, -0) {};
		\node [style=none] (1) at (-0.125, -0) {};
		\node [style=none] (2) at (1, 0.5) {};
		\node [style=none] (3) at (1, -0.5) {};
		\node [style=none] (4) at (-1, -0) {};
		\node [style=none] (5) at (-0.5, -.5) {};
		\node [style=none] (6) at (-0.5, .5) {};
	\draw[line width = 1.5pt] (2) to [in=180, out=180,looseness = 4] (3);
}
      \end{tikzpicture}
\end{aligned}
}

\def\mdef#1#2{\ea\ea\ea\gdef\ea\ea\noexpand#1\ea{\ea\ensuremath\ea{#2}}}
\def\alwaysmath#1{\ea\ea\ea\global\ea\ea\ea\let\ea\ea\csname your@#1\endcsname\csname #1\endcsname
  \ea\def\csname #1\endcsname{\ensuremath{\csname your@#1\endcsname}}}

\mdef\fahat{\hat{\fa}}

\alwaysmath{ell}
\alwaysmath{infty}
\alwaysmath{odot}
\def\frc#1/#2.{\frac{#1}{#2}}   % \frc x^2+1 / x^2-1 .
\mdef\ten{\mathrel{\otimes}}

\newcommand{\simRightarrow}{\xRightarrow{\raisebox{-3pt}[0pt][0pt]{\ensuremath{\sim}}}}

\newcommand*{\graysquare}{\textcolor{lightgray}{\blacksquare}}
\newcommand{\slashedrightarrow}{\mathrel {\mkern 5mu\vcenter {\hbox {$\shortmid $}}\mkern -8mu{\to }}}

\let\maps\colon

\newenvironment{proofsketch}{%
  \proof}{\endproof}
  
\def\defthm#1#2{%
  \newtheorem{#1}{#2}[section]%
  \expandafter\def\csname #1autorefname\endcsname{#2}%
  \expandafter\let\csname c@#1\endcsname\c@thm}
\newtheorem{thm}{Theorem}[section]

\defthm{cor}{Corollary}
\defthm{prop}{Proposition}
\defthm{lem}{Lemma}
\defthm{problem}{Challenge}
\defthm{claim}{Claim}
\defthm{hyp}{Hypothesis}
\defthm{exercise}{Exercise}
\defthm{fact}{Fact}
\defthm{defn}{Definition}
\defthm{tentativedefn}{Tentative Definition}

\theoremstyle{definition}

\defthm{notn}{Notation}

\theoremstyle{remark}
\defthm{rmk}{Remark}
\defthm{eg}{Example}

\mdef\fchk{\check{f}}

\definecolor{purple(x11)}{rgb}{0.5, 0.0, 0.5}

\usepackage[all,2cell,cmtip]{xy}
\UseAllTwocells

\newcommand{\N}{\mathbb{N}}
\newcommand{\R}{\mathbb{R}}
\newcommand{\Ob}{\mathrm{Ob}}
\newcommand{\Mor}{\mathrm{Mor}}

\newcommand{\ca}{\mathsf}
\newcommand{\Set}{\ca{Set}}
\newcommand{\Vect}{\ca{Vect}}

\newcommand{\Petri}{\ca{Petri}}
\newcommand{\wg}{\ca{wg}}

\newcommand{\Fin}{\ca{Fin}}
\newcommand{\Rel}{\ca{Rel}}

\newcommand{\Csp}{\ca{Csp}}
\newcommand{\CMC}{\ca{CMC}}
\newcommand{\SSMC}{\ca{SSMC}}

\newcommand{\Disc}{\ca{Disc}}

\newcommand{\ParaDynam}{\ca{ParaDynam}}

\newcommand{\A}{\mathsf{A}}
\newcommand{\C}{\mathsf{C}}

\newcommand{\D}{\mathsf{D}}
\newcommand{\X}{\mathsf{X}}

\newcommand{\bicat}{\mathbf}

\newcommand{\bD}{\bicat{D}}
\newcommand{\Cat}{\bicat{Cat}}

\newcommand{\Rex}{\bicat{Rex}}

\newcommand{\SMC}{\bicat{SymMonCat}}

\newcommand{\double}[1]{\mathbf{\mathbb #1}}

\newcommand{\lCsp}{\double{Csp}}

\newcommand{\lOpen}{\double{Open}}

\newcommand{\lRel}{\double{Rel}}

\newcommand{\lSemiAlgRel}{\double{SemiAlgRel}}

\newcommand{\lC}{\double{C}}
\newcommand{\lD}{\double{D}}
\newcommand{\lE}{\double{E}}
\newcommand{\lF}{\double{F}}

\newcommand{\define}[1]{{\bf \boldmath{#1}}}

\newcommand{\inta}{\raisebox{.3\depth}{$\smallint\hspace{-.01in}$}}

\newcommand{\op}{\mathrm{op}}

\title[Double Categories of Open Systems: the Cospan Approach]{Double Categories of Open Systems: \\ the Cospan Approach}

\author{John\ C.\ Baez}
\address{School of Mathematics, University of Edinburgh, Edinburgh UK, EH9 3FD}
\email{baez@math.ucr.edu}

\begin{document}

\begin{abstract}
This is an overview of double categories of `open systems': systems that can interact with their environment.   We focus on the variable sharing paradigm, where we compose open systems by identifying variables.   This paradigm is often implemented using structured or decorated cospans.   We explain this approach using three main examples: open Petri nets, open dynamical systems, and open Petri nets with rates.  We compare the virtues of structured and decorated cospan double categories, and study their common features.   We show that any symmetric monoidal structured or decorated cospan double category comes with maps from two simpler double categories: its `exoskeleton' and its `outer shell'.   Finally,  we study the concept of `hypergraph double category', a kind of double category that should subsume structured and decorated cospans in a common framework for studying open systems in the variable sharing paradigm.
\end{abstract}

\maketitle

\setcounter{tocdepth}{1}
\tableofcontents

\section{Introduction}
\label{sec:introduction}

When writing his foundational papers on double categories with Marco Grandis, Robert Par\'e never expected that their work would be applied to open systems: that is, systems that can interact with their environment.   Open systems are fundamental to modern technology, so they have been studied intensively, but still not as much as `closed' systems, where we neglect the system's interaction with its environment.    Some aspects of open systems have only recently been formalized.   One is the \emph{composition} of open systems---building larger systems out of smaller parts.   Another is \emph{functorial semantics} for open systems, where various kinds of diagrams are used to describe open systems and the maps between them.   It turns out that double categories are well suited to both these purposes!  

Diagrams of open systems, such as electrical circuit diagrams, can often be seen as loose morphisms in some double category $\lD$.   Maps between diagrams are then 2-cells in $\lD$.   This double category $\lD$ serves as a \emph{syntax}.   A double functor $F$ from $\lD$ to some other double category $\lE$ then provides a \emph{semantics} that says what the diagrams `mean'.    Loose morphisms in $\lE$ are typically systems of equations of some kind---algebraic equations, differential equations, etc.---with certain variables singled out that can interact with the outside world.    

A number of developments were required to enable the double categorical approach to open physical systems.   One key step was Grandis and Par\'e's 1999 paper \cite{GP1} introducing double categories where composition is strictly associative in one direction (now called the `tight' direction) and weakly associative in the other (now called the `loose' direction).   By now this concept is the default notion of double category---and the only one we use here. 

Another important step was the rise of double categories where a tight morphism $f \maps x \to y$ can be turned into a loose morphism in two ways, its `companion' $f_\ast \maps x \slashedrightarrow y$ and its `conjoint' $f^\ast \maps y \slashedrightarrow x$.   While these have their roots in Wood's `proarrow equipments' \cite{Wood1,Wood2}, the companion and conjoint have universal properties that Par\'e and others elegantly stated in double categorical language \cite{DPP2010,GP2}.   Around 2007, Shulman \cite{Shulman2008} studied double categories with companions and conjoints under the name `framed bicategories'.   In 2010 he gave them another name: `fibrant double categories' \cite{Shulman2010}, because they are the same as double categories $\lD$ where the source and target maps going from the category $\lD_1$ of loose morphisms and 2-cells to the category  $\lD_0$  of objects and tight morphisms combine to give a functor $\lD_1 \to \lD_0 \times \lD_0$ that is a Grothendieck fibration.

A crucial step toward applying these ideas to open systems was the introduction of topological quantum field theories described as functors.  This expanded Lawvere's idea of functorial semantics \cite{Lawvere} in a radically non-cartesian direction, and it emphasized the important of compact closed categories, and higher analogues of these, for understanding physical systems.  For reviews of this line of work see \cite{BL,BS,Kock}.    

Another crucial influence was the work on spans and cospans initiated by Walters and collaborators, motivated in large part by computer science, especially transition systems.  This program was launched in the mid-1990s by Katis, Sabadini and Walters in three papers on spans \cite{KSW1,KSW2,KSW3}.  Later the dual cospan side was explored by Rosebrugh, Sabadini and Walters \cite{RSW1,RSW2,RSW3}.  Among these papers, \cite{RSW1} is especially important here because it described the universal property of the category of cospans of finite sets, and the category of open graphs, in terms of special commutative Frobenius monoids.  This idea was further studied by Lack \cite{Lack} and others \cite{BCR,CF}, and we shall see its influence here.

When these ideas came together, it was only a matter of time before double categories were used to study open systems.   But this general idea can be made more precise in numerous ways.   For example, there are various choices of what it means to compose open systems.   In the `input-output paradigm', when we compose two open systems, variables of one can affect variables of the other while not being directly affected by them.    In the `variable sharing paradigm',  we instead \emph{identify} variables of one system with variables of another.  

In recent years Libkind has formalized both these paradigms and begun to unify them \cite{Libkind}, Myers has studied the input-output paradigm using double categories \cite{Myers2021}, and together they have formalized the very concept of `paradigm' \cite{LM,MyersBook}.   To limit the scope of this paper, we only discuss the variable sharing paradigm, and emphasize the approach to this paradigm based on cospans.   We illustrate this using three examples of open systems: Petri nets, dynamical systems described by systems of first-order differential equations, and finally Petri nets with rates.   Thus, we omit or only briefly touch upon many other examples of the variable sharing paradigm, including:
\begin{itemize}
\item electrical circuits \cite{BC,BCR,BF},
\item Markov processes \cite{ASW1,ASW2,BC2018,BFP},
\item bond graphs \cite{Coya2017,CoyaThesis},
\item signal flow diagrams \cite{BE,BSZ,CoyaThesis,FRS},
\item stock-flow diagrams and system structure diagrams \cite{BLLOR,BLOP},
\item gene regulatory networks and causal loop diagrams \cite{AFKOPS,BCh,BLOP},
\item thermodynamics \cite{BLM,LynchThesis},
\item classical mechanics \cite{BWY,LynchThesis}.
\end{itemize}

In Section \ref{sec:open} we explain open systems and various ways to formalize them, leading up the double categorical approach.   In Section \ref{sec:structured_and_decorated} we introduce two ways to construct double categories of open systems: structured and decorated cospans.   We discuss the relative merits of these approaches, and raise some questions about the relation between them.    In Section \ref{sec:Petri} we illustrate structured cospan double categories and maps between them using the example of open Petri nets.   In Section \ref{sec:dynamical} we illustrate decorated cospan categories using the example of open dynamical systems.   We look at some applications to classical mechanics, and revisit the question of when to use structured and when to use decorated cospans.   In Section \ref{sec:Petri_with_rates}, we illustrate maps between decorated cospans double categories using the map sending open Petri nets with rates to open dynamical systems.   We also look at applications of structured and decorated cospans to software for modeling in the field of public health.   

In Section \ref{sec:variable_sharing}, we conclude with a deeper investigation of the variable sharing paradigm.  We find that some structures important in topological quantum field theory, such as commutative Frobenius monoids, also appear in the variable sharing paradigm.   Indeed, in any symmetric monoidal double category of structured or decorated cospans, every object is a \emph{categorified version} of a commutative Frobenius monoid.   This fact expresses the laws obeyed by the operations of bending, splitting and joining wires---where `wires' should be taken quite generally as routes along which variables can affect one another.   In topological quantum field theory, the same laws appear in the description of spacetime itself \cite{BL,KockBook}.

\subsection{Conventions}

In this paper, we use a sans-serif font like $\C$ for categories, boldface like $\mathbf{B}$ for bicategories or 2-categories, and blackboard bold like $\lD$ for double categories. For double categories with names having more than one letter, such as $\lCsp(\X)$, only the first letter is in blackboard bold.   We use $(\C,\otimes)$ to stand for a monoidal or perhaps symmetric monoidal category with $\otimes$ as its tensor product.

In this paper, `double category' means `pseudo double category' and `double functor' means `pseudo double functor'.  For details, see for example \cite[\S 7.1]{GP1} or \cite[A.2]{CourserThesis}.    Thus, a double category has objects, tight morphisms $f \maps x \to x'$ with a composition operation that is associative and unital, loose morphisms $F \maps x \slashedrightarrow y$ with a composition operation that is associative and unital only up to 2-isomorphisms that obey the pentagon and triangle identity, and 2-cells
\[
\scalebox{1.2}{
\begin{tikzpicture}[scale=1.2]
\node (A) at (0,0) {$x$};
\node (B) at (1,0) {$y$};
\node (A') at (0,-1) {$x'$};
\node (B') at (1,-1) {$y'$};
\path[->,font=\scriptsize,>=angle 90]
(A) edge [tick] node[above]{$F$} (B)
(A') edge [tick] node[below]{$G$} (B')
(A) edge node [left]{$f$} (A')
(B) edge node [right]{$g$} (B');
\end{tikzpicture}
}
\]
that compose horizontally and vertically, with the vertical (tight) composition being associative and unital, and the horizontal (loose) composition being associative and unital only up to 2-isomorphisms that obey the pentagon and triangle identity.   Most importantly, loose and tight composition of 2-cells obey the `interchange law' that makes a four-fold composite of 2-cells such as this well-defined:
\[
\scalebox{1.2}{
\begin{tikzpicture}[scale=1.2]
\node (A) at (0,0) {$x$};
\node (B) at (1,0) {$x'$};
\node (C) at (2,0) {$x''$};
\node (A') at (0,-1) {$y$};
\node (B') at (1,-1) {$y'$};
\node (C') at (2,-1) {$y''$};
\node (A'') at (0,-2) {$z$};
\node (B'') at (1,-2) {$z'$};
\node (C'') at (2,-2) {$z''$.};
\path[->,font=\scriptsize,>=angle 90]
(A) edge [tick] node[above]{$F$} (B)
(A') edge [tick] node[above]{$G$} (B')
(A'') edge [tick] node[below]{$H$} (B'')
(B) edge [tick] node[above]{$F'$} (C)
(B') edge [tick] node[above]{$G'$} (C')
(B'') edge [tick] node[below]{$H'$} (C'')
(A) edge node [left]{$f$} (A')
(A') edge node [left]{$g$} (A'')
(B) edge node [right]{$f'$} (B')
(B') edge node [right]{$g'$} (B'')
(C) edge node [right]{$f''$} (C')
(C') edge node [right]{$g''$} (C'');
\end{tikzpicture}
}
\]

\section{Open systems}
\label{sec:open}

An `open system' is a system that we treat as interacting with the outside world---as opposed to a `closed system', where interactions with the outside world are absent or at least neglected.  To a mathematician this may seem hopelessly vague, but physicists and engineers have various fairly precise frameworks for discussing `systems' of various kinds, and in many of these frameworks one can set up a distinction between open and closed systems.   For example,  in the realm of electrical circuits, a circuit with `terminals' or `ports' exposed to the outside world is an open system, while one without these is a closed system.

A striking fact is that so far, more mathematical labor has gone into proving results about \emph{closed} systems than open ones.  The main reason is that closed systems are easier to understand.    Thus, there is a lot of interesting new territory yet to be explored in the realm of open systems.   I believe category theory can greatly assist this exploration, and that it will also benefit from new ideas discovered in this exploration.   The reason is that category theory provides many tools for discussing open systems.  

We can combine two closed systems by `setting them side by side', letting each do its own thing, neither influencing nor being influenced by the other.   In this situation we can easily reduce the study of the composite system to the study of each separate part.   Open systems can be composed in more interesting ways, which create interactions between the two parts.   For example, we can take two electrical circuits and attach the terminals of one to the terminals of the other, forming a larger circuit.  More generally, we can think of any open system as having one or more `interfaces', through which it can interact with the outside world, and to compose open systems we attach them along an interface.   This is an idealization, but a useful one.  Given this, several ways of formalizing a chosen class of open systems spring to mind.    

One approach is to use a symmetric monoidal category where:

\begin{itemize}
\item an object is an `interface',
\item a morphism $F \maps x \to y$ is an `open system' with `left interface' $x$ and `right interface' $y$,
\item composing morphisms $F \maps x \to y$ and $G \maps y \to z$ is attaching the left interface of $F$ to the right interface of $G$, obtaining an open system $G \circ F \maps x \to z$,
\item tensoring two morphisms  $F \maps x \to y$ and $F' \maps x' \to y'$ is `setting two open systems side by side', giving an open system $F \otimes F' \maps x \otimes x' \to y \otimes y'$.
\end{itemize}

There are many questions that can be raised about this approach.   First, why should an open system have just two interfaces?   There is no reason it must.   Luckily, this can be addressed by the fact that we are working in a symmetric monoidal category: two interfaces $x$ and $x'$ sitting side by side can be reinterpreted as a single interface $x \otimes x'$.  Thus, an open system $f$ with any number of left interfaces $x_1, \dots, x_m$ and any number of right interfaces $y_1, \dots, y_n$ can be interpreted as a morphism
\[         f \maps x_1 \otimes \cdots \otimes x_m \to y_1 \otimes \cdots \otimes y_n  .\]
We can also think of this as having a single left interface $x = x_1 \otimes \cdots \otimes x_m$ and a single right interface $y =  y_1 \otimes \cdots \otimes y_n$.

Second, why should there be a distinction between left and right interfaces, and why can we only compose systems by attaching the left interface of one to the right interface of the other?     Again, we do not need to take this approach.   It is natural for systems with a well-defined `input' and `output'---but these are more common among engineered systems than naturally occuring ones.   The language of `input' and `output' suggests that when we compose open systems $f \maps x \to y$ and $g \maps y \to z$ to create an open system $g \circ f \maps x \to z$, the behavior of $f$ affects $g$ but not conversely.   This is a very useful simplifying assumption.   But in the natural world, this assumption is at best an approximation that holds under very limited circumstances.  
A widely applicable principle in physics, called the `principle of reciprocity', says that if one system can affect another, then the second system can also affect the first \cite{DF}.

In any event, we want a formalism that lets us study open systems that have a rigid distinction between input and output, but also those where the distinction between left and right interface is purely conventional.    We can study \emph{both} using symmetric monoidal categories.   When desired, we can eliminate the distinction between input and output in several ways.  One is to work with a symmetric monoidal dagger-category: here any morphism $f \maps x \to y$ gives a morphism $f^\dagger \maps y \to x$.   Another is to work with a compact closed category: here every object $x$ has a dual $x^\ast$, and any morphism $f \maps x \to y$ gives a morphism $f^\ast \maps y^\ast \to x^\ast$, as well as a morphism $I \to x^\ast \otimes y$ and a morphism $x \otimes y^\ast \to I$.   Another way is to do both, and work with a `dagger compact category'.   All these options, and others, have been extensively studied \cite{BS,HV,Selinger2010,Selinger2011}.

There are further features of the variable-sharing approach to open systems that call for a more flexible approach based on operads, particularly the operad of undirected wiring diagrams \cite{Spivak2017,Yau2018}.    Operads easily allow for more composition patterns more general than the `end-to-end' composition of morphisms in categories.  But an even more urgent issue appears as soon as we consider an example: say, electrical circuits.

For simplicity, consider electrical circuits made using only resistors.  We can describe such a circuit as a finite graph where each edge is labeled by a positive real number called its `resistance':
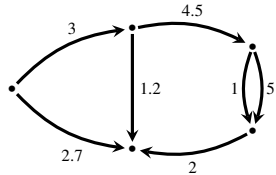
                \[
\scalebox{0.8}{
\begin{tikzpicture}
	\begin{pgfonlayer}{nodelayer}
			\node [contact] (n1) at (-2,0) {$\bullet$};
		\node [style = none] at (-2.1,0.3) {$ $};
		\node [contact] (n2) at (0,1) {$\bullet$};
		\node [style = none] at (0,1.3) {$ $};
		\node [contact] (n3) at (0,-1) {$\bullet$};
		\node [style = none] at (0.05,-1.3) {$ $};
		\node [contact] (n4) at (2,0.7) {$\bullet$};
		\node [style = none] at (2.1,1) {$ $};
  		\node [contact] (n5) at (2,-0.7) {$\bullet$};
		\node [style = none] at (2.15,-1) {$ $};

		\node [style = none] at (-1,1) {$3.9$};
		\node [style = none] at (-1,-1.1) {$2.7$};
		\node [style = none] at (1,1.3) {$4.5$};
		\node [style = none] at (1,-1.3) {$2.3$};
	        \node [style = none] at (0.3,0) {$1.3$};
                \node [style = none] at (1.5,0) {$1.2$};
                \node [style = none] at (2.5,0) {$5.1$};

	\end{pgfonlayer}
	\begin{pgfonlayer}{edgelayer}
		\draw [style=inarrow, bend left=20, looseness=1.00] (n1) to (n2);
		\draw [style=inarrow, bend right=20, looseness=1.00] (n1) to (n3);
		\draw [style=inarrow, bend left=20, looseness=1.00] (n2) to (n4);
  		\draw [style=inarrow, bend left=20, looseness=1.00] (n5) to (n3);
		\draw [style=inarrow] (n2) to (n3);
  		\draw [style=inarrow, bend left=20, looseness=1.00] (n4) to (n5);
  		\draw [style=inarrow, bend right=20, looseness=1.00] (n4) to (n5);
	\end{pgfonlayer}
\end{tikzpicture}
}
\]
(We apologize to electrical engineers for not decorating the edges with symbols indicating resistors.)   We can treat such a circuit as an \emph{open} system by choosing some nodes to be inputs and some to be outputs.   Even better, for mathematicians at least, is to choose two finite sets $X$ and $Y$  equipped with maps to the circuit's set of nodes:
                \[
\scalebox{0.8}
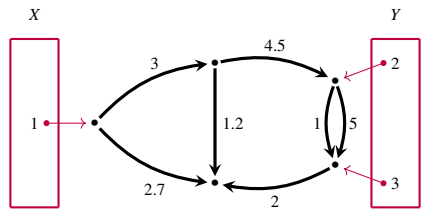
{
\begin{tikzpicture}
	\begin{pgfonlayer}{nodelayer}
		\node [contact] (n1) at (-2,0) {$\bullet$};
		\node [style = none] at (-2.1,0.3) {$ $};
		\node [contact] (n2) at (0,1) {$\bullet$};
		\node [style = none] at (0.1,1.3) {$ $};
		\node [contact] (n3) at (0,-1) {$\bullet$};
		\node [style = none] at (0.05,-1.3) {$ $};
		\node [contact] (n4) at (2,0.7) {$\bullet$};
		\node [style = none] at (2.1,1) {$ $};
  		\node [contact] (n5) at (2,-0.7) {$\bullet$};
		\node [style = none] at (2.15,-1) {$ $};

		\node [style = none] at (-1,1) {$3.9$};
		\node [style = none] at (-1,-1.1) {$2.7$};
		\node [style = none] at (1,1.3) {$4.5$};
		\node [style = none] at (1,-1.3) {$2.3$};
	        \node [style = none] at (0.1,0) {$1.3$};
                \node [style = none] at (1.5,0) {$1.2$};
                \node [style = none] at (2.45,0) {$5.1$};
                
		\node [style=none] (1) at (-3,0) {1};
		\node [style=none] (2) at (3.2,1) {2};
  	        \node [style=none] (3) at (3.2,-1) {3};

		\node [style=none] (ATL) at (-3.4,1.4) {};
		\node [style=none] (ATR) at (-2.6,1.4) {};
		\node [style=none] (ABR) at (-2.6,-1.4) {};
		\node [style=none] (ABL) at (-3.4,-1.4) {};

		\node [style=none] (X) at (-3,1.8) {$X$};
		\node [style=inputdot] (inI) at (-2.8,0) {};

		\node [style=none] (Y) at (3.2,1.8) {$Y$};
	         \node [style=inputdot] (outI) at (3,1) {};
                \node [style=inputdot] (outI') at (3,-1) {};

		\node [style=none] (MTL) at (2.8,1.4) {};
		\node [style=none] (MTR) at (3.6,1.4) {};
		\node [style=none] (MBR) at (3.6,-1.4) {};
		\node [style=none] (MBL) at (2.8,-1.4) {};

	\end{pgfonlayer}
	\begin{pgfonlayer}{edgelayer}
		\draw [style=inarrow, bend left=20, looseness=1.00] (n1) to (n2);
		\draw [style=inarrow, bend right=20, looseness=1.00] (n1) to (n3);
		\draw [style=inarrow, bend left=20, looseness=1.00] (n2) to (n4);
  		\draw [style=inarrow, bend left=20, looseness=1.00] (n5) to (n3);
		\draw [style=inarrow] (n2) to (n3);
  		\draw [style=inarrow, bend left=20, looseness=1.00] (n4) to (n5);
  		\draw [style=inarrow, bend right=20, looseness=1.00] (n4) to (n5);
  		
  		\draw [style=simple] (ATL.center) to (ATR.center);
		\draw [style=simple] (ATR.center) to (ABR.center);
		\draw [style=simple] (ABR.center) to (ABL.center);
		\draw [style=simple] (ABL.center) to (ATL.center);

		\draw [style=simple] (MTL.center) to (MTR.center);
		\draw [style=simple] (MTR.center) to (MBR.center);
		\draw [style=simple] (MBR.center) to (MBL.center);
		\draw [style=simple] (MBL.center) to (MTL.center);

		\draw [style=inputarrow] (inI) to (n1);
		\draw [style=inputarrow] (outI) to (n4);
  		\draw [style=inputarrow] (outI') to (n5);
	\end{pgfonlayer}
\end{tikzpicture}
}
\]
We can think of this as a morphism from $X$ to $Y$.   We can compose this with a morphism from $Y$ to $Z$:
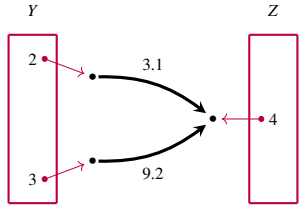
\[
\scalebox{0.8}{
\begin{tikzpicture}
	\begin{pgfonlayer}{nodelayer}
 
		\node [contact] (n6) at (0,0.7) {$\bullet$};
		\node [style = none] at (0,1.3) {$ $};
		\node [contact] (n7) at (0,-0.7) {$\bullet$};
		\node [style = none] at (0.05,-1.3) {$ $};
		\node [contact] (n8) at (2,0) {$\bullet$};
		\node [style = none] at (2.1,0.4) {$ $};

		\node [style = none] at (1,0.9) {$3.1$};
		\node [style = none] at (1,-0.9) {$9.2$};

		\node [style=none] (2) at (-1,1) {2};
  	         \node [style=none] (3) at (-1,-1) {3};
           	\node [style=none] (4) at (3,0) {4};

		\node [style=none] (ATL) at (-1.4,1.4) {};
		\node [style=none] (ATR) at (-0.6,1.4) {};
		\node [style=none] (ABR) at (-0.6,-1.4) {};
		\node [style=none] (ABL) at (-1.4,-1.4) {};

		\node [style=none] (Y) at (-1,1.8) {$Y$};
  	        \node [style=inputdot] (inI) at (-0.8,1) {};
                \node [style=inputdot] (inI') at (-0.8,-1) {};

		\node [style=none] (Z) at (3,1.8) {$Z$};
  		\node [style=inputdot] (outI) at (2.8,0) {};

		\node [style=none] (MTL) at (2.6,1.4) {};
		\node [style=none] (MTR) at (3.4,1.4) {};
		\node [style=none] (MBR) at (3.4,-1.4) {};
		\node [style=none] (MBL) at (2.6,-1.4) {};

	\end{pgfonlayer}
	\begin{pgfonlayer}{edgelayer}

		\draw [style=inarrow, bend left=20, looseness=1.00] (n6) to (n8);
  		\draw [style=inarrow, bend right=20, looseness=1.00] (n7) to (n8);

		\draw [style=simple] (ATL.center) to (ATR.center);
		\draw [style=simple] (ATR.center) to (ABR.center);
		\draw [style=simple] (ABR.center) to (ABL.center);
		\draw [style=simple] (ABL.center) to (ATL.center);

		\draw [style=simple] (MTL.center) to (MTR.center);
		\draw [style=simple] (MTR.center) to (MBR.center);
		\draw [style=simple] (MBR.center) to (MBL.center);
		\draw [style=simple] (MBL.center) to (MTL.center);

		\draw [style=inputarrow] (inI) to (n6);
		\draw [style=inputarrow] (inI') to (n7);
  		\draw [style=inputarrow] (outI) to (n8);

	\end{pgfonlayer}
\end{tikzpicture}
}
\]
and the result is this:
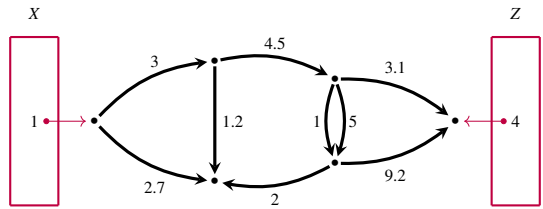
\[
\scalebox{0.8}{
\begin{tikzpicture}
	\begin{pgfonlayer}{nodelayer}
		\node [contact] (n1) at (-2,0) {$\bullet$};
		\node [style = none] at (-2.1,0.3) {$ $};
		\node [contact] (n2) at (0,1) {$\bullet$};
		\node [style = none] at (0,1.3) {$ $};
		\node [contact] (n3) at (0,-1) {$\bullet$};
		\node [style = none] at (0.05,-1.3) {$ $};
		\node [contact] (n4) at (2,0.7) {$\bullet$};
		\node [style = none] at (2.1,1) {$ $};
  		\node [contact] (n5) at (2,-0.7) {$\bullet$};
		\node [style = none] at (2.15,-1) {$ $};;
		\node [contact] (n8) at (4,0) {$\bullet$};
		\node [style = none] at (4.1,0.4) {$$};

		\node [style = none] at (-1,1) {$3.9$};
		\node [style = none] at (-1,-1.1) {$2.7$};
		\node [style = none] at (1,1.3) {$4.5$};
		\node [style = none] at (1,-1.3) {$2$};
	         \node [style = none] at (0.3,0) {$1.3$};
                 \node [style = none] at (1.5,0) {$1.2$};
                 \node [style = none] at (2.5,0) {$5.1$};
                 	\node [style = none] at (3,0.9) {$3.1$};
		\node [style = none] at (3,-0.9) {$9.2$};

		\node [style=none] (X) at (-3,1.8) {$X$};
	         \node [style=none] (1) at (-3,0) {1};
          	\node [style=inputdot] (inI) at (-2.8,0) {};
      
     	        \node [style=none] (ATL) at (-3.4,1.4) {};
		\node [style=none] (ATR) at (-2.6,1.4) {};
		\node [style=none] (ABR) at (-2.6,-1.4) {};
		\node [style=none] (ABL) at (-3.4,-1.4) {};

		\node [style=none] (Z) at (5,1.8) {$Z$};
  		\node [style=inputdot] (outI) at (4.8,0) {};
                 \node [style=none] (4) at (5,0) {4};
     
  		\node [style=none] (MTL) at (4.6,1.4) {};
		\node [style=none] (MTR) at (5.4,1.4) {};
		\node [style=none] (MBR) at (5.4,-1.4) {};
		\node [style=none] (MBL) at (4.6,-1.4) {};

	\end{pgfonlayer}
	\begin{pgfonlayer}{edgelayer}
		\draw [style=inarrow, bend left=20, looseness=1.00] (n1) to (n2);
		\draw [style=inarrow, bend right=20, looseness=1.00] (n1) to (n3);
		\draw [style=inarrow, bend left=20, looseness=1.00] (n2) to (n4);
  		\draw [style=inarrow, bend left=20, looseness=1.00] (n5) to (n3);
		\draw [style=inarrow] (n2) to (n3);
  		 \draw [style=inarrow, bend left=20, looseness=1.00] (n4) to (n5);
  		\draw [style=inarrow, bend right=20, looseness=1.00] (n4) to (n5);
          	\draw [style=inarrow, bend left=20, looseness=1.00] (n4) to (n8);
      	        \draw [style=inarrow, bend right=20, looseness=1.00] (n5) to (n8);

		\draw [style=simple] (ATL.center) to (ATR.center);
		\draw [style=simple] (ATR.center) to (ABR.center);
		\draw [style=simple] (ABR.center) to (ABL.center);
		\draw [style=simple] (ABL.center) to (ATL.center);

		\draw [style=inputarrow] (inI) to (n1);
  		\draw [style=inputarrow] (outI) to (n8);
  
		\draw [style=simple] (MTL.center) to (MTR.center);
		\draw [style=simple] (MTR.center) to (MBR.center);
		\draw [style=simple] (MBR.center) to (MBL.center);
		\draw [style=simple] (MBL.center) to (MTL.center);

	\end{pgfonlayer}
\end{tikzpicture}
}
\]
Nothing profound is going on here: we are simply gluing together two open electrical circuits to get a new one.   But there is something worth noting.   We form the underlying graph of the new circuit by taking a pushout in the category of graphs.   But since pushout is defined only up to canonical isomorphism, this means composition of circuits will not be strictly associative, only associative up to isomorphism.   

There are at least three responses to this problem.  One is to carefully choose the pushouts to make composition  associative.   In the electrical circuits example, it is sufficient to make the disjoint union of finite sets into an associative operation---for example, to replace the category of finite sets with coproduct as its monoidal structure with an equivalent strict monoidal category \cite[\S XI.3]{MacLane}.   But the philosophy of modern category theory is generally to avoid such strictification.

Another approach is to take isomorphism classes of open circuits as morphisms.  This gives a category, but unfortunately one cannot point to a specific node or edge in an \textsl{isomorphism class} of graphs: for that, one needs an actual graph.  

A third approach is to work with arbitrary pushouts, to accept that this makes composition of open circuits associative only up to isomorphism, and instead of seeking a category of open circuits, work with a bicategory---or better yet, a double category!  And this turns out to have great advantages.

In fact Courser \cite{Courser} had been trying to construct symmetric monoidal bicategories of open systems when he found that the coherence laws for symmetric monoidal bicategories \cite{GPS,McCrudden,Stay} are complicated and tiring to check directly.   He realized that the easiest approach was to use a result of Shulman \cite{Shulman2010}: any fibrant symmetric monoidal double category gives a symmetric monoidal bicategory.    Later it became clear that double categories are closer to what we actually want, since in addition to having open systems as loose morphisms between interfaces, which compose in a weakly associative manner, they have tight morphisms as \emph{maps} between interfaces, which compose in a strictly associative manner.  Most importantly, 2-cells in a double category allow us to describe maps between open systems.

In this approach, we describe open systems of a given kind using a symmetric monoidal double category where:
\begin{itemize}
\item An object is an `interface'.
\item A loose morphism $F \maps x \slashedrightarrow y$ is an `open system' with `input interface' $x$ and `output interface' $y$.
\item Composing loose morphisms $F \maps x \slashedrightarrow y$ and $G \maps y \slashedrightarrow z$ is attaching the output interface of $F$ to the input interface of $G$, obtaining an open system $G \circ F \maps x \to z$.
\item A tight morphism $f \maps x \to y$ is a `map between interfaces'.
\item Composing tight morphisms is composing maps between interfaces.
\item A 2-cell
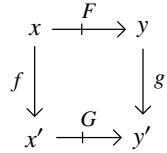
\[
\scalebox{1.2}{
\begin{tikzpicture}[scale=1.2]
\node (A) at (0,0) {$x$};
\node (B) at (1,0) {$y$};
\node (A') at (0,-1) {$x'$};
\node (B') at (1,-1) {$y'$};
\path[->,font=\scriptsize,>=angle 90]
(A) edge [tick] node[above]{$F$} (B)
(A') edge [tick] node[below]{$G$} (B')
(A) edge node [left]{$f$} (A')
(B) edge node [right]{$g$} (B');
\end{tikzpicture}
}
\]
is a `map between open systems' from $F$ to $G$.
\item  Tensoring two loose morphisms $F \maps x \slashedrightarrow y$ and $F' \maps x' \slashedrightarrow y'$ is `setting two open systems side by side', giving an open system $F \otimes F' \maps x \otimes x' \slashedrightarrow y \otimes y'$.
\item Tensoring two tight morphisms is `setting two maps of interfaces side by side'.
\item  Tensoring two 2-cells is `setting two maps of open systems side by side'.
\end{itemize}

Next we describe two methods of defining such double categories: structured and decorated cospans.   It will often be convenient to state theorems in two parts: first one for double categories, then one for symmetric monoidal double categories.  In practice, the symmetric monoidal version is always the more useful one---but it relies on the simpler version that ignores the symmetric monoidal structure.

\section{Structured and decorated cospans}
\label{sec:structured_and_decorated}

Experience has shown that many open systems are nicely modeled using cospans \cite{CourserThesis, FongThesis, PollardThesis}. A cospan in a category $\A$ is a diagram of this form:
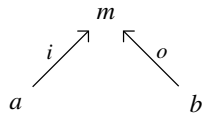
\[
\scalebox{1.2}{
\begin{tikzpicture}[scale=1]
\node (A) at (0,0) {$a$};
\node (B) at (1,1) {$m$};
\node (C) at (2,0) {$b$};
\path[->,font=\scriptsize,>=angle 90]
(A) edge node[above, pos=0.3]{$i$} (B)
(C) edge node[above, pos=0.3]{$o$} (B);
\end{tikzpicture}
}
\]
We call $m$ the \define{apex}, $a$ and $b$ the \define{feet}, and $i$ and $o$ the \define{legs} of the cospan.   The apex is the system itself.  The feet are `interfaces'  through which the system can interact with the outside world.  The legs say how the interfaces are included in the system.    If the category $\A$ has pushouts, we can compose cospans in $\A$: this describes the operation of attaching two open systems together in series by identifying one interface of the first with one of the second:
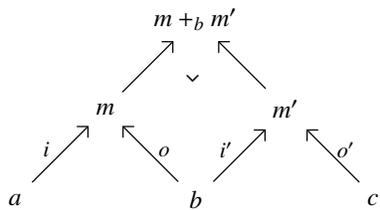
\[
\scalebox{1.2}{
\begin{tikzpicture}
\node (A) at (0,0) {$a$};
\node (B) at (1,1) {$m$};
\node (C) at (2,0) {$b$};
\node (B') at (3,1) {$m'$};
\node (C') at (4,0) {$c$.};
\node (D) at (2,2) {$m+_b m'$};
\node (push) at (2,1.4) {\rotatebox{135}{$\ulcorner$}};
\path[->,font=\scriptsize,>=angle 90]
(A) edge node[above, pos = 0.25]{$i$} (B)
(C) edge node[above, pos = 0.25]{$o$} (B)
(C) edge node[above, pos = 0.25]{$i'$} (B')
(C') edge node[above, pos = 0.25]{$o'$} (B')
(B) edge node[above, pos = 0.25]{} (D)
(B') edge node[above, pos = 0.25]{} (D);
\end{tikzpicture}
}
\]
If $\A$ also has coproducts, we can also `tensor' cospans: this describes setting open systems side by side, in parallel:
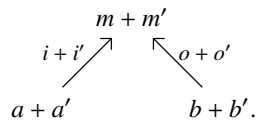
\[
\scalebox{1.2}{
\begin{tikzpicture}[scale=1]
\node (A) at (0,0) {$a+a'$};
\node (B) at (1,1) {$m+m'$};
\node (C) at (2,0) {$b+b'$.};
\path[->,font=\scriptsize,>=angle 90]
(A) edge node[left, pos=0.7]{$i+i'$} (B)
(C) edge node[right, pos=0.7]{$o+o'$} (B);
\end{tikzpicture}
}
\]
However, we often want the system itself to have more structure than its interfaces.   We know of two main ways to do this: structured and decorated cospans.  

\subsection{Structured cospans}

Given a functor $L \maps \A \to \X$, an $L$-\define{structured cospan} is a cospan in $\X$ whose feet come from a specified pair of objects in $a, b \in \A$:
\[
\scalebox{1.2}{
\begin{tikzpicture}[scale=1.0]
\node (A) at (0,0) {$L(a)$};
\node (B) at (1,1) {$x$};
\node (C) at (2,0) {$L(b).$};
\path[->,font=\scriptsize,>=angle 90]
(A) edge node[above,pos=0.3]{$i$} (B)
(C)edge node[above,pos=0.3]{$o$}(B);
\end{tikzpicture}
}
\]
Very often $\A$ is some category of objects with `less structure', like finite sets, while $\X$ is some category of objects with `more structure', like finite graphs, and $L \maps \A \to \X$ is left adjoint to some functor $R \maps \X \to \A$ that forgets the extra structure.   Thus we often think of an $L$-structured cospan as consisting of a `system' $x \in \X$ and two `interfaces' $a, b \in \A$ that have less structure than the system, mapped into the system via $i$ and $o$.  

When $L$ is left adjoint to some functor $R$, as is often the case, an $L$-structured  cospan is equivalent to a diagram
\[
\scalebox{1.2}{
\begin{tikzpicture}[scale=1.0]
\node (A) at (0,0) {$a$};
\node (B) at (1,1) {$R(x)$};
\node (C) at (2,0) {$b$};
\path[->,font=\scriptsize,>=angle 90]
(A) edge node[above,pos=0.3]{$i$} (B)
(C)edge node[above,pos=0.3]{$o$}(B);
\end{tikzpicture}
}
\]
together with a specific choice of $x \in \X$, and this is often a more intuitive way to think about it.  However, the description using $L$ works better technically, since it lets us compose structured cospans as long as $\X$ has pushouts.

\begin{thm}\label{thm:structured_cospan_1}
Let $L \maps \A \to \X$ be a functor and let $\X$ be a category with pushouts.  Then there is a fibrant double category $_L\lCsp$, the \define{double category of} $L$-\define{structured cospans}, in which
\begin{itemize}
\item an object is an object of $\A$,
\item a tight morphism is a morphism of $\A$,
\item a loose morphism from $a$ to $b$ is an $L$-structured cospan
\[
\scalebox{1.2}{
\begin{tikzpicture}[scale=1.0]
\node (A) at (0,0) {$L(a)$};
\node (B) at (1,1) {$x$};
\node (C) at (2,0) {$L(b).$};
\path[->,font=\scriptsize,>=angle 90]
(A) edge node[above,pos=0.3]{$i$} (B)
(C)edge node[above,pos=0.3]{$o$}(B);
\end{tikzpicture}
}
\]
\item a 2-cell is a \define{map of} $L$-\define{structured cospans}, that is, a commutative diagram in $\X$ of the form
\[
\scalebox{1.2}{
\begin{tikzpicture}[scale=1.2]
\node (A) at (0,0) {$L(a)$};
\node (B) at (1,0) {$x$};
\node (C) at (2,0) {$L(b)$};
\node (A') at (0,-1) {$L(a')$};
\node (B') at (1,-1) {$x'$};
\node (C') at (2,-1) {$L(b')$.};
\path[->,font=\scriptsize,>=angle 90]
(A) edge node[above]{$i$} (B)
(C) edge node[above]{$o$} (B)
(A') edge node[below]{$i'$} (B')
(C') edge node[below]{$o'$} (B')
(A) edge node [left]{$L(f)$} (A')
(B) edge node [left]{$\alpha$} (B')
(C) edge node [right]{$L(g)$} (C');
\end{tikzpicture}
}
\]
\end{itemize}
Composition of tight morphisms is composition in $\A$.
Composition of loose morphisms and 2-cells is done using pushouts in $\X$. 
\end{thm} 

\begin{proof}
See \cite[Thm.\ 2.3]{BC} or \cite[Thm.\ 3.2.1]{CourserThesis}, and \cite[Prop.\ 2.2]{Patterson2023} for the fibrancy.
\end{proof}

In practice we almost always seem to work with symmetric monoidal double categories of structured cospans---and in fact, cocartesian ones.    Dualizing Aleiferi's work on cartesian double categories \cite{AleiferiThesis}, we define a \define{cocartesian} double category $\lD$ to be a cocartesian object in the 2-category of double categories (with double functors as morphisms, as usual meaning pseudo double functors). This implies that the category $\lD_0$ of objects and tight morphisms and the category $\lD_1$ of loose morphisms and 2-cells are categories with finite coproducts, and the source, target, composition, and identity-assigning maps preserve these.

\begin{thm}
\label{thm:structured_cospan_2}
Let $L \maps \A \to \X$ be a functor preserving finite coproducts, where $\A$ has finite coproducts and $\X$ has finite colimits.  Then $_L\lCsp$ naturally has the structure of a cocartesian double category.  It also becomes a symmetric monoidal double category, where the symmetric monoidal structure is defined using finite coproducts in $\A$ and $\X$.
\end{thm}

\begin{proof} 
This is a combination of \cite[Thm.\ 3.2.3]{CourserThesis} , \cite[Thm.\ 3.9]{BC}, \cite[Thm.\ 3.2.1]{BCV},  and \cite[Thm.\ 2.3]{Patterson2023}.   
\end{proof}

It is worth saying a bit about the hypotheses in Theorem \ref{thm:structured_cospan_2}.  They have been pared down to almost the minimum necessary.  In practice $L \maps \A \to \X$ is often a left adjoint functor between categories with finite colimits,  which is more than enough for the theorem to apply.   However, in one application \cite[Thm.\ 7.6]{BCh} it seemed necessary to pare down the hypotheses even more.  To get the conclusions of Theorem \ref{thm:structured_cospan_2}, we do not need $\X$ to have all finite colimits.   It is enough that it have finite coproducts and have pushouts for diagrams of this form:
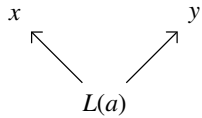
\[   
\scalebox{1.2}{
\begin{tikzpicture}[scale=1.0]
\node (A) at (0,1) {$x$};
\node (B) at (1,0) {$L(a)$};
\node (C) at (2,1) {$y$};
\path[->,font=\scriptsize,>=angle 90]
(B) edge node[above]{$$} (A)
(B) edge node[above]{$$}(C);
\end{tikzpicture}
}
\]
since these are the only pushouts required to compose structured cospans.

Finally, a word about history.  Structured cospans were first discovered in 2007 by Fiadeiro and Schmitt \cite{FiadeiroSchmitt}.  In 2015 Par\'e gave a talk about the dual concept, structured spans, which he called `superspans' \cite{Pare}.   Completely ignorant of these earlier developments, Courser and the author \cite{BC} reinvented structured cospans around 2020, even giving them the same name that Fiaideiro and Schmitt used.    This is an example of how a mathematical concept can be invented repeatedly, but only catch on when its applications become clear.

\subsection{Decorated cospans}

Structured cospans handle a large class of cospans where the apex has more structure than the feet---but apparently not all.   As mentioned, structured cospans assume we fix a category $\A$ of `interfaces' and a category $\X$ of `systems', together with a functor $L \maps \A \to \X$ that gives a standard way to turn an interface into a system.   To get a double category of structured cospans, we need $\X$ to have pushouts.  But we shall see examples where the category of systems does not have pushouts.

We can get around this problem using decorated cospans \cite{BCV,Fong}.    Here we start with a category $\A$ with finite colimits and a pseudofunctor $F \maps \A \to \Cat$.  We again think of objects of $\A$ as `interfaces'.   But now, for each interface $m \in \A$, we have a category $F(m)$ of ways to equip it with extra data making it into a system.     We call a choice of this extra data $d \in F(m)$ a \define{decoration} of $m$.    

In this approach we define an open system to be an $F$-\define{decorated cospan}: a cospan in $\A$ together with a decoration of its apex.  We write an $F$-decorated cospan in this way:
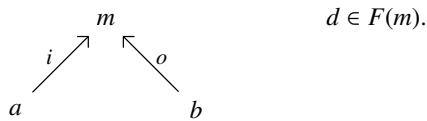
\[
\scalebox{1.2}{
\begin{tikzpicture}[column sep={.4in,between origins}, row sep=.08in]
\node (A) at (0,0) {$a$};
\node (B) at (1,1) {$m$};
\node (C) at (2,0) {$b$};
\node (E) at (4,1) {$d \in F(m)$.};
\path[->,font=\scriptsize,>=angle 90]
(A) edge node[above,pos=0.3]{$i$} (B)
(C) edge node[above,pos=0.3]{$o$} (B);
\end{tikzpicture}
}
\]
To compose decorated cospans, we need to equip $F$ with the structure of a lax monoidal pseudofunctor from $(\A, +)$ to $(\Cat, \times)$.  The key ingredient here beyond $F$ itself is the `laxator', which gives for each pair of objects $m,m' \in \A$ a functor
\[   \phi_{m.m'} \maps F(m) \times F(m') \to F(m+m') . \] 
Given a composable pair of $F$-decorated cospans
\[
\scalebox{1.2}{
\begin{tikzpicture}
\node (A) at (0,0) {$a$};
\node (B) at (1,1) {$m$};
\node (C) at (2,0) {$b$};
\node (D) at (3,1) {$d \in F(m)$};
\node (A') at (5,0) {$b$};
\node (B') at (6,1) {$m'$};
\node (C') at (7,0) {$c$};
\node (D') at (8,1) {$d' \in F(m')$};
\path[->,font=\scriptsize,>=angle 90]
(A) edge node[above, pos = 0.25]{$i$} (B)
(C) edge node[above, pos = 0.25]{$o$} (B)
(A') edge node[above, pos = 0.25]{$i'$} (B')
(C') edge node[above, pos = 0.25]{$o'$} (B');
\end{tikzpicture}
}
\]
we define their composite to be 
\[
\scalebox{1.2}{
\begin{tikzpicture}
\node (A) at (0,0) {$a$};
\node (B) at (1,1) {$m$};
\node (C) at (2,0) {$b$};
\node (B') at (3,1) {$m'$};
\node (C') at (4,0) {$c$.};
\node (D) at (2,2) {$m+_b m'$};
\node (E) at (7,2) { $F(j) \left( \phi_{m,m'} (d,d') \right) \in F(m +_b m') $};
\node (push) at (2,1.4) {\rotatebox{135}{$\ulcorner$}};
\path[->,font=\scriptsize,>=angle 90]
(A) edge node[above, pos = 0.25]{$i$} (B)
(C) edge node[above, pos = 0.25]{$o$} (B)
(C) edge node[above, pos = 0.25]{$i'$} (B')
(C') edge node[above, pos = 0.25]{$o'$} (B')
(B) edge node[above, pos = 0.25]{} (D)
(B') edge node[above, pos = 0.25]{} (D);
\end{tikzpicture}
}
\]
Here the cospans are composed using a pushout in $\A$, while the decoration is defined using
the laxator and $F$ applied to $j \maps m + m' \to m +_b m'$, the canonical map from the coproduct to the pushout.    The choice of laxator gives a fair amount of flexibility in how the new cospan is decorated---but the laxator must obey some laws that guarantee that this prescription gives a double category \cite[App.\ A.1]{BCV}.

\begin{thm}
\label{thm:decorated_cospan_1}
Let $\A$ be a category with finite colimits and let $F \maps (\A,+) \to (\Cat,\times)$ be a lax monoidal pseudofunctor. Then there is a double category $F\lCsp$, the \define{double category of} $F$-\define{decorated cospans}, in which
\begin{itemize}
\item an object is an object of $\A$,
\item a tight morphism is a morphism of $\A$,
\item a loose morphism is an $F$-decorated cospan
\[
\scalebox{1.2}{
\begin{tikzpicture}[column sep={.4in,between origins}, row sep=.08in]
\node (A) at (0,0) {$a$};
\node (B) at (1,1) {$m$};
\node (C) at (2,0) {$b,$};
\node (E) at (4,1) {$d \in F(m)$};
\path[->,font=\scriptsize,>=angle 90]
(A) edge node[above,pos=0.3]{$i$} (B)
(C) edge node[above,pos=0.3]{$o$} (B);
\end{tikzpicture}
}
\]
\item a 2-cell is a \define{map of} $F$\define{-decorated cospans}, that is, 
a commutative diagram in $\A$ of the form
\[
\scalebox{1.2}{
\begin{tikzpicture}[scale=1.2]
\node (A) at (0,0.5) {$a$};
\node (A') at (0,-0.5) {$a'$};
\node (B) at (1,0.5) {$m$};
\node (C) at (2,0.5) {$b$};
\node (C') at (2,-0.5) {$b'$};
\node (D) at (1,-0.5) {$m'$};
\node (E) at (3,0.5) {$d \in F(m)$};
\node (F) at (3,-0.5) {$d' \in F(m')$};
\path[->,font=\scriptsize,>=angle 90]
(A) edge node[above]{$i$} (B)
(C) edge node[above]{$o$} (B)
(A) edge node[left]{$f$} (A')
(C) edge node[right]{$g$} (C')
(A') edge node[below] {$i'$} (D)
(C') edge node[below] {$o'$} (D)
(B) edge node [left] {$h$} (D);
\end{tikzpicture}
}
\]
together with a \define{decoration morphism} $\tau \maps F(h)(d) \to d'$ in $F(m')$,
\item composition of tight morphisms and vertical composition of 2-cell is done using composition in $\A$,
\item composition of loose morphisms is done as described above,
\item horizontal composition of 2-cells is done as described in \cite[Thm.\ 2.1]{BCV}.
\end{itemize}
\end{thm}

\begin{proof}
See \cite[Thm.\ 2.1]{BCV}.
\end{proof}

Equipping $F$ with some more structure, the double category of $F$-decorated cospans becomes symmetric monoidal.   We could also study decorated cospan double categories that are merely monoidal, or braided monoidal, but so far these come up much less often.

\begin{thm}
\label{thm:decorated_cospan_2}
Let $\A$ be a category with finite colimits and let $F \maps (\A,+) \to (\Cat,\times)$ be a symmetric lax monoidal pseudofunctor with laxator $\phi_{a,b} \maps F(a) \times F(b) \to F(a + b)$. Then the double category $F\lCsp$ becomes symmetric monoidal, where the tensor product
\begin{itemize}
\item of two objects $a$ and $b$ is their coproduct $a+b$ in $\A$,
\item of two tight morphisms $f \maps a_1\to a_2$ and $f' \maps a_1' \to a_2'$ is $f+f' \maps a_1+a_1' \to a_2+a_2'$ in $\A$,
\item of two loose morphisms $(a \to m \leftarrow b , d \in F(m))$ and $(a'\to m' \leftarrow b', d'\in F(n))$ is:
\[
\begin{tikzpicture}
\scalebox{1.2}{
\node (A) at (0,0) {$a+a'$};
\node (B) at (1,1) {$m+m'$};
\node (C) at (2,0) {$b+b'$};
\node (D) at (4,1) {$\phi_{m,m'}(d,d') \in F(m + m')$};
\path[->,font=\scriptsize,>=angle 90]
(A) edge node[above,pos=0.3]{$i$} (B)
(C)edge node[above,pos=0.3]{$o$}(B);
}
\end{tikzpicture}
\]
\item of two 2-cells $\alpha$ and $\beta$ is:
\[
\begin{tikzpicture}
\scalebox{1.2}{
\node (A) at (0,0.5) {$a_1$};
\node (A') at (0,-0.5) {$a_2$};
\node (B) at (1,0.5) {$m_1$};
\node (C) at (2,0.5) {$b_1$};
\node (C') at (2,-0.5) {$b_2$};
\node (D) at (1,-0.5) {$m_2$};
%\node (E) at (2.7,0.5) {$\scriptstyle x \in F(m)$};
\node () at (2.7,0) {$\otimes$};
\node () at (6.2,0) {$=$};
%\node (F) at (2.7,-0.5) {$\scriptstyle x' \in F(m')$};
\node (G) at (3.5,0.5) {$a'_1$};
\node (H) at (4.5,0.5) {$m'_1$};
\node (I) at (5.5,0.5) {$b'_1$};
\node (G') at (3.5,-0.5) {$a_2'$};
\node (H') at (4.5,-0.5) {$m'_2$};
\node (I') at (5.5,-0.5) {$b'_2$};
%\node (J) at (6.7,0.5) {$\scriptstyle y \in F(n)$};
%\node (K) at (6.7,-0.5) {$\scriptstyle y' \in F(n')$};
\node (L) at (1,-1.2) {$\scriptstyle\tau_\alpha \maps F(h)(d_1) \to d_2\textrm{ in } F(m_2)$};
\node (M) at (4.5,-1.2) {$\scriptstyle\tau_\beta \maps F(h')(d'_1) \to d'_2\textrm{ in } F(m'_2)$};
\path[->,font=\scriptsize,>=angle 90]
(A) edge node[above]{$i_1$} (B)
(C) edge node[above]{$o_1$} (B)
(A) edge node[left]{$f$} (A')
(C) edge node[right]{$g$} (C')
(A') edge node[below] {$i_2$} (D)
(C') edge node[below] {$o_2$} (D)
(B) edge node [left] {$h$} (D)
(G) edge node [above] {$i'_2$} (H)
(G) edge node [left] {$f'$} (G')
(H) edge node [left] {$h'$} (H')
(G') edge node [below] {$i_2'$} (H')
(I) edge node [below] {$o_2$} (H)
(I) edge node [right] {$g'$} (I')
(I') edge node [below] {$o_2'$} (H');
}
\end{tikzpicture}
\]
\vskip 1em
\[
\begin{tikzpicture}
\scalebox{1.2}{
\node (A) at (0,0.5) {$a_1 + a'_1$};
\node (A') at (0,-0.5) {$a_2 + a'_2$};
\node (B) at (2,0.5) {$m_1 + m'_1$};
\node (C) at (4,0.5) {$b_1 + b'_1$};
\node (C') at (4,-0.5) {$b_2 + b'_2$};
\node (D) at (2,-0.5) {$m_2 + m'_2$};
%\node (E) at (5.2,0.5) {$\scriptstyle \phi_{m,n}(x,y)$};
%\node (F) at (5.2,-0.5) {$\scriptstyle \phi_{m',n'}(x',y')$};
\node () at (2,-1.3) {$\scriptstyle\tau_{\alpha\otimes\beta}\maps F(h+h')(\phi_{m_1,m'_1}(d_1,d'_1))\to \phi_{m_2,m'_2}(d_2,d'_2)\textrm{ in }F(d_2,d'_2)$};
\path[->,font=\scriptsize,>=angle 90]
(A) edge node[above]{$i_1+i'_2$} (B)
(C) edge node[above]{$o_1+o'_1$} (B)
(A) edge node[left]{$f+f'$} (A')
(C) edge node[right]{$g+g'$} (C')
(A') edge node [below]{$i_2+i_2'$} (D)
(C') edge node [below]{$o_2+o_2'$} (D)
(B) edge node [left] {$h+h'$} (D);
}
\end{tikzpicture}
\]
with decoration morphism $\tau_{\alpha\otimes\beta}$ constructed functorially from $\tau_\alpha$ and $\tau_\beta$.
\end{itemize}
\end{thm}

\begin{proof}
See \cite[Thm.\ 2.2]{BCV}, which also includes the details of how we construct the decoration morphism $\tau_{\alpha \otimes \beta}$.
\end{proof}

\subsection{Structured versus decorated cospans}
\label{subsec:structured_versus_decorated_cospans}

When should we use structured cospans, and when should we use decorated cospans?   Structured cospans are usually easier to work with, because they are defined using 1-categorical data: a functor $L \maps \A \to \X$ with some properties.  Decorated cospans are generally defined using more complicated 2-categorical data: a lax monoidal pseudofunctor $F \maps \A \to \Cat$.   Furthermore, decorated cospans can often be reformulated as structured cospans.   These observations favor structured cospans---and thus, these are more widely used in software \cite{Catlab,HP}.

There is an important exception.   When the categories $F(a)$ for $a \in \A$ are all discrete, we can think of them as sets and treat $F$ as a lax monoidal functor $F \maps \A \to \Set$.  Then 2-categorical concepts are not needed, and the theory of decorated cospans simplifies dramatically.  This in fact was Fong's original approach to decorated cospans \cite{Fong,FongThesis}.   Further, as we shall see, some examples of this sort seem to give decorated cospans that cannot be reformulated as structured cospans.

When are all the categories $F(a)$ discrete?    This happens precisely when the opfibration $U \maps \inta F \to \A$ is discrete, where $\inta F$ is defined using the Grothendieck construction.    An object of $\inta F$ is a pair $(a,d)$ consisting of an object $a \in \A$ and a decoration $d \in F(a)$.    The opfibration $U \maps \inta F \to \A$ forgets the decoration.   When $U$ is a discrete opfibration, it means that for any $d \in F(a)$ there is a unique way to equip any morphism $h \maps a \to a'$ in $\A$ with a decoration morphism $F(h)(d) \to d'$.    This is an unusual situation.   But we shall see an important example in Section \ref{sec:dynamical}: dynamical systems.
 
Since structured cospans are usually easier to work with than decorated cospans, another question naturally arises: when are we \emph{forced} to use decorated cospans?   The dynamical systems just mentioned appear to be an example. However, there is not yet an airtight proof.  

We do know a general result on when decorated cospan double categories \emph{are} equivalent to structured cospan double categories \cite{BCV}.  Begin with the data needed to construct a symmetric monoidal decorated cospan double category, as in  Theorem \ref{thm:decorated_cospan_2}.  That is, suppose $\A$ has finite colimits and $ F \maps (\A , +) \to (\Cat, \times)$ is a symmetric lax monoidal pseudofunctor.   It turns out that $F$ can naturally be promoted to a pseudofunctor $F \maps \A \to \SMC$, where $\SMC$ is the 2-category with 
\begin{itemize}
\item symmetric monoidal categories as objects,
\item strong symmetric monoidal functors as morphisms,
\item monoidal natural transformation as 2-morphisms.
\end{itemize}
Let $\Rex$ be the 2-category of with
\begin{itemize}
\item categories with chosen finite colimits as objects,
\item functors preserving finite colimits as morphisms,
\item natural transformations as 2-morphisms.
\end{itemize}
There is an evident pseudofunctor $\Rex \to \SMC$.  Now suppose that $F \maps \A \to \SMC$ factors, as a pseudofunctor, through $\Rex$: this is the key hypothesis.   Then one can show that the opfibration $U \maps \inta F \to \A$ is a right adjoint.   From its left adjoint $L \maps \A \to \inta F$ we can construct a structured cospan double category ${}_L \lCsp(\X)$ by taking
\[         \X = \inta F  .\]
And in fact, under the hypotheses given, the structured cospan double category ${}_L \lCsp(\X)$ is not only equivalent but isomorphic to the decorated cospan double category $F \lCsp$.

In short:

\begin{thm}
\label{thm:equivalence_of_structured_and_decorated}
Let $\A$ be a category with finite colimits and $F \maps (\A,+) \to (\Cat,\times)$ a symmetric lax monoidal pseudofunctor.  Suppose the resulting pseudofunctor $F \maps \A \to \SMC$ factors through $\Rex$.  Then $U \maps \inta F \to \A$ has a left adjoint $L$. Furthermore, the structured cospan double category ${}_L \lCsp$ is isomorphic, as a symmetric monoidal double category, to the decorated cospan double category $F \lCsp$.
\end{thm}

\begin{proof}
This is \cite[Thm.\ 4.1]{BCV}.
\end{proof}

This lets us see certain structured cospans as decorated cospans.  The converse question is also interesting: when is a structured cospan double category equivalent to a decorated cospan double category?  We believe this is true under certain conditions that let us pass from the functor $L \maps \A \to \X$ to an appropriate pseudofunctor $F \maps \A \to \Cat$.   For details, see \cite[Sec.\ 7]{BCV}.   However, the program outlined there has not yet been completed.  Thus, we pose this challenge:

\begin{problem}
Find necessary and sufficient conditions for a structured cospan category to be equivalent to some decorated cospan double category, or vice versa. 
\end{problem}

Next we turn from these rather abstruse issues to something simpler: an example of how we \emph{use} structured cospans.

\section{Petri nets}
\label{sec:Petri}

Here is a Petri net:
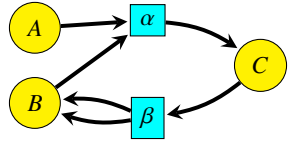
\[
\begin{tikzpicture}
	\begin{pgfonlayer}{nodelayer}
\node [style=species] (C) at (-1, 0) {$C$};
\node [style=species] (B) at (-4, -0.5) {$B$};
		\node [style=species] (A) at (-4, 0.5) {$A$};
		\node [style=transition] (tau1) at (-2.5, 0.6) {$\alpha$};
             \node [style=transition] (tau2) at (-2.5,-0.7) {$\beta$};
	\end{pgfonlayer}
	\begin{pgfonlayer}{edgelayer}
		\draw [style=inarrow] (A) to (tau1);
		\draw [style=inarrow] (B) to (tau1);
	%	\draw [style=inarrow, bend right=15, looseness=1.00] (tau1) to (C);
		\draw [style=inarrow, bend left=15, looseness=1.00] (tau1) to (C);
	       \draw [style=inarrow, bend left=15, looseness=1.00] (C) to (tau2);
             \draw [style=inarrow, bend left=15, looseness=1.00] (tau2) to (B); 
            \draw [style=inarrow, bend right =15, looseness=1.00] (tau2) to (B); 
	\end{pgfonlayer}
\end{tikzpicture}
\]
We can use Petri nets to describe processes where collections of things move between the yellow circles, called `places', by going through the aqua boxes, called `transitions'.   Petri nets are widely used as models of systems in engineering and computer science \cite{GiraultValk, Peterson}.   One of the simplest examples of the resource sharing paradigm is the theory of open Petri nets \cite{BM}, which lets us build Petri nets out of smaller pieces.    We can construct a double category with open Petri nets as loose morphisms using the theory of structured cospans.   There are many possible semantics for open Petri nets, but we shall only describe one, where `tokens' move from place to place via transitions.  This gives an excuse for studying maps between structured cospan double categories.   We also take a look at how to use the 2-cells in double categories of open systems.

While various subtly different definitions are useful \cite{BGMS}, here we take a  \define{Petri net} to be a pair of finite sets $S$ and $T$ and functions $s,t \maps T \to \N[S]$.  Here $S$ is the set of \define{places}, $T$ is the set of \define{transitions}, and $\N[S]$ is the underlying set of the free commutative monoid on $S$.   Thus, elements of $\N[S]$ are finite formal sums of places, and each transition goes from one such formal sum to another.   For example, the source of the transition $\alpha$ above is $\text{A}+\text{B}$, because this transition has one arrow coming into it from $A$ and one from $B$.  The target of $\alpha$ is is $\text{C}$, since it has one arrow going out to $C$.   We can summarize all this information by writing 
\[    \text{A} + \text{B} \xrightarrow{\alpha} \text{C}. \] 
Similarly, we can write the other transition as
\[          \text{C}  \xrightarrow{\beta} 2 \text{B} .\]
Beware: $\alpha$ and $\beta$ are not morphisms in a category!  They are just transitions in a Petri net.  But next we shall see how a Petri net can generate a symmetric monoidal category, with the transitions giving morphisms.

To do this, we define a \define{marking} of a Petri net $s,t \maps T \to \N[S]$ to be an element of $\N[S]$.  Thus, it assigns to each place a natural number.  We think of this as specifying a number of \define{tokens} in each place.   We commonly draw these tokens as dots in the yellow circles that represent places.   For example, given this Petri net:
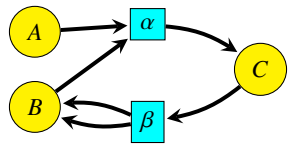
\[
\begin{tikzpicture}
	\begin{pgfonlayer}{nodelayer}

\node [style=species] (C) at (-1, 0) {$C$};
\node [style=species] (B) at (-4, -0.5) {$B$};
\node [style=species] (A) at (-4, 0.5) {$A$};
	    \node [style=transition] (tau1) at (-2.5, 0.6) {$\alpha$};
             \node [style=transition] (tau2) at (-2.5,-0.7) {$\beta$};
	\end{pgfonlayer}
	\begin{pgfonlayer}{edgelayer}
		\draw [style=inarrow] (A) to (tau1);
		\draw [style=inarrow] (B) to (tau1);
	%	\draw [style=inarrow, bend right=15, looseness=1.00] (tau1) to (C);
		\draw [style=inarrow, bend left=15, looseness=1.00] (tau1) to (C);
	       \draw [style=inarrow, bend left=15, looseness=1.00] (C) to (tau2);
             \draw [style=inarrow, bend left=15, looseness=1.00] (tau2) to (B); 
            \draw [style=inarrow, bend right =15, looseness=1.00] (tau2) to (B); 
	\end{pgfonlayer}
\end{tikzpicture}
\]
the marking $A + 2B$ would be written as follows:
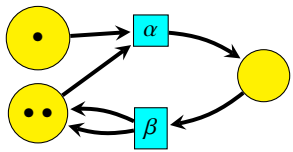
\[
\begin{tikzpicture}
	\begin{pgfonlayer}{nodelayer}
\node [style=species] (C) at (-1, 0) {$\phantom{C}$};
\node [style=species] (B) at (-4, -0.5) {$\bullet\phantom{.} \bullet$};
		\node [style=species] (A) at (-4, 0.5) {$\phantom{.}\bullet\phantom{.}$};
		\node [style=transition] (tau1) at (-2.5, 0.6) {$\alpha$};
             \node [style=transition] (tau2) at (-2.5,-0.7) {$\beta$};
	\end{pgfonlayer}
	\begin{pgfonlayer}{edgelayer}
		\draw [style=inarrow] (A) to (tau1);
		\draw [style=inarrow] (B) to (tau1);
	%	\draw [style=inarrow, bend right=15, looseness=1.00] (tau1) to (C);
		\draw [style=inarrow, bend left=15, looseness=1.00] (tau1) to (C);
	       \draw [style=inarrow, bend left=15, looseness=1.00] (C) to (tau2);
             \draw [style=inarrow, bend left=15, looseness=1.00] (tau2) to (B); 
            \draw [style=inarrow, bend right =15, looseness=1.00] (tau2) to (B); 
	\end{pgfonlayer}
\end{tikzpicture}
\]
We change the marking using transitions.  For example one token of type $A$ and one of type $B$ can enter the transition $\alpha$, and a token of type $C$ can come out, giving this marking:
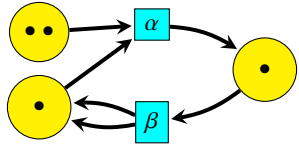
\[
\begin{tikzpicture}
	\begin{pgfonlayer}{nodelayer}
\node [style=species] (C) at (-1, 0) {$\phantom{.}\bullet\phantom{.}$};
\node [style=species] (B) at (-4, -0.5) {$\phantom{.}\bullet\phantom{.}$};
		\node [style=species] (A) at (-4, 0.5) {$\bullet\phantom{.} \bullet$};
		\node [style=transition] (tau1) at (-2.5, 0.6) {$\alpha$};
             \node [style=transition] (tau2) at (-2.5,-0.7) {$\beta$};
	\end{pgfonlayer}
	\begin{pgfonlayer}{edgelayer}
		\draw [style=inarrow] (A) to (tau1);
		\draw [style=inarrow] (B) to (tau1);
	%	\draw [style=inarrow, bend right=15, looseness=1.00] (tau1) to (C);
		\draw [style=inarrow, bend left=15, looseness=1.00] (tau1) to (C);
	       \draw [style=inarrow, bend left=15, looseness=1.00] (C) to (tau2);
             \draw [style=inarrow, bend left=15, looseness=1.00] (tau2) to (B); 
            \draw [style=inarrow, bend right =15, looseness=1.00] (tau2) to (B); 
	\end{pgfonlayer}
\end{tikzpicture}
\]
Then the token of type $C$ can enter the transition $\beta$, and two of type $B$ can come out, giving this marking:
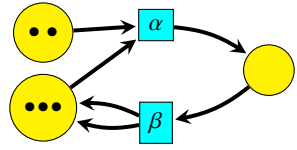
\[
\begin{tikzpicture}
	\begin{pgfonlayer}{nodelayer}
\node [style=species] (C) at (-1, 0) {$\phantom{||}$};
\node [style=species] (B) at (-4, -0.5) {$\bullet\!\bullet\!\bullet$};
		\node [style=species] (A) at (-4, 0.5) {$\bullet\phantom{.} \bullet$};
		\node [style=transition] (tau1) at (-2.5, 0.6) {$\alpha$};
             \node [style=transition] (tau2) at (-2.5,-0.7) {$\beta$};
	\end{pgfonlayer}
	\begin{pgfonlayer}{edgelayer}
		\draw [style=inarrow] (A) to (tau1);
		\draw [style=inarrow] (B) to (tau1);
	%	\draw [style=inarrow, bend right=15, looseness=1.00] (tau1) to (C);
		\draw [style=inarrow, bend left=15, looseness=1.00] (tau1) to (C);
	       \draw [style=inarrow, bend left=15, looseness=1.00] (C) to (tau2);
             \draw [style=inarrow, bend left=15, looseness=1.00] (tau2) to (B); 
            \draw [style=inarrow, bend right =15, looseness=1.00] (tau2) to (B); 
	\end{pgfonlayer}
\end{tikzpicture}
\]
It is natural to think of these ways of changing a marking as morphisms in a category where markings are objects.  This is a symmetric monoidal category, but of a very strict sort: the tensor product commutes on the nose!    The reason is that markings form a commutative monoid, namely the free commutative monoid on the set of places.

A \define{commutative monoidal category} is a symmetric monoidal category $(\C,\otimes)$ such that
the associators $\alpha_{a,b,c} \maps (a \otimes b) \otimes c$, unitors $\lambda_a \maps I \otimes a \to a$, $\rho \maps a \otimes I \to a$, and even the symmetry isomorphisms $\sigma_{a,b} \maps a \otimes b \to b \otimes a$ are all identity morphisms.  Thus, for all objects $a$ and $b$ and morphisms $f$ and $g$ in $\C$ we have
	\[a \otimes b = b \otimes a \text{ and } f \otimes g = g \otimes f.\]
We let $\CMC$ be the category whose objects are commutative monoidal categories and whose morphisms are symmetric strict monoidal functors.   

Another useful way to think of $\CMC$ is as the category of categories internal to the category of commutative monoids.   In this viewpoint, a commutative monoidal category has a commutative monoid of objects and a commutative monoid of morphisms, and all the category operations are monoid homomorphisms.

We can formalize the process of turning a Petri net $P = (s,t \maps T \to \N[S])$ into a commutative monoidal category $FP$ as follows.  We take the commutative monoid of objects $\Ob(FP)$ to be the free commutative monoid $\N[S]$.  Note that element of $\N[S]$ are markings of $P$.    We construct the commutative monoid of morphisms $\Mor(FP)$ as follows.  First we generate morphisms recursively, starting from the transitions of $P$:
\begin{itemize}
\item for every transition $\tau \in T$ we include a morphism $\tau \maps s(\tau) \to t(\tau)$;
\item for any object $a$ we include a morphism $1_a \maps a \to a$;
\item for any morphisms $f \maps a \to b$ and $g \maps a' \to b'$ we include a morphism denoted $f+g \maps a +a' \to b +b'$ to serve as their tensor product;
\item for any morphisms $f \maps a \to b$ and $g \maps b \to c$ we include a morphism $g\circ f \maps a \to c$ to serve as their composite.
\end{itemize}
Then we mod out by an equivalence relation that imposes the laws of a commutative monoidal category, obtaining the commutative monoid $\Mor(FP)$.  The rest of the category structure on $F P$ is straightforward.	

We can extend $F$ to a functor from Petri nets to commutative monoidal categories.   Indeed, there is a category $\Petri$ where Petri nets are objects and a morphism from the Petri net $s, t \maps T \to \N[S]$ to the Petri net $s', t' \maps T' \to \N[S']$ is a pair of functions $f \maps S \to S', g \maps T \to T'$ such that the following diagrams commute:
\[
\scalebox{1.2}{
\begin{tikzpicture}[scale=1.2]
\node (A) at (0,0) {$T$};
\node (B) at (1,0) {$\N[S]$};
\node (C) at (0,-1) {$T'$};
\node (D) at (1,-1) {$\N[S']$};
\path[->,font=\scriptsize,>=angle 90]
(A) edge node[above]{$s$} (B)
(A) edge node[left]{$f$} (C)
(B) edge node[right]{$\N[g]$} (D)
(C) edge node[below]{$s'$} (D);
\end{tikzpicture}
}
\qquad
\scalebox{1.2}{
\begin{tikzpicture}[scale=1.2]
\node (A) at (0,0) {$T$};
\node (B) at (1,0) {$\N[S]$};
\node (C) at (0,-1) {$T'$};
\node (D) at (1,-1) {$\N[S']$,};
\path[->,font=\scriptsize,>=angle 90]
(A) edge node[above]{$t$} (B)
(A) edge node[left]{$f$} (C)
(B) edge node[right]{$\N[g]$} (D)
(C) edge node[below]{$t'$} (D);
\end{tikzpicture}
}
\]
where  $\mathbb{N}[g]$ acts to map any finite sum $\sum_i \sigma_i$ with $\sigma_i \in S$ to 
$\sum_i g(\sigma_i)$.   For example, there is a morphism from this Petri net:
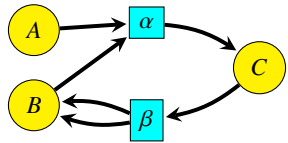
\[
\begin{tikzpicture}
	\begin{pgfonlayer}{nodelayer}
\node [style=species] (C) at (-1, 0) {$C$};
\node [style=species] (B) at (-4, -0.5) {$B$};
		\node [style=species] (A) at (-4, 0.5) {$A$};
		\node [style=transition] (tau1) at (-2.5, 0.6) {$\alpha$};
             \node [style=transition] (tau2) at (-2.5,-0.7) {$\beta$};
	\end{pgfonlayer}
	\begin{pgfonlayer}{edgelayer}
		\draw [style=inarrow] (A) to (tau1);
		\draw [style=inarrow] (B) to (tau1);
	%	\draw [style=inarrow, bend right=15, looseness=1.00] (tau1) to (C);
		\draw [style=inarrow, bend left=15, looseness=1.00] (tau1) to (C);
	       \draw [style=inarrow, bend left=15, looseness=1.00] (C) to (tau2);
             \draw [style=inarrow, bend left=15, looseness=1.00] (tau2) to (B); 
            \draw [style=inarrow, bend right =15, looseness=1.00] (tau2) to (B); 
	\end{pgfonlayer}
\end{tikzpicture}
\]
to this one:
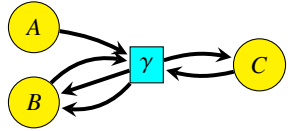
\[
\begin{tikzpicture}
	\begin{pgfonlayer}{nodelayer}
\node [style=species] (C) at (-1, 0) {$C$};
\node [style=species] (B) at (-4, -0.5) {$B$};
		\node [style=species] (A) at (-4, 0.5) {$A$};
		\node [style=transition] (tau) at (-2.5, 0) {$\gamma$};
	\end{pgfonlayer}
	\begin{pgfonlayer}{edgelayer}
		\draw [style=inarrow, bend left = 10, looseness=1.0] (A) to (tau);
		\draw [style=inarrow, bend left = 30, looseness=1.0] (B) to (tau);
	%	\draw [style=inarrow, bend right=15, looseness=1.00] (tau1) to (C);
		\draw [style=inarrow, bend left=15, looseness=1.00] (tau) to (C);
	       \draw [style=inarrow, bend left=15, looseness=1.00] (C) to (tau);
             \draw [style=inarrow, bend left=0, looseness=1.00] (tau) to (B); 
            \draw [style=inarrow, bend left =30, looseness=1.00] (tau) to (B); 
	\end{pgfonlayer}
\end{tikzpicture}
\]
mapping $\alpha$ and $\beta$ to $\gamma$ and acting as the identity on the places $A,B$ and $C$.   Then, there is a functor 
\[    F \maps \Petri \to \CMC \]
sending any Petri net $P$ to the free commutative monoidal category $FP$ that we have already described.  Moreover, this is a left adjoint \cite[Thm.\ 5.1]{Master}. 
	
We should think of $F$ as a `semantics' for Petri nets, 	saying what they can `mean'.   In this semantics
the `meaning' of a Petri net $P$ is the free commutative monoidal category $FP$ describing how tokens 
can move around from place to place.   Thus we call $F$ the \define{token semantics}.
	
\subsection{Open Petri nets}
\label{subsec:open_petri}

We now turn to `open' Petri nets, and construct a double category whose loose morphisms are
open Petri nets.  This lets us build Petri nets out of smaller open parts, which allows us to study Petri nets and their semantics compositionally.   For example, here is a picture of an open Petri net from a finite set $X$ to a finite set $Y$:
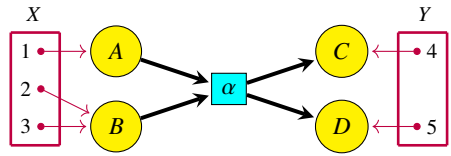
\[
\begin{tikzpicture}
	\begin{pgfonlayer}{nodelayer}
		\node [style=species] (A) at (-4, 0.5) {$A$};
		\node [style=species] (B) at (-4, -0.5) {$B$};
		\node [style=species] (C) at (-1, 0.5) {$C$};
		\node [style=species] (D) at (-1, -0.5) {$D$};
             \node [style=transition] (a) at (-2.5, 0) {$\alpha$}; 
		
		\node [style=empty] (X) at (-5.1, 1) {$X$};
		\node [style=none] (Xtr) at (-4.75, 0.75) {};
		\node [style=none] (Xbr) at (-4.75, -0.75) {};
		\node [style=none] (Xtl) at (-5.4, 0.75) {};
             \node [style=none] (Xbl) at (-5.4, -0.75) {};
	
		\node [style=inputdot] (1) at (-5, 0.5) {};
		\node [style=empty] at (-5.2, 0.5) {$1$};
		\node [style=inputdot] (2) at (-5, 0) {};
		\node [style=empty] at (-5.2, 0) {$2$};
		\node [style=inputdot] (3) at (-5, -0.5) {};
		\node [style=empty] at (-5.2, -0.5) {$3$};

		\node [style=empty] (Y) at (0.1, 1) {$Y$};
		\node [style=none] (Ytr) at (.4, 0.75) {};
		\node [style=none] (Ytl) at (-.25, 0.75) {};
		\node [style=none] (Ybr) at (.4, -0.75) {};
		\node [style=none] (Ybl) at (-.25, -0.75) {};

		\node [style=inputdot] (4) at (0, 0.5) {};
		\node [style=empty] at (0.2, 0.5) {$4$};
		\node [style=inputdot] (5) at (0, -0.5) {};
		\node [style=empty] at (0.2, -0.5) {$5$};		
		
%		\node [style=empty] (Z) at (3, 1) {$Z$};
%		\node [style=none] (Ztr) at (3.25, 0.75) {};
%		\node [style=none] (Ztl) at (2.75, 0.75) {};
%		\node [style=none] (Zbl) at (2.75, -0.75) {};
%		\node [style=none] (Zbr) at (3.25, -0.75) {};
		
	\end{pgfonlayer}
	\begin{pgfonlayer}{edgelayer}
		\draw [style=inarrow] (A) to (a);
		\draw [style=inarrow] (B) to (a);
		\draw [style=inarrow] (a) to (C);
		\draw [style=inarrow] (a) to (D);
		\draw [style=inputarrow] (1) to (A);
		\draw [style=inputarrow] (2) to (B);
		\draw [style=inputarrow] (3) to (B);
		\draw [style=inputarrow] (4) to (C);
		\draw [style=inputarrow] (5) to (D);
		\draw [style=simple] (Xtl.center) to (Xtr.center);
		\draw [style=simple] (Xtr.center) to (Xbr.center);
		\draw [style=simple] (Xbr.center) to (Xbl.center);
		\draw [style=simple] (Xbl.center) to (Xtl.center);
		\draw [style=simple] (Ytl.center) to (Ytr.center);
		\draw [style=simple] (Ytr.center) to (Ybr.center);
		\draw [style=simple] (Ybr.center) to (Ybl.center);
		\draw [style=simple] (Ybl.center) to (Ytl.center);
	\end{pgfonlayer}
\end{tikzpicture}
\]
We shall think of it as a loose morphism in a double category and write it as $P \maps X \slashedrightarrow Y$ for short.   Given another open Petri net $Q \maps Y \slashedrightarrow Z$:
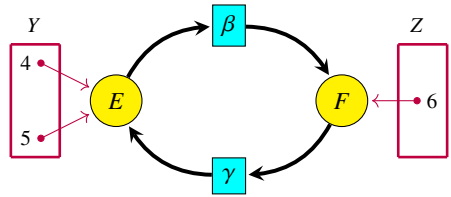
\[
\begin{tikzpicture}
	\begin{pgfonlayer}{nodelayer}

		\node [style = transition] (b) at (2.5, 1) {$\beta$};
		\node [style = transition] (c) at (2.5, -1) {$\gamma$};
		\node [style = species] (E) at (1, 0) {$E$};
		\node [style = species] (F) at (4,0) {$F$};

		\node [style=empty] (Y) at (-0.1, 1) {$Y$};
		\node [style=none] (Ytr) at (.25, 0.75) {};
		\node [style=none] (Ytl) at (-.4, 0.75) {};
		\node [style=none] (Ybr) at (.25, -0.75) {};
		\node [style=none] (Ybl) at (-.4, -0.75) {};

		\node [style=inputdot] (4) at (0, 0.5) {};
		\node [style=empty] at (-0.2, 0.5) {$4$};
		\node [style=inputdot] (5) at (0, -0.5) {};
		\node [style=empty] at (-0.2, -0.5) {$5$};		
		
		\node [style=empty] (Z) at (5, 1) {$Z$};
		\node [style=none] (Ztr) at (4.75, 0.75) {};
		\node [style=none] (Ztl) at (5.4, 0.75) {};
		\node [style=none] (Zbl) at (5.4, -0.75) {};
		\node [style=none] (Zbr) at (4.75, -0.75) {};

		\node [style=inputdot] (6) at (5, 0) {};
		\node [style=empty] at (5.2, 0) {$6$};	
		
	\end{pgfonlayer}
	\begin{pgfonlayer}{edgelayer}
%		\draw [style=inarrow] (A) to (a);
%		\draw [style=inarrow] (B) to (a);
%		\draw [style=inarrow] (a) to (C);
%		\draw [style=inarrow] (a) to (D);
		\draw [style=inarrow, bend left=30, looseness=1.00] (E) to (b);
		\draw [style=inarrow, bend left=30, looseness=1.00] (b) to (F);
		\draw [style=inarrow, bend left=30, looseness=1.00] (c) to (E);
		\draw [style=inarrow, bend left=30, looseness=1.00] (F) to (c);
%		\draw [style=inputarrow] (1) to (A);
%		\draw [style=inputarrow] (2) to (B);
%		\draw [style=inputarrow] (3) to (B);
%		\draw [style=inputarrow] (4) to (C);
%		\draw [style=inputarrow] (5) to (D);
		\draw [style=inputarrow] (4) to (E);
		\draw [style=inputarrow] (5) to (E);
		\draw [style=inputarrow] (6) to (F);
%		\draw [style=simple] (Xtl.center) to (Xtr.center);
%		\draw [style=simple] (Xtr.center) to (Xbr.center);
%		\draw [style=simple] (Xbr.center) to (Xbl.center);
%		\draw [style=simple] (Xbl.center) to (Xtl.center);
		\draw [style=simple] (Ytl.center) to (Ytr.center);
		\draw [style=simple] (Ytr.center) to (Ybr.center);
		\draw [style=simple] (Ybr.center) to (Ybl.center);
		\draw [style=simple] (Ybl.center) to (Ytl.center);
		\draw [style=simple] (Ztl.center) to (Ztr.center);
		\draw [style=simple] (Ztr.center) to (Zbr.center);
		\draw [style=simple] (Zbr.center) to (Zbl.center);
		\draw [style=simple] (Zbl.center) to (Ztl.center);
	\end{pgfonlayer}
\end{tikzpicture}
\]
we can compose them and get an open Petri net $Q \circ P \maps X \slashedrightarrow Z$:
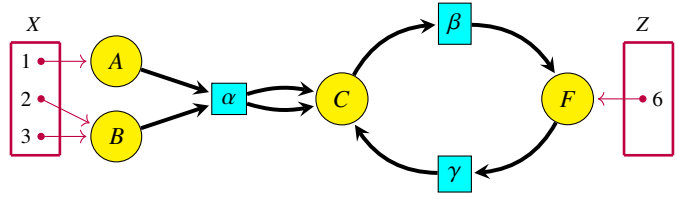
\[
\begin{tikzpicture}
	\begin{pgfonlayer}{nodelayer}
		\node [style=species] (A) at (-4, 0.5) {$A$};
		\node [style=species] (B) at (-4, -0.5) {$B$};;
             \node [style=transition] (a) at (-2.5, 0) {$\alpha$}; 
		\node [style = species] (E) at (-1, 0) {$C$};
		\node [style = species] (F) at (2,0) {$F$};

	     \node [style = transition] (b) at (.5, 1) {$\beta$};
		\node [style = transition] (c) at (.5, -1) {$\gamma$};
		
		\node [style=empty] (X) at (-5.1, 1) {$X$};
		\node [style=none] (Xtr) at (-4.75, 0.75) {};
		\node [style=none] (Xbr) at (-4.75, -0.75) {};
		\node [style=none] (Xtl) at (-5.4, 0.75) {};
             \node [style=none] (Xbl) at (-5.4, -0.75) {};
	
		\node [style=inputdot] (1) at (-5, 0.5) {};
		\node [style=empty] at (-5.2, 0.5) {$1$};
		\node [style=inputdot] (2) at (-5, 0) {};
		\node [style=empty] at (-5.2, 0) {$2$};
		\node [style=inputdot] (3) at (-5, -0.5) {};
		\node [style=empty] at (-5.2, -0.5) {$3$};	
		
		\node [style=empty] (Z) at (3, 1) {$Z$};
		\node [style=none] (Ztr) at (2.75, 0.75) {};
		\node [style=none] (Ztl) at (3.4, 0.75) {};
		\node [style=none] (Zbl) at (3.4, -0.75) {};
		\node [style=none] (Zbr) at (2.75, -0.75) {};

		\node [style=inputdot] (6) at (3, 0) {};
		\node [style=empty] at (3.2, 0) {$6$};	
		
	\end{pgfonlayer}
	\begin{pgfonlayer}{edgelayer}
		\draw [style=inarrow] (A) to (a);
		\draw [style=inarrow] (B) to (a);
	     \draw [style=inarrow, bend right=15, looseness=1.00] (a) to (E);
	     \draw [style=inarrow, bend left =15, looseness=1.00] (a) to (E);	
	     	\draw [style=inarrow, bend left=30, looseness=1.00] (E) to (b);
		\draw [style=inarrow, bend left=30, looseness=1.00] (b) to (F);
		\draw [style=inarrow, bend left=30, looseness=1.00] (c) to (E);
		\draw [style=inarrow, bend left=30, looseness=1.00] (F) to (c);	
		\draw [style=inputarrow] (1) to (A);
		\draw [style=inputarrow] (2) to (B);
		\draw [style=inputarrow] (3) to (B);
		\draw [style=inputarrow] (6) to (F);
		\draw [style=simple] (Xtl.center) to (Xtr.center);
		\draw [style=simple] (Xtr.center) to (Xbr.center);
		\draw [style=simple] (Xbr.center) to (Xbl.center);
		\draw [style=simple] (Xbl.center) to (Xtl.center);
		\draw [style=simple] (Ztl.center) to (Ztr.center);
		\draw [style=simple] (Ztr.center) to (Zbr.center);
		\draw [style=simple] (Zbr.center) to (Zbl.center);
		\draw [style=simple] (Zbl.center) to (Ztl.center);
	\end{pgfonlayer}
\end{tikzpicture}
\]

To formalize this, first note that there is a functor $R \maps \Petri \to \Fin\Set$ sending any Petri net to its set of places.   This has a left adjoint $L \maps \Fin\Set \to \Petri$ sending any finite set $S$ to the Petri net with $S$ as its set of places and no transitions \cite[Lem.\ 11]{BM}.   Since both $\Fin\Set$ and $\Petri$ have finite colimits and $L$ preserves them, Theorem \ref{thm:structured_cospan_2} yields a symmetric monoidal double category ${}_L \lCsp$ in which:
\begin{itemize}
\item an object is a finite set,
\item a tight morphism is a function,
\item a loose morphism is an \define{open Petri net}, meaning a cospan in $\Petri$ of this form:
\[
\scalebox{1.2}{
\begin{tikzpicture}[scale=1.0]
\node (A) at (0,0) {$L(X)$};
\node (B) at (1,1) {$P$};
\node (C) at (2,0) {$L(Y)$};
\path[->,font=\scriptsize,>=angle 90]
(A) edge node[above,pos=0.3]{$i$} (B)
(C)edge node[above,pos=0.3]{$o$}(B);
\end{tikzpicture}
}
\]
\item a 2-cell is a \define{map of open Petri nets}, meaning a commutative diagram in $\Petri$ of this form:
\[
\scalebox{1.2}{
\begin{tikzpicture}[scale=1.2]
\node (E) at (3,0) {$L(X)$};
\node (F) at (5,0) {$L(Y)$};
\node (G) at (4,0) {$P$};
\node (E') at (3,-1) {$L(X')$};
\node (F') at (5,-1) {$L(Y')$.};
\node (G') at (4,-1) {$P'$};
\path[->,font=\scriptsize,>=angle 90]
(F) edge node[above]{$o$} (G)
(E) edge node[left]{$L(f)$} (E')
(F) edge node[right]{$L(g)$} (F')
(G) edge node[left]{$\alpha$} (G')
(E) edge node[above]{$i$} (G)
(E') edge node[below]{$i'$} (G')
(F') edge node[below]{$o'$} (G');
\end{tikzpicture}
}
\]
\end{itemize}
To be more descriptive we call this double category $\lOpen(\Petri)$.

We can equivalently describe open Petri nets using decorated cospans.   There is a symmetric lax monoidal pseudofunctor $F \maps (\Fin\Set, +) \to (\Cat, \times)$ such that for any finite set $S$, the category $F(S)$ has:
\begin{itemize}
\item objects given by Petri nets whose set of places is $S$,
\item morphisms given by morphisms of Petri nets that are the identity on the set of places.
\end{itemize}
By Theorem \ref{thm:decorated_cospan_2} this gives a symmetric monoidal double category $F \lCsp$.  Using Theorem \ref{thm:equivalence_of_structured_and_decorated} we can show that $F \lCsp$ is isomorphic, as a symmetric monoidal double category, to $\lOpen(\Petri)$.   However, it seems simpler to work with open Petri nets using structured cospans.   We begin by using them to define a semantics for open Petri nets.  This illustrates a general method for constructing double functors between structured cospan double categories.

\subsection{The token semantics for open Petri nets}

We have described a `token semantics' mapping Petri nets to commutative monoidal categories.   Now we would like to go further and extend this to a semantics for \emph{open} Petri nets.  This should send open Petri nets to `open commutative monoidal categories'.

To define these, we use the left adjoint functor $L' \maps \Set \to \CMC$ sending any set $X$ to the free commutative monoidal category on this set, which has $\N[X]$ as its set of objects, and only identity morphisms.    
The category $\CMC$ is cocomplete \cite[Thm.\ 16]{BM}.   Thus, all the machinery is in place to use Theorem \ref{thm:structured_cospan_2} to define a symmetric monoidal double category ${}_{L'}\lCsp$ where:
\begin{itemize}
\item an object is a set,
\item a tight morphism is a function,
\item a loose morphism is an \define{open commutative monoidal category}, that is, a cospan in $\CMC$ of the form
\[
\scalebox{1.2}{
\begin{tikzpicture}[scale=1.0]
\node (A) at (0,0) {$L'(X)$};
\node (B) at (1,1) {$C$};
\node (C) at (2,0) {$L'(Y)$,};
\path[->,font=\scriptsize,>=angle 90]
(A) edge node[above,pos=0.3]{$i$} (B)
(C)edge node[above,pos=0.3]{$o$}(B);
\end{tikzpicture}
}
\]
\item a 2-cell is a \define{map of open commutative monoidal categories}, that is, a commutative diagram in $\CMC$ of the form
\[
\scalebox{1.2}{
\begin{tikzpicture}[scale=1.2]
\node (A) at (0,0) {$L'(X)$};
\node (B) at (1,0) {$C$};
\node (C) at (2,0) {$L'(Y)$};
\node (A') at (0,-1) {$L'(X')$};
\node (B') at (1,-1) {$C'$};
\node (C') at (2,-1) {$L'(Y')$.};
\path[->,font=\scriptsize,>=angle 90]
(A) edge node[above]{$i$} (B)
(C) edge node[above]{$o$} (B)
(A') edge node[below]{$i'$} (B')
(C') edge node[below]{$o'$} (B')
(A) edge node [left]{$L'(f)$} (A')
(B) edge node [left]{$\alpha$} (B')
(C) edge node [right]{$L'(g)$} (C');
\end{tikzpicture}
}
\]
\end{itemize}
To be more descriptive we call this double category $\lOpen(\CMC)$.

Next, we want to parlay the operational semantics for Petri nets
\[    F \maps \Petri \to \CMC \]
into an operational semantics for open Petri nets, which should be a symmetric monoidal double functor 
\[   \lOpen(F) \maps \lOpen(\Petri) \to \lOpen(\CMC) .\]
To do this, we can use a general method for constructing double functors between structured cospan double categories:

\begin{thm}
\label{thm:structured_cospan_functoriality_1}
Suppose we have a square in $\Cat$:
\[
\scalebox{1.2}{
\begin{tikzpicture}[scale=1.2]
\node (A) at (0,0) {$\A$};
\node (B) at (1,0) {$\X$};
\node (C) at (0,-1) {$\A'$};
\node (D) at (1,-1) {$\X'$};
\node (E) at (0.5,-0.5) {$\Downarrow \alpha$};
\path[->,font=\scriptsize,>=angle 90]
(A) edge node[above]{$L$} (B)
(B) edge node [right]{$F_1$} (D)
(C) edge node [below] {$L'$} (D)
(A)edge node[left]{$F_0$}(C);
\end{tikzpicture}
}
\] 
where $\X$ and $\X'$ have pushouts, $F_1$ preserves pushouts and $\alpha$ is a natural isomorphism. 
Then there is a double functor $\lF \maps {}_L \lCsp \to {}_{L'} \lCsp$ such that
\begin{itemize}
\item on objects and tight morphisms, $\lF$ acts as $F_0$,
\item $\lF$ sends any loose morphism
\[
\scalebox{1.2}{
\begin{tikzpicture}[scale=1.5]
\node (A) at (0,0) {$L(a)$};
\node (B) at (0.75,0) {$x$};
\node (C) at (1.5,0) {$L(b)$};
\path[->,font=\scriptsize,>=angle 90]
(A) edge node[above]{$i$} (B)
(C)edge node[above]{$o$}(B);
\end{tikzpicture}
}
\]
to
\[
\scalebox{1.2}{
\begin{tikzpicture}[scale=1.5]
\node (A) at (0,0) {$L'(F_0(a))$};
\node (B) at (1.5,0) {$F_1(x)$};
\node (C) at (3,0) {$L'(F_0(b))$};
\path[->,font=\scriptsize,>=angle 90]
(A) edge node[above]{$ F_1(i) \alpha_a$} (B)
(C)edge node[above]{$ F_1(o)\alpha_b $}(B);
\end{tikzpicture}
}
\]
\item $\lF$ sends any 2-cell
\[
\scalebox{1.2}{
\begin{tikzpicture}[scale=1.5]
\node (A) at (0,0) {$L(a)$};
\node (B) at (0.75,0) {$x$};
\node (C) at (1.5,0) {$L(b)$};
\node (A') at (0,-0.8) {$L(a')$};
\node (B') at (0.75,-0.8) {$x'$};
\node (C') at (1.5,-0.8) {$L(b')$};
\path[->,font=\scriptsize,>=angle 90]
(A) edge node[above]{$i$} (B)
(C)edge node[above]{$o$}(B)
(A') edge node[below]{$i'$} (B')
(C')edge node[below]{$o'$}(B')
(C)edge node[right]{$L(g)$}(C')
(B)edge node[left]{$\gamma$}(B')
(A)edge node[left]{$L(f)$}(A');
\end{tikzpicture}
}
\]
to
\[
\scalebox{1.2}{
\begin{tikzpicture}[scale=1.5]
\node (A) at (0,0) {$L'(F_0(a))$};
\node (B) at (1.5,0) {$F_1(x)$};
\node (C) at (3,0) {$L'(F_0(b))$};
\node (A') at (0,-1) {$L'(F_0(a'))$};
\node (B') at (1.5,-1) {$F_1(x')$};
\node (C') at (3,-1) {$L'(F_0(b'))$};
\path[->,font=\scriptsize,>=angle 90]
(A) edge node[above]{$ F_1(i) \alpha_a$} (B)
(C)edge node[above]{$ F_1(o)\alpha_b $}(B)
(A') edge node[below]{$ F_1(i') \alpha_{a'}$} (B')
(C')edge node[below]{$ F_1(o')\alpha_{b'}$}(B')
(A)edge node[left]{$ L'(F_0(f)) $}(A')
(B)edge node[left]{$ F_1(\gamma)$}(B')
(C)edge node[right]{$L'(F_0(g))$}(C');
\end{tikzpicture}
}
\]
\end{itemize}
\end{thm}

\begin{proof} 
This is \cite[Thm.\ 4.2]{BC}.  A double functor involves extra structure besides that mentioned above, and the theorem describes that extra structure for $\lF$.
\end{proof}

As usual, this result has an enhancement that covers the symmetric monoidal case.  This is most easily stated using $\Rex$, the 2-category of categories with finite colimits introduced in Section \ref{subsec:structured_versus_decorated_cospans}.

\begin{thm}
\label{thm:structured_cospan_functoriality_2}
Suppose we have a square in $\Rex$:
\[
\scalebox{1.2}{
\begin{tikzpicture}[scale=1.2]
\node (A) at (0,0) {$\A$};
\node (B) at (1,0) {$\X$};
\node (C) at (0,-1) {$\A'$};
\node (D) at (1,-1) {$\X'$};
\node (E) at (0.5,-0.5) {$\Downarrow \alpha$};
\path[->,font=\scriptsize,>=angle 90]
(A) edge node[above]{$L$} (B)
(B) edge node [right]{$F_1$} (D)
(C) edge node [below] {$L'$} (D)
(A)edge node[left]{$F_0$}(C);
\end{tikzpicture}
}
\] 
Then the double categories ${}_L \lCsp$ and ${}_{L'} \lCsp$ become symmetric monoidal as in 
Theorem \ref{thm:structured_cospan_2}, and the double functor $\lF \maps {}_L \lCsp \to {}_{L'} \lCsp$ 
is symmetric monoidal.
\end{thm}

\begin{proof} 
This is \cite[Thm.\ 4.3]{BC}.  For a double functor to be symmetric monoidal is not just a property: it involves extra structure, and the theorem describes this extra structure.
\end{proof}

We can apply these results if we note that the free commutative monoidal category on the free Petri net on a finite set $S$ is naturally isomorphic to the free commutative monoidal category on $S$, so this
square commutes up to some natural isomorphism $\alpha$:
\[
\scalebox{1.2}{
\begin{tikzpicture}[scale=1.2]
\node (A) at (0,0) {$\Fin\Set$};
\node (B) at (1.2,0) {$\Petri$};
\node (C) at (0,-1) {$\Fin\Set$};
\node (D) at (1.2,-1) {$\CMC$};
\node (E) at (0.5,-0.5) {$\Downarrow \alpha$};
\path[->,font=\scriptsize,>=angle 90]
(A) edge node[above]{$L$} (B)
(B) edge node [right]{$F$} (D)
(C) edge node [below] {$L'$} (D)
(A)edge node[left]{$1$}(C);
\end{tikzpicture}
}
\] 
Furthermore this is a square in $\Rex$.  We thus obtain a symmetric monoidal double functor
$\lF \maps {}_L \lCsp \to {}_{L'} \lCsp$, which we call
\[   \lOpen(F) \maps \lOpen(\Petri) \to \lOpen(\CMC) .\]
This double functor is the token semantics for \emph{open} Petri nets.

\subsection{The uses of 2-cells}

As mentioned, one rationale for the double category approach to open systems is that the composition of these systems is associative only up to isomomorphism.  But there are other advantages to using a double category.   For example, the 2-cells can be used to describe maps from simple open systems to more complicated ones, and vice versa.

Here is a simple open Petri net which describes how an ionized hydrogen atom $\text{H}^+$ and a hydroxyl ion $\text{OH}^-$ combine to form water:
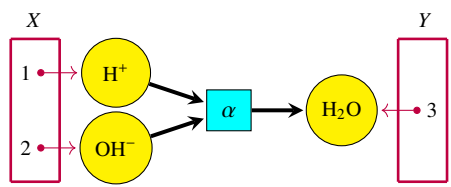
\[
\begin{tikzpicture}
	\begin{pgfonlayer}{nodelayer}
		\node [style=species] (A) at (-4, 0.5) {\scriptsize $\phantom{l}\text{H}^+\phantom{l}$};
		\node [style=species] (B) at (-4, -0.5) {\scriptsize $\text{OH}^-$};
		\node [style=species] (C) at (-1, 0) {\scriptsize $\text{H}_2\text{O}$};	
             \node [style=transition] (a) at (-2.5, 0) {$\phantom{|}\alpha\phantom{|}$}; 
		
		\node [style=empty] (X) at (-5.1, 1.2) {$X$};
		\node [style=none] (Xtr) at (-4.75, 0.95) {};
		\node [style=none] (Xbr) at (-4.75, -0.95) {};
		\node [style=none] (Xtl) at (-5.4, 0.95) {};
             \node [style=none] (Xbl) at (-5.4, -0.95) {};
	
		\node [style=inputdot] (1) at (-5, 0.5) {};
		\node [style=empty] at (-5.2, 0.5) {$1$};
		\node [style=inputdot] (2) at (-5, -0.5) {};
		\node [style=empty] at (-5.2, -0.5) {$2$};

		\node [style=empty] (Y) at (0.1, 1.2) {$Y$};
		\node [style=none] (Ytr) at (.4, 0.95) {};
		\node [style=none] (Ytl) at (-.25, 0.95) {};
		\node [style=none] (Ybr) at (.4, -0.95) {};
		\node [style=none] (Ybl) at (-.25, -0.95) {};

		\node [style=inputdot] (3) at (0, 0) {};
		\node [style=empty] at (0.2, 0) {$3$};		
		
%		\node [style=empty] (Z) at (3, 1) {$Z$};
%		\node [style=none] (Ztr) at (3.25, 0.75) {};
%		\node [style=none] (Ztl) at (2.75, 0.75) {};
%		\node [style=none] (Zbl) at (2.75, -0.75) {};
%		\node [style=none] (Zbr) at (3.25, -0.75) {};
		
	\end{pgfonlayer}
	\begin{pgfonlayer}{edgelayer}
		\draw [style=inarrow] (A) to (a);
		\draw [style=inarrow] (B) to (a);
		\draw [style=inarrow] (a) to (C);
		\draw [style=inputarrow] (1) to (A);
		\draw [style=inputarrow] (2) to (B);
		\draw [style=inputarrow] (3) to (C);
		\draw [style=simple] (Xtl.center) to (Xtr.center);
		\draw [style=simple] (Xtr.center) to (Xbr.center);
		\draw [style=simple] (Xbr.center) to (Xbl.center);
		\draw [style=simple] (Xbl.center) to (Xtl.center);
		\draw [style=simple] (Ytl.center) to (Ytr.center);
		\draw [style=simple] (Ytr.center) to (Ybr.center);
		\draw [style=simple] (Ybr.center) to (Ybl.center);
		\draw [style=simple] (Ybl.center) to (Ytl.center);
	\end{pgfonlayer}
\end{tikzpicture}
\]
There is an obvious `inclusion' 2-cell from the above Petri net to this larger one:
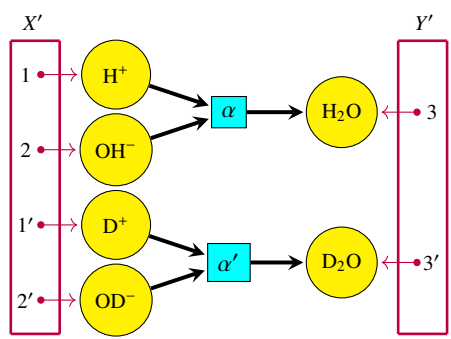
\[
\begin{tikzpicture}
	\begin{pgfonlayer}{nodelayer}
		\node [style=species] (A) at (-4, 0.5) {\scriptsize $\phantom{l}\text{H}^+\phantom{l}$};
		\node [style=species] (B) at (-4, -0.5) {\scriptsize $\text{OH}^-$};
		\node [style=species] (A') at (-4, -1.5) {\scriptsize $\phantom{l}\text{D}^+\phantom{l}$};
		\node [style=species] (B') at (-4, -2.5) {\scriptsize $\text{OD}^-$};
		\node [style=species] (C) at (-1, 0) {\scriptsize $\text{H}_2\text{O}$};	
		\node [style=species] (C') at (-1, -2) {\scriptsize $\text{D}_2\text{O}$};
                  \node [style=transition] (a) at (-2.5, 0) {$\alpha$}; 
                  \node [style=transition] (a') at (-2.5, -2) {$\alpha'$}; 
		
		\node [style=empty] (X) at (-5.1, 1.2) {$X'$};
		\node [style=none] (Xtr) at (-4.75, 0.95) {};
		\node [style=none] (Xbr) at (-4.75, -2.95) {};
		\node [style=none] (Xtl) at (-5.4, 0.95) {};
             \node [style=none] (Xbl) at (-5.4, -2.95) {};
	
		\node [style=inputdot] (1) at (-5, 0.5) {};
		\node [style=empty] at (-5.2, 0.5) {$1$};
		\node [style=inputdot] (2) at (-5, -0.5) {};
		\node [style=empty] at (-5.2, -0.5) {$2$};
		\node [style=inputdot] (1') at (-5, -1.5) {};
		\node [style=empty] at (-5.2, -1.5) {$1'$};
		\node [style=inputdot] (2') at (-5, -2.5) {};
		\node [style=empty] at (-5.2, -2.5) {$2'$};

		\node [style=empty] (Y) at (0.1, 1.2) {$Y'$};
		\node [style=none] (Ytr) at (.4, 0.95) {};
		\node [style=none] (Ytl) at (-.25, 0.95) {};
		\node [style=none] (Ybr) at (.4, -2.95) {};
		\node [style=none] (Ybl) at (-.25, -2.95) {};

		\node [style=inputdot] (3) at (0, 0) {};
		\node [style=empty] at (0.2, 0) {$3$};	
		\node [style=inputdot] (3') at (0, -2) {};
		\node [style=empty] at (0.2, -2) {$3'$};		
		
%		\node [style=empty] (Z) at (3, 1) {$Z$};
%		\node [style=none] (Ztr) at (3.25, 0.75) {};
%		\node [style=none] (Ztl) at (2.75, 0.75) {};
%		\node [style=none] (Zbl) at (2.75, -0.75) {};
%		\node [style=none] (Zbr) at (3.25, -0.75) {};
		
	\end{pgfonlayer}
	\begin{pgfonlayer}{edgelayer}
		\draw [style=inarrow] (A) to (a);
		\draw [style=inarrow] (B) to (a);
		\draw [style=inarrow] (a) to (C);
		\draw [style=inputarrow] (1) to (A);
		\draw [style=inputarrow] (2) to (B);
		\draw [style=inputarrow] (3) to (C);
		\draw [style=inarrow] (A') to (a');
		\draw [style=inarrow] (B') to (a');
		\draw [style=inarrow] (a') to (C');
		\draw [style=inputarrow] (1') to (A');
		\draw [style=inputarrow] (2') to (B');
		\draw [style=inputarrow] (3') to (C');
		\draw [style=simple] (Xtl.center) to (Xtr.center);
		\draw [style=simple] (Xtr.center) to (Xbr.center);
		\draw [style=simple] (Xbr.center) to (Xbl.center);
		\draw [style=simple] (Xbl.center) to (Xtl.center);
		\draw [style=simple] (Ytl.center) to (Ytr.center);
		\draw [style=simple] (Ytr.center) to (Ybr.center);
		\draw [style=simple] (Ybr.center) to (Ybl.center);
		\draw [style=simple] (Ybl.center) to (Ytl.center);
	\end{pgfonlayer}
\end{tikzpicture}
\]
Hydrogen has two stable isotopes, the usual one $\text{H}$ and a heavier one called deuterium, $\text{D}$.  In this second Petri net we are including a second reaction involving deuterium, which can create so-called `heavy water' $\text{D}_2\text{O}$.  

Conversely, there is a 2-cell from the second open Petri net to the first which forgets the difference between $\text{H}$ and $\text{D}$, maps the transitions $\alpha$ and $\alpha'$ to $\alpha$, and maps the sets $X'$ and $Y'$ to $X$ and $Y$ in the obvious way.   

This example may be too simple to be interesting to chemists, but it illustrates two key uses of 2-cells in structured cospan double categories:
\begin{itemize}
\item We can include a simpler model of an open system in a more complicated one.
\item We can project a more complicated model of an open system down to a simpler one. 
\end{itemize}
Both processes are important in modeling.   As we continue to change a model of an open system, either refining it or simplifying it, we should not have to treat each model we build as a separate, isolated
entity.   It is better, when possible, to keep track of 2-cells between them.  This lets us treat the history of the modeling process as a formal entity in its own right.

Furthermore, the category of loose morphisms and 2-cells between them may have pullbacks.   This is true, for example, of $\lOpen(\Petri)$.   These pullbacks let us build more complicated open systems from simpler ones, in a process that modelers---somewhat confusingly to mathematicians---call `stratification'.   To see how pullbacks have been used to stratify Petri nets in the sphere of public health modeling, see \cite{BFHLP}.   

\subsection{Whole-grain Petri nets}
\label{subsec:whole-grain}

There are various alternative notions of Petri net beside the one we have been using here \cite{BGMS}.   While it is a digression, we would be remiss not to mention Kock's `whole-grain Petri nets' \cite{Kock}, for two reasons.  First, when people implement open Petri nets using structured cospans in category-based software, they often use whole-grain Petri nets \cite{AlgebraicPetri,BFHLP}.  Second, while the Petri nets discussed so far present only free \emph{commutative} monoidal categories, whole-grain Petri nets have the ability to present free symmetric strict monoidal categories.

In the Petri nets discussed so far, the arrows between places and transitions have no real individuality: permuting them has no effect.    Thus, instead of drawing a \emph{finite set} of arrows from $B$ to $\tau$ or from $\tau$ to $C$ in this picture:
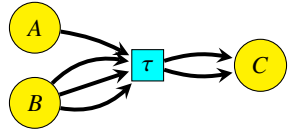
\[
\begin{tikzpicture}
	\begin{pgfonlayer}{nodelayer}
\node [style=species] (C) at (-1, 0) {$C$};
\node [style=species] (B) at (-4, -0.5) {$B$};
		\node [style=species] (A) at (-4, 0.5) {$A$};
		\node [style=transition] (tau) at (-2.5, 0) {$\tau$};
	\end{pgfonlayer}
	\begin{pgfonlayer}{edgelayer}
		\draw [style=inarrow, bend left = 10, looseness=1.0] (A) to (tau);
		\draw [style=inarrow, bend left = 30, looseness=1.0] (B) to (tau);
		 \draw [style=inarrow, bend left=0, looseness=1.00] (B) to (tau); 
                 \draw [style=inarrow, bend right =30, looseness=1.00] (B) to (tau); 
		\draw [style=inarrow, bend left=15, looseness=1.00] (tau) to (C);
	        \draw [style=inarrow, bend right=15, looseness=1.00] (tau) to (C);
	\end{pgfonlayer}
\end{tikzpicture}
\]
it would be more honest to put a \emph{natural number} on each arrow, like this:
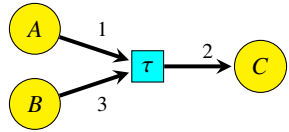
\[
\begin{tikzpicture}
	\begin{pgfonlayer}{nodelayer}
\node [style=species] (C) at (-1, 0) {$C$};
\node [style=species] (B) at (-4, -0.5) {$B$};
\node [style=species] (A) at (-4, 0.5) {$A$};
\node [style=transition] (tau) at (-2.5, 0) {$\tau$};
\node at (-3.1,0.5){$1$};
\node at (-3.1,-0.5){$3$};
\node at (-1.7,0.2){$2$};
	\end{pgfonlayer}
	\begin{pgfonlayer}{edgelayer}
		\draw [style=inarrow] (A) to (tau);
		 \draw [style=inarrow] (B) to (tau); 
		\draw [style=inarrow] (tau) to (C);
	\end{pgfonlayer}
\end{tikzpicture}
\]
However, in a whole-grain Petri net there really is a finite set $I$ of arrows called \define{input arcs} going from places to transitions, and a finite set $O$ of \define{output arcs} going from transitions to places.   

More precisely, a \textbf{whole-grain Petri net} is a diagram of finite sets
  \[
\adjustbox{scale=1.2,center}{
    \begin{tikzcd}
      S & I \ar[l] \ar[r] & T & O \ar[l] \ar[r] & S.
    \end{tikzcd}
}
\]
 A morphism of whole-grain Petri nets, sometimes called an \define{etale map}, is a diagram
  \[
  \adjustbox{scale=1.2,center}{
    \begin{tikzcd}
      S \ar[d] & I \ar[l] \ar[r]\ar[d] \ar[dr,phantom,near start,"\lrcorner"] & T \ar[d] & O \ar[l] \ar[r]\ar[d] \ar[dl,phantom,near start,"\llcorner"] & S\ar[d]\\
      S' & I' \ar[l] \ar[r] & T' & O' \ar[l] \ar[r] & S'.
    \end{tikzcd}
}
  \]
Without the pullback conditions here, there would be a map from this whole-grain Petri net:
 
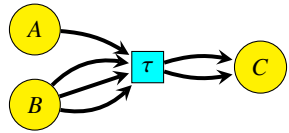
\[
\begin{tikzpicture}
	\begin{pgfonlayer}{nodelayer}
\node [style=species] (C) at (-1, 0) {$C$};
\node [style=species] (B) at (-4, -0.5) {$B$};
		\node [style=species] (A) at (-4, 0.5) {$A$};
		\node [style=transition] (tau) at (-2.5, 0) {$\tau$};
	\end{pgfonlayer}
	\begin{pgfonlayer}{edgelayer}
		\draw [style=inarrow, bend left = 15, looseness=1.0] (A) to (tau);
		\draw [style=inarrow, bend left = 30, looseness=1.0] (B) to (tau);
		 \draw [style=inarrow, bend left=0, looseness=1.00] (B) to (tau); 
                 \draw [style=inarrow, bend right =30, looseness=1.00] (B) to (tau); 
		\draw [style=inarrow, bend left=15, looseness=1.00] (tau) to (C);
	        \draw [style=inarrow, bend right=15, looseness=1.00] (tau) to (C);
	\end{pgfonlayer}
\end{tikzpicture}
\]
to this one:
  
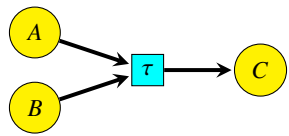
\[
\begin{tikzpicture}
	\begin{pgfonlayer}{nodelayer}
\node [style=species] (C) at (-1, 0) {$C$};
\node [style=species] (B) at (-4, -0.5) {$B$};
		\node [style=species] (A) at (-4, 0.5) {$A$};
		\node [style=transition] (tau) at (-2.5, 0) {$\tau$};
	\end{pgfonlayer}
	\begin{pgfonlayer}{edgelayer}
		\draw [style=inarrow] (A) to (tau);
		 \draw [style=inarrow, bend left=0, looseness=1.00] (B) to (tau);
		\draw [style=inarrow] (tau) to (C);
	\end{pgfonlayer}
\end{tikzpicture}
\]

Just as there is a left adjoint functor
\[    F \maps \Petri \to \CMC, \]
there is a left adjoint
\[    F_\wg \maps \wg\Petri\to \SSMC \]
from the category $\wg\Petri$ of whole-grain Petri nets and etale maps to the category of
symmetric strict monoidal categories and symmetric strict monoidal functors.  
Just as the functor $F$ gives rise to a symmetric monoidal double functor 
\[   \lOpen(F) \maps \lOpen(\Petri) \to \lOpen(\CMC) ,\]
the functor $F_\wg$ gives a symmetric monoidal double functor
\[   \lOpen(F_\wg) \maps \lOpen(\Petri_\wg) \to \lOpen(\CMC_\wg). \]
Moreover, there is a square of symmetric monoidal double functors, commuting up to isomorphism:
\[
\scalebox{1.2}{
\begin{tikzpicture}[scale=1.5]
\node (A) at (0,0) {$\lOpen(\Petri_\wg)$};
\node (B) at (2,0) {$\lOpen(\SSMC)$};
\node (C) at (0,-1) {$\lOpen(\Petri)$};
\node (D) at (2,-1) {$\lOpen(\CMC)$};
\path[->,font=\scriptsize,>=angle 90]
(A) edge node[above]{$\lOpen(F)$} (B)
(B) edge node [right]{$\lOpen(G)$} (D)
(C) edge node [below] {$\lOpen(F_\wg)$} (D)
(A)edge node[left]{$\lOpen(H)$}(C);
\end{tikzpicture}
}
\] 
where $\lOpen(G)$ comes from a left adjoint functor $G \maps \Petri_\wg \to \Petri$, and
$\lOpen(H)$ comes from a left adjoint $H \maps \SSMC \to \CMC$.  For details, see \cite[Sec.\ 9]{BGMS}.

\section{Dynamical systems}
\label{sec:dynamical}

Next we consider a double category of open dynamical systems, defined using decorated cospans.   There are many kinds of dynamical system, but we shall discuss only one, as an example: a vector field $v$ on $\R^n$, which we can treat as simply a function $v \maps \R^n \to \R^n$.    This determines a first-order ordinary differential equation
\begin{equation}
\label{eq:diffeq_basic}     
\frac{d}{dt} x(t) = v(x(t))   
\end{equation}
which describes how a point $x(t) \in \R^n$ depends on time $t \in \R$.     Higher-order ordinary differential equations can be put into first-order form by including extra variables.   In physics we commonly treat $\R^n$ as the set of states of some system, and use a differential equation as above to describe how states evolve in time.  

More conceptually, we can replace $\R^n$ with $\R^S$ where $S$ is any finite set, whose elements we call  \define{variables}.   To each variable $a \in S$ we associate a real-valued function of time, say $[a] \maps \R \to \R$.   The function $x \maps \R \to \R^S$ contains the information in all these functions.    We can drop the brackets and use the same notation for the variables and their corresponding functions, but for expository purposes we wish to clarify the distinction between them.

To make this idea precise we should single out some class of vector fields: for example continuous, smooth, etc.    There is a lot of flexibility here.   If we restrict to smooth bounded vector fields, the differential equation will have a unique solution for all times for any initial data $x(0)$, and this solution $x(t)$ will be a smooth function of time.    However, in our applications to Petri nets in Section \ref{sec:Petri}, we need smooth vector fields that may not be bounded.   For the sake of specificity, we shall use these.  Beware: in this case, while the differential equation is guaranteed to have a smooth solution `locally in time' for any initial condition $x(0)$, the solution may shoot off to infinity and become undefined after a while.  

With these decisions made, for any finite set $S$ we define 
\[ D(S) = \{ v \maps \R^S \to \R^S | \; v \textrm{ is smooth}  \}. \]
We define a \define{dynamical system} to be a finite set $S$ together with an element of $D(S)$.   We then define an \define{open dynamical system} to be a cospan of finite sets where the apex, say $S$, is decorated by an element of $D(S)$:
\[
\scalebox{1.2}{
\begin{tikzpicture}
\node (A) at (0,0) {$X$};
\node (B) at (1,1) {$S$};
\node (C) at (2,0) {$Y$};
\node (D) at (3,1) {$v \in D(S)$};
\path[->,font=\scriptsize,>=angle 90]
(A) edge node[above,pos = 0.3]{$i$} (B)
(C) edge node[above, pos = 0.3]{$o$} (B);
\end{tikzpicture}
}
\]

How can we compose open dynamical systems?    Let us look at an example.   Consider this open dynamical system:
\[
\scalebox{1.2}{
\begin{tikzpicture}
\node (A) at (0,0) {$\{a\}$};
\node (B) at (1,1) {$\{a,b\}$};
\node (C) at (2,0) {$\{b\}$};
\node (E) at (3,1) {$v \in D(\R^{\{a,b\}})$};
\path[->,font=\scriptsize,>=angle 90]
(A) edge node[above,pos=0.3]{$i$} (B)
(C) edge node[above,pos=0.3]{$o$} (B);
\end{tikzpicture}
}
\]
where for simplicity the functions $i$ and $o$ map each variable to the like-named variable.   The dynamical system here specifies the differential equations
\begin{equation}
\label{eq:diffeq_1}
\begin{array}{ccl}
\displaystyle{\frac{d [a]}{d t}} &=& v_a([a], [b])  \\ [8 pt]
\displaystyle{\frac{d [b]}{d t}} &=& v_b([a], [b]) 
\end{array}
\end{equation}
where $[a], [b] \maps \R \to \R$ are the functions associated to the variable names $a, b \in X$. 

Suppose we compose the above open dynamical system with the following one:
\[
\scalebox{1.2}{
\begin{tikzpicture}[column sep={.4in,between origins}, row sep=.08in]
\node (A) at (0,0) {$\{b\}$};
\node (B) at (1,1) {$\{b,c\}$};
\node (C) at (2,0) {$\{c\}$};
\node (E) at (3,1) {$w\in D(\R^{\{b,c\}})$};
\path[->,font=\scriptsize,>=angle 90]
(A) edge node[above,pos = 0.3]{$i'$} (B)
(C) edge node[above,pos = 0.3]{$\; o'$} (B);
\end{tikzpicture}
}
\]
which describes the differential equations
\begin{equation}
\label{eq:diffeq_2}
\begin{array}{ccl}
\displaystyle{ \frac{d [b]}{d t} } &=& w_b([b],[c])  \\ [8 pt]
\displaystyle{ \frac{d [c]}{d t} } &=& w_c([b],[c]) 
\end{array}
\end{equation}
 Since we compose cospans by taking pushouts, the composite will be of the form 
 \[
 \scalebox{1.2}{
\begin{tikzpicture}[column sep={.4in,between origins}, row sep=.08in]
\node (A) at (0,0) {$\{a\}$};
\node (B) at (1,1) {$\{a,b,c\}$};
\node (C) at (2,0) {$\{c\}$};
\node (E) at (3,1) {$u \in D(\R^{\{a,b,c\}})$};
\path[->,font=\scriptsize,>=angle 90]
(A) edge node[above,pos=0.25]{$i$} (B)
(C) edge node[above,pos=0.25]{$\; o'$} (B);
\end{tikzpicture}
}
\]
where $u$ is some vector field.  

What should $u$ be?   We want to combine Equations \eqref{eq:diffeq_1} and \eqref{eq:diffeq_2} somehow, identifying the variable $b$ in the first set of equations with the like-named variable in the second set---not because they happen to have the same name, but because those variables get identified when taking the pushout.  Since we have two different equations describing $d[b]/dt$, we need to reconcile these somehow.   We do what a physicist would do, and \emph{add} the right-hand sides of these equations:
\begin{equation}
\label{eq:diffeq_3}
\begin{array}{ccl}
\displaystyle{ \frac{d [a]}{d t} } &=& v_a([a], [b])  \\ [8 pt]
\displaystyle{ \frac{d [b]}{d t} } &=& v_b([a], [b]) + w_b([b], [c])  \\ [8 pt]
\displaystyle{ \frac{d [c]}{d t} } &=& w_c([b], [c]) 
\end{array}
\end{equation}
Since $[b]$ is changing for two different reasons, we sum those effects.   This is a nontrivial decision on our part, and we could decide to do something else.  But this choice is surprisingly effective.  For example, in classical mechanics, when an object is affected by several forces, the rate of change of its velocity is the \emph{sum} of the rates of change you would predict from each force individually.     In chemistry, when molecules of a certain sort are getting created by several reactions simultaneously, the rate at which their number changes is the \emph{sum} of the rates of change due to the various individual reactions.  We shall see examples of both kinds.

From Equation \eqref{eq:diffeq_3} we can read off the vector field \(u\):
\[   u([a],[b],[c]) = \big( v_a ([a],[b]), \; v_b([a], [b]) + w_b([b], [c]),  \;  w_c([b], [c])  \big) .\]
But how do we formalize this method of composing open dynamical systems \emph{in general}, not just in this one example?     We can use theory of decorated cospans.   

First, we need some generalities.  Given any function $f \maps S \to S'$ between finite sets, we define the \define{pullback} $ f^* \maps \R^{S'} \to \R^S $ to be the linear map given by
\[ f^*(\psi)(s) = \psi(f(s)) \]
for all $\psi \in \R^{S'}$, $s \in S$.   This defines a contravariant functor from $\Fin\Set$ to $\Fin\Vect_\R$, the category of finite-dimensional real vector spaces and linear maps.   We define the \define{pushforward} $ f_* \maps \R^{S} \to \R^{S'} $ by
\[ f_*(\psi)(s') = \sum_{ \{ s \in S : f(s) = s' \} } \psi(s) \]
for all $\psi \in \R^S$, $s' \in S$.  This defines a covariant functor from $\Fin\Set$ to $\Fin\Vect_\R$.

Then, we can show there is a symmetric lax monoidal functor $D \maps \Fin\Set \to \Set$ such that:
\begin{itemize}
\item $D$ maps any finite set $S$ to 
\[ D(S) = \{ v \maps \R^S \to \R^S | \; v \textrm{ is smooth}  \}, \]
\item $D$ maps any function $f \maps S \to S'$ between finite sets to the function $D(f) \maps D(S) \to D(S')$ given by
\[ D(f)(v) = f_* \circ v \circ f^* \]
for all $v \in D(S)$,
\item the laxator $\delta_{S,S'} \maps D(S) \times D(S') \to D(S+S')$ is given by
\[    \delta_{S,S'}(v,v') = i_* \circ v \circ i^* + i'_* \circ v' \circ {i'}^*,
\] 
where $i \maps S \to S+S'$ and $i' \maps S' \to S+S'$ are the inclusions into the coproduct.
\item the unitor is the unique map  $\delta \maps 1 \to D(\emptyset)$, since there is only one vector field
on the empty set.
\end{itemize}
For details, see \cite[Sec.\ 6]{BP}.  The laxator is chosen to create the summing effect we saw in our earlier example.  

Since every set gives a discrete category with that set of objects, we can reinterpret $D$ as a symmetric lax monoidal functor $D \maps (\Fin\Set, +) \to (\Cat, \times)$.  Since a functor to $\Cat$ is a special case of a pseudofunctor, we can use Theorem \ref{thm:decorated_cospan_2} and obtain a symmetric monoidal double category $D \lCsp$ of \define{open dynamical systems}.  In this double category
\begin{itemize}
\item an object is a finite set,
\item a tight morphism is a function,
\item a loose morphism from a finite set $X$ to a finite set $Y$ is an open dynamical system
\[
\scalebox{1.2}{
\begin{tikzpicture}
\node (A) at (0,0) {$X$};
\node (B) at (1,1) {$S$};
\node (C) at (2,0) {$Y$};
\node (D) at (3,1) {$v \in D(S)$};
\path[->,font=\scriptsize,>=angle 90]
(A) edge node[above,pos = 0.3]{$i$} (B)
(C) edge node[above, pos = 0.3]{$o$} (B);
\end{tikzpicture}
}
\]
\item a 2-cell is a commuting diagram
\[
\scalebox{1.2}{
\begin{tikzpicture}[scale=1.5]
\node (E) at (3,0) {$X$};
\node (F) at (5,0) {$Y$};
\node (G) at (4,0) {$S$};
\node (H) at (6.5,0) {$v \in D(S)$};
\node (E') at (3,-1) {$X'$};
\node (F') at (5,-1) {$Y'$};
\node (G') at (4,-1) {$S'$};
\node (H') at (6.5,-1) {$v \in D(S')$};
\path[->,font=\scriptsize,>=angle 90]
(F) edge node[above]{$o$} (G)
(E) edge node[left]{$f$} (E')
(F) edge node[right]{$g$} (F')
(G) edge node[left]{$h$} (G')
(E) edge node[above]{$i$} (G)
(E') edge node[below]{$i'$} (G')
(F') edge node[below]{$o'$} (G');
\end{tikzpicture}
}
\]
in $\Fin\Set$ such that $D(h)(v) = v'$.
\end{itemize}

\subsection{The open dynamical system equation}
\label{subsec:open_dynamical_system_equation}

We have seen how open dynamical systems can be composed.  But we have not fully lived up to one of our claims: that we can use double categories to study systems that interact with their environment.  We have seen how an open dynamical system
\[
\scalebox{1.2}{
\begin{tikzpicture}
\node (A) at (0,0) {$X$};
\node (B) at (1,1) {$S$};
\node (C) at (2,0) {$Y$};
\node (D) at (3,1) {$v \in D(S)$};
\path[->,font=\scriptsize,>=angle 90]
(A) edge node[above,pos = 0.3]{$i$} (B)
(C) edge node[above, pos = 0.3]{$o$} (B);
\end{tikzpicture}
}
\]
determines a differential equation
\[
 \frac{d}{dt} x(t) = v(x(t)).  
\]
However, this is a so-called `autonomous' differential equation, in which the state of the system at any time completely determines its future state, without the intrusion of any influences from the outside world.   So far, the open dynamical system interacts with other systems only after we compose them as loose morphisms in $D\lCsp$.   We now generalize to a formalism where the outside world, \emph{not modeled by any particular dynamical system}, can affect an open dynamical system through its interfaces.

For example, return to this open dynamical system:
\[
\scalebox{1.2}{
\begin{tikzpicture}
\node (A) at (0,0) {$\{a\}$};
\node (B) at (1,1) {$\{a,b\}$};
\node (C) at (2,0) {$\{b\}.$};
\node (E) at (3,1) {$u \in D(\R^{\{a,b\}})$};
\path[->,font=\scriptsize,>=angle 90]
(A) edge node[above,pos=0.3]{$i$} (B)
(C) edge node[above,pos=0.3]{$o$} (B);
\end{tikzpicture}
}
\]
Suppose we add terms to its differential equation, Equation \eqref{eq:diffeq_1}, as follows: 
\begin{equation}
\label{eq:diffeq_4}
\begin{array}{ccl}
\displaystyle{\frac{d [a]}{d t}} &=& u_1([a],[b]) + f(t)  \\ [8 pt]
\displaystyle{\frac{d [b]}{d t}} &=& u_2([a],[b]) + g(t)
\end{array}
\end{equation}
Now the equation is no longer autonomous, because $f$ and $g$ are arbitrary functions of time.  They represent additional `external influences'.   

More generally, let $I \maps \R \to \R^X $ and $O \maps \R \to \R^Y$ be arbitrary smooth functions of time.     We can compose these with the pushforward maps $i_* \maps \R^X \to \R^S$ and $o_* \maps \R^Y \to \R^S$ to add extra terms to our autonomous differential equation, Equation \eqref{eq:diffeq_basic}, and obtain the \define{open dynamical system equation}
\begin{equation}
\label{eq:open_dynamical_system}
 \frac{dx(t)}{dt} = v(x(t)) + i_*(I(t)) - o_*(O(t)) 
\end{equation}
The minus sign is just a matter of convention.  It breaks the symmetry between left and right interfaces, so it may be considered undesirable.   We could leave it out altogether.   However, we shall soon turn to some examples where the variables represent `stocks' or `concentrations'---loosely, amounts of stuff---which change in time due to `flows'.    Then it is common to adopt a convention where flows from left to right are given a positive sign, while flows in the other direction are given a negative sign.   In this situation it is reasonable to call $I$ the \define{inflow} and $O$ the \define{outflow}.   

\iffalse 
In examples of this sort, when we compose two open dynamical systems:
\[
\scalebox{1.2}{
\begin{tikzpicture}[scale=1.2]
\node (A) at (0,0) {$X$};
\node (B) at (1,1) {$S$};
\node (C) at (2,0) {$Y$};
\node (D) at (1.8,1) {\scriptsize $v \in D(S)$};
\node (E) at (3,1) {$T$};
\node (F) at (4,0) {$Z$};
\node (G) at (3.8,1) {\scriptsize $w \in D(T)$};
\path[->,font=\scriptsize,>=angle 90]
(A) edge node[above,pos = 0.3]{$i$} (B)
(C) edge node[above, pos = 0.3]{$o$} (B)
(C) edge node[above,pos = 0.3]{$i'$} (E)
(F) edge node[above, pos = 0.3]{$o'$} (E);
\end{tikzpicture}
}
\]
 stuff flowing out of $S$ through the interface $Y$ flows into $T$.   Similarly, stuff flowing out of $T$ through the interface $Y$ flows into $S$.   
 \fi
 
\subsection{Classical mechanics}
\label{subsec:classical_mechanics}

Let us see this formalism at work in classical mechanics.    Suppose we have two particles on the line, connected to each other by a spring.   We can describe them as a dynamical system with four variables, since each particle's state is described by one position variable $q_i$ and one momentum variable $p_i$:
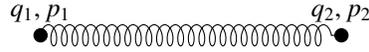
\[
\begin{tikzpicture}
%\draw[\dwinvisiblebordercolor, opacity = \dwinvisibleborderopacity, line width = \dwinvisibleborderwidth] (-2, -1.5) rectangle (2,.45);%.3 space
%
%Rocks
\coordinate(r1) at (-2,0);
\coordinate(r2) at (2,0);
% \coordinate(r3) at (1.5,0);
% Arrow Tails
{\draw[decorate,decoration={coil,segment length=4pt, amplitude = 4pt},rotate=0] (r1)  -- (r2);}
% {\draw[decorate,decoration={coil,segment length=4pt, amplitude = 4pt},rotate=0] (r2)  -- (r3);}
% Rocks
{\draw[fill= black] (r1) circle (2.5pt);}
{\draw[fill= black] (r2) circle (2.5pt);}
\node at (-2,0.3) {\normalsize $q_1,p_1$};
\node at (2,0.3) {\normalsize $q_2,p_2$};
% {\draw[fill= black] (r3) circle (2.5pt);}
\end{tikzpicture}
\]
 Suppose the $i$th particle has mass $m_i \in \R$ and the spring has spring constant $k$.  The first particle feels a force $k(q_2 - q_1)$ pulling it toward the second, while the second feels a force $k(q_1 - q_2)$ pulling it toward the first.   Since momentum is mass times velocity and force is the time derivative of momentum, we have
\begin{equation}
\label{eq:rocks}
\begin{array}{ccl} 
\displaystyle{ \frac{d}{dt} q_1(t)}  &=& p_1(t)/m_1 \\ [8pt]  
\displaystyle{ \frac{d}{dt} p_1(t)}  &=&k (q_2(t) - q_1(t)) \\ [8pt]  
\displaystyle{ \frac{d}{dt} q_2(t)}  &=& p_2(t)/m_2 \\ [8pt]  
\displaystyle{ \frac{d}{dt} p_2(t)} &=& k(q_1(t) - q_2(t))
\end{array}
\end{equation}
Now we are acting like physicists and not distinguishing the variables $q_i, p_i$ from the functions
$[q_i] \maps \R \to \R, [p_i] \maps \R \to \R$ they name.
If we let
\[       x(t) = (q_1(t), p_1(t), q_2(t), p_2(t)) \in \R^4 \]
we can write Equation \eqref{eq:rocks} more tersely as
\[    \frac{d}{dt} x(t) = v(x(t))   \]
where
\[    v(q_1, p_1, q_2, p_2) = \big(p_1/m_1, k(q_2 - q_1), p_2/m, k(q_1 - q_2) \big) .\]

We can also treat this system as an \emph{open} dynamical system, for example by decreeing that the first particle is the left interface and the second is the right interface:
\begin{equation}
\label{eq:rocks_as_open_system}
\scalebox{1.2}{
\begin{tikzpicture}
\node (A) at (0,0) {$\{q_1,p_1\}$};
\node (B) at (1,1) {$\{q_1,p_1,q_2,p_2\}$};
\node (C) at (2,0) {$\{q_2,p_2\}.$};
\node (E) at (4,1) {$v \in D(\R^{\{q_1,p_1,q_2,p_2\}})$};
\path[->,font=\scriptsize,>=angle 90]
(A) edge node[above,pos=0.2]{$i$} (B)
(C) edge node[above,pos=0.2]{$o$} (B);
\end{tikzpicture}
}
\end{equation}
What does the open dynamical system equation, Equation \eqref{eq:open_dynamical_system}, look like in this case?   For simplicity suppose we take
\[      I(t) = (0, F_1(t)),   \qquad  O(t) = (0, F_2(t)). \]
Then the open dynamical system equation becomes
\[
\begin{array}{ccl} 
\displaystyle{ \frac{d}{dt} q_1}  &=& p_1/m_1 \\ [8pt]  
\displaystyle{ \frac{d}{dt} p_1}  &=&k (q_2 - q_1) + F_1(t) \\ [8pt]  
\displaystyle{ \frac{d}{dt} q_2}  &=& p_2/m_2 \\ [8pt]  
\displaystyle{ \frac{d}{dt} p_2} &=& k(q_1 - q_2) - F_2(t).
\end{array}
\]
This system of equations describes two rocks connected by a spring where the first is pushed to the right by an external force $F_1(t)$ and the second is pushed to the left by an external force $F_2(t)$.   If the sign convention here feels unnatural, the reader is free to get rid of the minus sign in Equation \eqref{eq:open_dynamical_system}, but we will try to justify the sign in the next two sections.

\begin{exercise}
\label{ex:three_rocks_1}
Describe a system of three particles connected by springs:
 
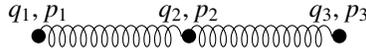
\[
\begin{tikzpicture}
%\draw[\dwinvisiblebordercolor, opacity = \dwinvisibleborderopacity, line width = \dwinvisibleborderwidth] (-2, -1.5) rectangle (2,.45);%.3 space
%
%Rocks
\coordinate(r1) at (-2,0);
\coordinate(r2) at (0,0);
\coordinate(r3) at (2,0);
% Arrow Tails
{\draw[decorate,decoration={coil,segment length=4pt, amplitude = 4pt},rotate=0] (r1)  -- (r2);}
{\draw[decorate,decoration={coil,segment length=4pt, amplitude = 4pt},rotate=0] (r2)  -- (r3);}
% Rocks
{\draw[fill= black] (r1) circle (2.5pt);}
{\draw[fill= black] (r2) circle (2.5pt);}
{\draw[fill= black] (r3) circle (2.5pt);}
\node at (-2,0.3) {\normalsize $q_1,p_1$};
\node at (0,0.3) {\normalsize $q_2,p_2$};
\node at (2,0.3) {\normalsize $q_3,p_3$};
\end{tikzpicture}
\]
as an open dynamical system.  To do this, compute the composite of two open dynamical systems of the kind just described, regarded as loose morphisms in $D\lCsp$:
\[
\scalebox{1.2}{
\begin{tikzpicture}
\node (A) at (0,0) {$\{q_1,p_1\}$};
\node (B) at (2,1) {$S = \{q_1,p_1,q_2,p_2\}$};
\node (C) at (4,0) {$\{q_2,p_2\}$};
\node (E) at (3.5,1.3) {\scriptsize $v \in D(S)$};
\node (F) at (6,1) {$T = \{q_2,p_2,q_3,p_3\}$};
\node (G) at (8,0) {$\{q_3,p_3\}.$};
\node (H) at (7.5,1.3) {\scriptsize $w \in D(T)$};
\path[->,font=\scriptsize,>=angle 90]
(A) edge node[above,pos=0.2]{$i$} (B)
(C) edge node[above,pos=0.2]{$o$} (B)
(C) edge node[above,pos=0.2]{$i'$} (F)
(G) edge node[above,pos=0.2]{$o'$} (F);
\end{tikzpicture}
}
\]
\end{exercise}

\begin{exercise}
\label{ex:three_rocks_2}
Figure out the open dynamical system equation for the system in Exercise \ref{ex:three_rocks_1}, assuming the first particle feels an external force $F_1(t)$ pushing to the left, while the third feels an external force $F_3(t)$ pushing to the right.   In this example the second particle is not part of an interface, so it feels no external force.
\end{exercise}

The reader may wonder whether such an elaborate formalism is required to study three rocks connected by two springs.   It is not.   What the formalism does is to make explicit the physicist's intuitive understanding of how to take two systems of differential equations, each describing a physical system, and combine them to describe a larger system built from those two parts.   This should be useful both for mathematicians wishing to prove theorems about composites of open systems, and also people who want to automate the process of modeling composite systems.

The examples just discussed can be vastly generalized.   For systems of finitely many particles moving in $\R^n$ and interacting by arbitrary forces, this is straightforward.   It is more challenging to generalize the framework to handle open Lagrangian systems where the configuration space is a manifold, open Hamiltonian systems where the phase space is a symplectic manifold , and port-Hamiltonian systems where the phase space is a Dirac manifold.   Much work has been done on these issues \cite{BE,BWY,Coya2017,CoyaThesis,LynchThesis}, but many interesting open questions remain.

\subsection{Black-boxing}
\label{subsec:black-boxing}

Next let us take an open dynamical system and extract from it the relations between externally observable quantities that hold whenever the system is in a `steady state': a state where nothing is changing.   The process of extracting this relation is an example of what we call `black-boxing', since it discards information that cannot be seen at the interfaces.    We shall see that it defines a double functor 
\[        \blacksquare \maps D\lCsp \to \lRel \]
from the double category of open dynamical systems to the double category of relations.    The functoriality of black-boxing implies that we can compose two open dynamical systems and then black-box them, or black-box each one and compose the resulting relations: either way, the final answer is the same.   

Consider an open dynamical system $F \maps X \slashedrightarrow Y$ given by the decorated cospan
\[
\scalebox{1.2}{
\begin{tikzpicture}
\node (A) at (0,0) {$X$};
\node (B) at (1,1) {$S$};
\node (C) at (2,0) {$Y$.};
\node (D) at (3,1) {$v \in D(S)$};
\path[->,font=\scriptsize,>=angle 90]
(A) edge node[above,pos = 0.3]{$i$} (B)
(C) edge node[above, pos = 0.3]{$o$} (B);
\end{tikzpicture}
}
\]
If we take its open dynamical system equation, Equation \eqref{eq:open_dynamical_system}, and fix $x$, $I$ and $O$ to be constant in time, we can treat them as vectors $x \in \R^S$, $I \in \R^X$, $O \in \R^Y$.  Then the equation reduces to
\[     v(x) + i_*(I) - o_*(O) = 0 .\]   
Thus, we define a \define{steady state} with inflows $I \in \R^X$ and outflows $O \in \R^Y$ to be a vector $x \in \R^S$ for which the above equation holds.  We define the \define{black-boxing} of our open dynamical system to be the set
\[   \blacksquare(F) \subseteq \R^X \times \R^X \times \R^Y \times \R^Y \]
consisting of all 4-tuples $(i^*(x),I,o^*(x),O)$ where $x \in \R^S$ is a steady state
with inflows $I \in \R^X$ and outflows $O \in \R^Y$:
\begin{equation}
\label{eq:blackboxing}
   \blacksquare(F) = \left\{  (i^*(x),I,o^*(x),O) \, \big\vert \, x \in \R^S, \; I \in \R^X, \; O \in \R^Y, \;   
   v(x) + i_*(I) - o_*(O) = 0\, \right\}.
\end{equation}
 The idea is that black-boxing an open dynamical system records its `externally observable steady state behavior'.
 
In the example from Section \ref{subsec:classical_mechanics}, with two particles connected by a spring, we get
 \[   \blacksquare(F) = \left\{\Big((q_1,0), (0, k(q_1 - q_2), (q_2,0), (0, k(q_1 - q_2)\Big) \, \big\vert \, q_1, q_2 \in \R \right\} .\]
 This says that in steady state, the particles can be at rest in any position, with equal and opposite forces acting on them, chosen precisely to counteract the force of the spring.   Of course this example is somewhat degenerate because every variable for the system is `exposed', that is, also a variable for an interface.  A more typical example would involve three particles connected by springs, with the first particle being the left interface and the second being the right interface:
 
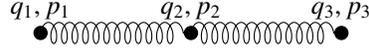
\[
\begin{tikzpicture}
%\draw[\dwinvisiblebordercolor, opacity = \dwinvisibleborderopacity, line width = \dwinvisibleborderwidth] (-2, -1.5) rectangle (2,.45);%.3 space
%
%Rocks
\coordinate(r1) at (-2,0);
\coordinate(r2) at (0,0);
\coordinate(r3) at (2,0);
% Arrow Tails
{\draw[decorate,decoration={coil,segment length=4pt, amplitude = 4pt},rotate=0] (r1)  -- (r2);}
{\draw[decorate,decoration={coil,segment length=4pt, amplitude = 4pt},rotate=0] (r2)  -- (r3);}
% Rocks
{\draw[fill= black] (r1) circle (2.5pt);}
{\draw[fill= black] (r2) circle (2.5pt);}
{\draw[fill= black] (r3) circle (2.5pt);}
\node at (-2,0.3) {\normalsize $q_1,p_1$};
\node at (0,0.3) {\normalsize $q_2,p_2$};
\node at (2,0.3) {\normalsize $q_3,p_3$};
\end{tikzpicture}
\]

\begin{exercise}
Black-box the composite system of Exercise \ref{ex:three_rocks_1}, using the open dynamical system equation obtained in Exercise \ref{ex:three_rocks_2}.
\end{exercise}

Now let us return to the general theory.    Given an open dynamical system $F$ with left interface $X$ and right interface $Y$, we can think of $\blacksquare(F)$ as a relation from $\R^X \times \R^X$ to $\R^Y \times \R^Y$.  There is a double category $\lRel$ where:
\begin{itemize}
\item an object is a set,
\item a tight morphism from $X$ to $Y$ is a function $f \maps X \to Y$,
\item a loose morphism from $X$ to $Y$ is a relation $R \maps X \slashedrightarrow Y$, 
meaning a subset $R \subseteq X \times Y$,
\item a 2-cell is a square
	\[ \scalebox{1.2}{
	\begin{tikzpicture}[scale=1.2]
	\node (D) at (0,0) {$X_1$};
	\node (E) at (1,0) {$Y_1$};
	\node (F) at (0,-1) {$X_2$};
	\node (A) at (1,-1) {$Y_2$};
	\path[->,font=\scriptsize,>=angle 90]
	(D) edge [tick] node [above]{$R_1$}(E)
	(E) edge node [right]{$g$}(A)
	(D) edge node [left]{$f$}(F)
	(F) edge [tick] node [below]{$R_2$} (A);
	\end{tikzpicture}
	}
	\]
obeying $(f \times g)R_1 \subseteq R_2$,
\item composition of tight morphisms is composition of functions,
\item composition of loose morphisms is the usual composition of relations,
\item vertical and horizontal composition of 2-cells is done in the only way possible.
\end{itemize}
The last item deserves some explanation.   We say a double category is \define{degenerate} if given any \define{frame}---that is, any collection of objects, tight morphisms and loose morphisms as follows: 
\[
\scalebox{1.2}{
\begin{tikzpicture}[scale=1.2]
\node (D) at (-4,0) {$X_1$};
\node (E) at (-3,0) {$Y_1$};
\node (F) at (-4,-1) {$X_2$};
\node (A) at (-3,-1) {$Y_2$};
\node (B) at (-3,-0.25) {};
\path[->,font=\scriptsize,>=angle 90]
(D) edge [tick] node [above]{$M$}(E)
(E) edge node [right]{$g$}(A)
(D) edge node [left]{$f$}(F)
(F) edge [tick] node [below]{$N$} (A);
\end{tikzpicture}
}
\]
there exists at most one 2-cell 
\[
\scalebox{1.2}{
\begin{tikzpicture}[scale=1.2]
\node (D) at (-4,0) {$X_1$};
\node (E) at (-3,0) {$Y_1$};
\node (F) at (-4,-1) {$X_2$};
\node (A) at (-3,-1) {$Y_2$};
\node (B) at (-3.5,-0.5) {$\Downarrow \alpha$};
\path[->,font=\scriptsize,>=angle 90]
(D) edge [tick] node [above]{$M$}(E)
(E) edge node [right]{$g$}(A)
(D) edge node [left]{$f$}(F)
(F) edge [tick] node [below]{$N$} (A);
\end{tikzpicture}
}
\]
filling this frame.  In a degenerate double category, like $\lRel$, there is no choice in how to define the vertical or horizontal composite of 2-cells once composition of loose and tight morphisms is fixed.  Thus we do not need to check the laws governing composition of 2-cells: we need merely check that the composites exist, which is easy to do in the case of $\lRel$.

In this language, black-boxing maps any loose morphism in $D\lCsp$, namely an open dynamical system $F \maps X \slashedrightarrow Y$, to a loose morphism in $\lRel$, namely
\[   \blacksquare(F) \maps \R^X \times \R^X \slashedrightarrow \R^Y \times \R^Y  .\]
This immediately leads to the question of whether black-boxing extends to a functor from $D\lCsp$ to $\lRel$.   We shall see it does.

Furthermore, the double category $\lRel$ is symmetric monoidal in a natural way, where the tensor product of objects and tight morphisms is given by the cartesian product in $\Set$, while the tensor product of loose morphisms (i.e.\ relations) is given by the cartesian product of their underlying subsets.    We can now state the main theorem about black-boxing:

\begin{thm}
\label{thm:black}
There is a symmetric monoidal double functor $\blacksquare \maps D\lCsp \to \lRel$
sending
\begin{itemize}
\item any finite set $X$ to the vector space $\R^X \times \R^X$,
\item any function $f \maps X \to Y$ to the pushforward $f_\ast \maps \R^X \times \R^X \to \R^Y \times \R^Y$,
\item any open dynamical system $F \maps X \slashedrightarrow Y$ to its black-boxing $\blacksquare(F)$ as defined in Equation \eqref{eq:blackboxing},
\item any 2-cell to the only 2-cell with the appropriate frame.
\end{itemize}
\end{thm}

\begin{proof}
A closely related theorem was proved at the level of categories in \cite[Thm.\ 23]{BP}.
It was shown that isomorphic open dynamical systems $F \maps X \slashedrightarrow Y$ have the 
same black-boxing $\blacksquare(F)$, and there is a symmetric monoidal functor from $D\Csp$ to $\Rel$ sending any isomorphism class of open dynamical systems $F \maps X \slashedrightarrow Y$ to $\blacksquare(F)$.   (In fact it was shown that this functor maps $F$ to a relation of a special kind, called
a `semialgebraic' relation, but this is irrelevant here.)   At the level of double categories, this theorem implies that $\blacksquare$ preserves composition of horizontal morphisms on the nose.  Clearly it preserves composition of vertical morphisms.   Since $\lRel$ is a degenerate double category there is at most one choice of where $\blacksquare$ can send any 2-cell, and the symmetric monoidality of the double functor $\blacksquare$ follows easily from the corresponding result at the level of categories.
\end{proof}

The functoriality is the most interesting step in proving the above theorem.  One needs to check that when two open dynamical systems are composed, one can compose their steady states to form a steady state of the composite system when the outflow of the first steady state equals the inflow of the second.

\subsection{Structured versus decorated cospans, revisited}
\label{subsec:structured_versus_decorated_cospans_revisited}

Before moving on, let us briefly return to a technical issue raised in Section \ref{subsec:structured_versus_decorated_cospans}.   We used decorated cospans to define a double category of open dynamical systems, $D\lCsp$.  Can we define an equivalent double category using structured cospans?

It seems not, but at present the main evidence comes from \cite[Thm.\ 4.1]{BCV}.  This result says we \emph{could} define such a structured cospan double category if the natural lift of $D \maps \Fin\Set \to \Cat$ to a pseudofunctor $D \maps \Fin\Set \to \SMC$ factored through $\Rex$.   It also says that in this case, the opfibration $U \maps \inta D \to \Fin\Set$ would be a right adjoint.   However, $U$ is not a right adjoint.   After all, this would imply that for every $S \in \Fin\Set$ the comma category $S \downarrow U$ has an initial object.
But this is not true.   Because the empty set is initial in $\Fin\Set$, the comma category
$\emptyset \downarrow U$ is just $\inta D$.  This contains an object $(\emptyset, v_\emptyset)$, where $v_\emptyset$ is the only possible vector field on $\R^\emptyset$, namely, the zero vector field.   The only object in $\inta D$ with any morphisms to $(\emptyset, v_\emptyset)$ is $(\emptyset, v_\emptyset)$ itself, so no other object can be initial.  However $(\emptyset, v_\emptyset)$ is not initial either, because it has no morphisms to an object $(S,v)$ unless $v$ is the zero vector field on $\R^S$. 

While this argument is impressively technical, its conclusions are quite weak.   It does not imply that no structured cospan double category is equivalent to $D\lCsp$.   It merely says that one particular strategy for constructing such a structured cospan double category fails.  

\begin{problem}
Find conditions implying that a decorated cospan double category is \emph{not} equivalent to any structured cospan double category.
\end{problem}

Luckily, there is a practical workaround in this case: a structured cospan category of `linearly parametrized' dynamical systems \cite[Sec.\ 3.1]{AFKOPS}.   Here instead of equipping a finite set $S$ with a fixed vector field $v \in D(S)$, we equip it with a parametrized family of vector fields---or more precisely a finite set $P$ and a linear map $v \maps \R^P \to D(S)$.    We call $S$ the set of \define{variables} and $P$ the set of \define{parameters}.  This captures the fact that in physics and other subjects, the vector fields describing time evolution often depend linearly on some $P$-tuple of real numbers called `parameters' or `coupling constants'.

There is a category $\ParaDynam$ for which:
\begin{itemize}
\item an object is a  \define{linearly parametrized dynamical system}: a pair of finite sets $S$ and $P$ and a linear map $v \maps \R^P \to D(S)$;
\item a morphism from $(S,P,v)$ to $(S',P',v')$ is a pair of functions $f \maps S \to S', q \maps P \to P'$ making this square commute:
\[
\scalebox{1.2}{
\begin{tikzpicture}[scale=1.2]
\node (A) at (0,0) {$\R^P$};
\node (A') at (0,-1) {$\R^{P'}$};
\node (B) at (1,0) {$D(S)$};
\node (B') at (1,-1) {$D(S')$};
\path[->,font=\scriptsize,>=angle 90]
(A) edge node[above]{$v$} (B)
(A') edge node[below]{$v'$} (B')
(A) edge node[left]{$q_\ast$} (A')
(B) edge node[right]{$f_\ast \circ - \circ f^\ast$} (B');
\end{tikzpicture}
}
\]
\end{itemize}

The category $\ParaDynam$ has finite colimits \cite[Prop.\ 3.5]{AFKOPS}, and the forgetful functor
$U \maps \ParaDynam \to \Fin\Set$ sending any linearly parametrized dynamical system to its set of variables has a left adjoint $L$ sending any finite set to the linearly parametrized dynamical system with the empty set of parameters.  Thus, we can construct a structured cospan double category of open linearly parametrized dynamical systems and show that it is symmetric monoidal \cite[Prop.\ 3.9]{AFKOPS}.

\section{Petri nets with rates}
\label{sec:Petri_with_rates}

Even if we fix a class of dynamical systems, there are typically many choices of syntax for describing these systems.  This is certainly true of the dynamical systems we have just discussed: systems of first-order ordinary differential equations.   We explain one syntax here, namely `Petri nets with rates'.  These are used in chemistry, population biology, epidemiology and other fields \cite{BFHLP,CTF,Haas,Koch,Wilkinson}.   Here we describe a double category of \emph{open} Petri nets with rates, and a double functor from this to our double category of open dynamical systems, $D\lCsp$.   The easiest way to understand all this starts with chemistry.  

\subsection{Chemistry}
\label{subsec:chemistry}

Chemists sometimes use Petri nets where the places represent `chemical species': different kinds of molecules, elements and ions.   Then transitions represent chemical reactions.  For example, we can describe the formation of water from hydrogen and oxygen molecules
\[         2 \text{H}_2 + \text{O}_2 \longrightarrow 2 \text{H}_2\text{O}  \]
using this Petri net $P$:
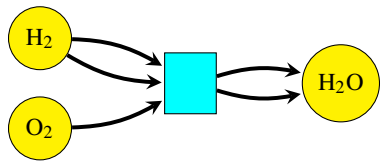
\[
\begin{tikzpicture}
	\begin{pgfonlayer}{nodelayer}
		\node [style=species] (H) at (0,0.6) {$\text{H}_2$};
		\node [style=species] (O) at (0,-0.6) {$\text{O}_2$};
		\node [style=transition] (alpha) at (2,0) {$\phantom{\Big|r_1\Big|}$};
		\node [style=species] (water) at (4,0) {$\text{H}_2\text{O}$};
%		\node [style=species] (peroxide) at (4,-1.4) {$\text{H}_2\text{O}_2$};
%		\node [style=transition] (beta) at (2,-1.4) {$\phantom{\Big|}r_2$\phantom{\Big|}};
%		\node [style=transition] (Something) at (6,0) {\tiny{Something}};
%		\node [style=species] (A) at (8,1) {O$\text{H}^{-}$};
%		\node [style=species] (B) at (8,-1) {$\text{H}_3 \text{O}^{+}$};
	\end{pgfonlayer}
	\begin{pgfonlayer}{edgelayer}
		\draw [style=inarrow, bend right=15, looseness=1.00] (H) to (alpha);
		\draw [style=inarrow, bend left=15, looseness=1.00] (H) to (alpha);
		\draw [style=inarrow, bend right=15, looseness=1.00] (O) to (alpha);
		\draw [style=inarrow,  bend left=15, looseness=1.00] (alpha) to (water);
		\draw [style=inarrow,  bend right=15, looseness=1.00] (alpha) to (water);
%		\draw [style=inarrow,  bend left=15, looseness=1.00] (peroxide) to (beta);
%		\draw [style=inarrow,  bend right=15, looseness=1.00] (peroxide) to (beta);
%		\draw [style=inarrow] (beta) to (O);
%		\draw [style=inarrow, bend right=10, looseness=1.00] (beta) to (water);
%		\draw [style=inarrow, bend left=10, looseness=1.00] (beta) to (water);
%		\draw [style=inarrow, bend left=40, looseness=1.00] (Water) to (Something);
%		\draw [style=inarrow, bend right=40, looseness=1.00] (Water) to (Something);
%		\draw [style=inarrow, bend left=40, looseness=1.00] (Something) to (A);
%		\draw [style=inarrow, bend right=40, looseness=1.00] (Something) to (B);
	\end{pgfonlayer}
\end{tikzpicture}
\]
The transition in aqua here describes a reaction where two molecules of hydrogen and one of oxygen become two molecules of water.    We could use the token semantics described in Section \ref{sec:Petri} to model how individual molecules react in this way.   For example, there is a morphism in the category $FP$ from the marking
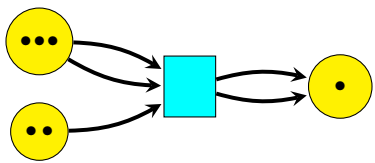
\[
\begin{tikzpicture}
	\begin{pgfonlayer}{nodelayer}
		\node [style=species] (H) at (0,0.6) {$\bullet\!\bullet\!\bullet$};
		\node [style=species] (O) at (0,-0.6) {$\bullet\,\bullet$};
		\node [style=transition] (alpha) at (2,0) {$\phantom{\Big|r_1\Big|}$};
		\node [style=species] (water) at (4,0) {$\phantom{.}\bullet\phantom{.}$};
%		\node [style=species] (peroxide) at (4,-1.4) {$\text{H}_2\text{O}_2$};
%		\node [style=transition] (beta) at (2,-1.4) {$\phantom{\Big|}r_2$\phantom{\Big|}};
%		\node [style=transition] (Something) at (6,0) {\tiny{Something}};
%		\node [style=species] (A) at (8,1) {O$\text{H}^{-}$};
%		\node [style=species] (B) at (8,-1) {$\text{H}_3 \text{O}^{+}$};
	\end{pgfonlayer}
	\begin{pgfonlayer}{edgelayer}
		\draw [style=inarrow, bend right=15, looseness=1.00] (H) to (alpha);
		\draw [style=inarrow, bend left=15, looseness=1.00] (H) to (alpha);
		\draw [style=inarrow, bend right=15, looseness=1.00] (O) to (alpha);
		\draw [style=inarrow,  bend left=15, looseness=1.00] (alpha) to (water);
		\draw [style=inarrow,  bend right=15, looseness=1.00] (alpha) to (water);
%		\draw [style=inarrow,  bend left=15, looseness=1.00] (peroxide) to (beta);
%		\draw [style=inarrow,  bend right=15, looseness=1.00] (peroxide) to (beta);
%		\draw [style=inarrow] (beta) to (O);
%		\draw [style=inarrow, bend right=10, looseness=1.00] (beta) to (water);
%		\draw [style=inarrow, bend left=10, looseness=1.00] (beta) to (water);
%		\draw [style=inarrow, bend left=40, looseness=1.00] (Water) to (Something);
%		\draw [style=inarrow, bend right=40, looseness=1.00] (Water) to (Something);
%		\draw [style=inarrow, bend left=40, looseness=1.00] (Something) to (A);
%		\draw [style=inarrow, bend right=40, looseness=1.00] (Something) to (B);
	\end{pgfonlayer}
\end{tikzpicture}
\]
to the marking
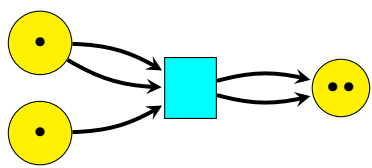
\[
\begin{tikzpicture}
	\begin{pgfonlayer}{nodelayer}
		\node [style=species] (H) at (0,0.6) {$\phantom{.}\bullet\phantom{.}$};
		\node [style=species] (O) at (0,-0.6) {$\phantom{.}\bullet\phantom{.}$};
		\node [style=transition] (alpha) at (2,0) {$\phantom{\Big|r_1\Big|}$};
		\node [style=species] (water) at (4,0) {$\bullet \, \bullet$};
%		\node [style=species] (peroxide) at (4,-1.4) {$\text{H}_2\text{O}_2$};
%		\node [style=transition] (beta) at (2,-1.4) {$\phantom{\Big|}r_2$\phantom{\Big|}};
%		\node [style=transition] (Something) at (6,0) {\tiny{Something}};
%		\node [style=species] (A) at (8,1) {O$\text{H}^{-}$};
%		\node [style=species] (B) at (8,-1) {$\text{H}_3 \text{O}^{+}$};
	\end{pgfonlayer}
	\begin{pgfonlayer}{edgelayer}
		\draw [style=inarrow, bend right=15, looseness=1.00] (H) to (alpha);
		\draw [style=inarrow, bend left=15, looseness=1.00] (H) to (alpha);
		\draw [style=inarrow, bend right=15, looseness=1.00] (O) to (alpha);
		\draw [style=inarrow,  bend left=15, looseness=1.00] (alpha) to (water);
		\draw [style=inarrow,  bend right=15, looseness=1.00] (alpha) to (water);
%		\draw [style=inarrow,  bend left=15, looseness=1.00] (peroxide) to (beta);
%		\draw [style=inarrow,  bend right=15, looseness=1.00] (peroxide) to (beta);
%		\draw [style=inarrow] (beta) to (O);
%		\draw [style=inarrow, bend right=10, looseness=1.00] (beta) to (water);
%		\draw [style=inarrow, bend left=10, looseness=1.00] (beta) to (water);
%		\draw [style=inarrow, bend left=40, looseness=1.00] (Water) to (Something);
%		\draw [style=inarrow, bend right=40, looseness=1.00] (Water) to (Something);
%		\draw [style=inarrow, bend left=40, looseness=1.00] (Something) to (A);
%		\draw [style=inarrow, bend right=40, looseness=1.00] (Something) to (B);
	\end{pgfonlayer}
\end{tikzpicture}
\]
However, chemists often deal with vast numbers of molecules in solution.   Instead of counting molecules individually they count them in `moles': a mole is about $6 \cdot 10^{23}$ molecules.   Then the `concentration' of a given species, measured in moles per liter, is treated approximately as a smoothly varying real-valued function of time, and we want differential equations describing how these concentrations change.   For this chemists commonly use a recipe called the `law of mass action'.

It is easiest to explain this with an example.  For the Petri net above, the law of mass action gives these differential equations:
\[
  \begin{array}{lcl} 
\displaystyle{\frac{d}{dt} [\text{H}_2]} &=& -2r_1 \,  [\text{H}_2]^2 \, [\text{O}_2]  \\ \\
\displaystyle{\frac{d}{dt} [\text{O}_2]}  &=&  -r_2 \,  [\text{H}_2]^2 \, [\text{O}_2]  \\ \\
\displaystyle{\frac{d}{dt} [\text{H}_2\text{O}]}  &=&  2r_1 \,  [\text{H}_2]^2 \, [\text{O}_2] .
\end{array}
\]
Here $[\text{H}_2] \maps \R \to \R$ is the concentration of hydrogen molecules as a function of time,
and similarly for $[\text{O}_2]$ and $[\text{H}_2\text{O}]$.   The constant $r_1 \in [0,\infty)$ is the `rate constant' of this particular reaction, which depends on various environmental conditions.   All the time derivatives are proportional to $[\text{H}_2]^2 \, [\text{O}_2]$, because this chemical reaction takes two hydrogen molecules and one oxygen molecule as inputs.   We should imagine molecules randomly moving around; then the probability that two $\text{H}_2$ molecules and one $\text{O}_2$ molecules are in the same very small region of space is approximately proportional to  $[\text{H}_2]^2 \, [\text{O}_2]$.    The time derivative of $[\text{H}_2]$ is also proportional to $-2$, because $2$ hydrogen molecules get destroyed in this reaction.   Similarly, $\frac{d}{dt} [\text{O}_2]$ is proportional to $-1$ because $1$ oxygen molecule gets destroyed, and $\frac{d}{dt}[\text{H}_2\text{O}]$ is proportional to $1$ because one water molecule gets created.

Generalizing from this example, we can state the law of mass action precisely as follows.   We start with a \define{Petri net with rates}, meaning a Petri net $s,t \maps T \to \N[S]$ together with a function $r \maps T \to [0,\infty)$ assigning to each transition $\tau \in T$ a nonnegative real number $r(\tau)$ called its \define{rate constant}.    In chemistry the elements of $S$ are called \define{species} rather than places.   

From our Petri net with rates, we get a differential equation 
\[    \frac{d}{dt} x(t) = v(x(t)) \]
called the \define{rate equation} describing the evolution of a function $x \maps \R \to \R^S$.  This function specifies the concentration of each species as a function of time.   The vector field $v$ on $\R^S$ is given by \define{law of mass action}:
\begin{equation}
\label{eq:law_of_mass_action}
v(x) = \sum_{\tau \in T} r(\tau) \, ( t(\tau) - s(\tau) ) x^{s(\tau)} 
\end{equation}
where $x \in \R^S$.   The notation here requires some explanation.  We think of $t(\tau), s(\tau) \in \N[S]$ as vectors in $\R^S$.  Their difference $t(\tau) - s(\tau)$ says, for each species, the number of tokens of that species created by the transition $\tau$ minus the number destroyed.    Finally, we define
\[     x^{s(\tau)} = \prod_{i \in S} {x_i}^{s(\tau)_i}. \]
Here for any $i \in S$ we write $x_i \in \R$ for the $i$th component of $x \in \R^S$, i.e., the concentration of the $i$th species.  Similarly, we write $s(\tau)_i$ for the $i$th component of $s(\tau) \in \R^S$.

\begin{exercise}   
Here is a Petri net with rates:
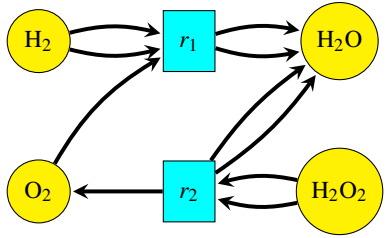
\[
\begin{tikzpicture}
	\begin{pgfonlayer}{nodelayer}
		\node [style=species] (H) at (0,0.6) {$\mathrm{H}_2$};
		\node [style=species] (O) at (0,-1.4) {$\mathrm{O}_2$};
		\node [style=transition] (alpha) at (2,0.6) {$\phantom{\Big|}r_1$\phantom{\Big|}};
		\node [style=species] (water) at (4,0.6) {$\mathrm{H}_2\mathrm{O}$};
		\node [style=species] (peroxide) at (4,-1.4) {$\mathrm{H}_2\mathrm{O}_2$};
		\node [style=transition] (beta) at (2,-1.4) {$\phantom{\Big|}r_2$\phantom{\Big|}};
%		\node [style=transition] (Something) at (6,0) {\tiny{Something}};
%		\node [style=species] (A) at (8,1) {O$\mathrm{H}^{-}$};
%		\node [style=species] (B) at (8,-1) {$\mathrm{H}_3 \mathrm{O}^{+}$};
	\end{pgfonlayer}
	\begin{pgfonlayer}{edgelayer}
		\draw [style=inarrow, bend right=15, looseness=1.00] (H) to (alpha);
		\draw [style=inarrow, bend left=15, looseness=1.00] (H) to (alpha);
		\draw [style=inarrow, bend left=15, looseness=1.00] (O) to (alpha);
		\draw [style=inarrow,  bend left=15, looseness=1.00] (alpha) to (water);
		\draw [style=inarrow,  bend right=15, looseness=1.00] (alpha) to (water);
		\draw [style=inarrow,  bend left=15, looseness=1.00] (peroxide) to (beta);
		\draw [style=inarrow,  bend right=15, looseness=1.00] (peroxide) to (beta);
		\draw [style=inarrow] (beta) to (O);
		\draw [style=inarrow, bend right=10, looseness=1.00] (beta) to (water);
		\draw [style=inarrow, bend left=10, looseness=1.00] (beta) to (water);
%		\draw [style=inarrow, bend left=40, looseness=1.00] (Water) to (Something);
%		\draw [style=inarrow, bend right=40, looseness=1.00] (Water) to (Something);
%		\draw [style=inarrow, bend left=40, looseness=1.00] (Something) to (A);
%		\draw [style=inarrow, bend right=40, looseness=1.00] (Something) to (B);
	\end{pgfonlayer}
\end{tikzpicture}
\]
Transitions are labeled with their rate constants. The transition with rate constant $r_1$ represents the oxidation of hydrogen to form water as before, while the transition with rate constant $r_2$  represents the decomposition of hydrogen peroxide.   Use the law of mass action to derive the following differential equations:
\[
  \begin{array}{lcl} \displaystyle{\frac{d}{dt} [\mathrm{H}_2]} &=& -2r_1 \,  [\mathrm{H}_2]^2 \, [\mathrm{O}_2]  \\ \\
\displaystyle{\frac{d}{dt} [\mathrm{O}_2]}  &=& -r_1 \,  [\mathrm{H}_2]^2 \, [\mathrm{O}_2] + r_2\,  [\mathrm{H}_2 \mathrm{O}_2]^2  \\  \\
\displaystyle{\frac{d}{dt}[\mathrm{H}_2\mathrm{O}_2]}  &=&  -2r_2 \, [\mathrm{H}_2 \mathrm{O}_2]^2 \\ \\
\displaystyle{\frac{d}{dt}[\mathrm{H}_2\mathrm{O}]} &=& 2r_1 \,  [\mathrm{H}_2]^2 \, [\mathrm{O}_2]  + 2r_2 \, [\mathrm{H}_2 \mathrm{O}_2]^2.
\end{array}
\]
\end{exercise}

\subsection{Open Petri nets with rates}

 %Mathematical chemists have proved deep theorems relating the topology of Petri nets with rates to the qualitative behavior of their dynamical systems \cite{CTF}.   More recently the subject has broadened to include \emph{open} Petri nets with rates \cite{BP}, though few results are known about these.
 
 %Recently Halter and Patterson \cite{HP} implemented this idea in a software tool that makes it easy to build epidemiological models using structured cospans.  They used this to rebuild part of the COVID-19 model that the UK has been using to make policy decisions.  The advantage of using structured or decorated cospans is that one can build complex models from smaller pieces, so one can easily add or change parts, and better understand the effects of doing this.
 
Now we consider \emph{open} Petri nets with rates, and explain a semantics mapping them to open dynamical systems.   The first step is to define a category of Petri nets with rates.  Then we use decorated cospans to construct a double category of open Petri nets with rates.  Then we use a general recipe for constructing maps between decorated cospan categories.   Readers uninterested in Petri nets with rates may still be interested in this general recipe, Theorem \ref{thm:decorated_cospan_functoriality}, since it is the decorated cospan analogue of Theorems \ref{thm:structured_cospan_functoriality_1} and \ref{thm:structured_cospan_functoriality_2} for structured cospans.
 
 The first step is simple enough.   There is a category whose objects are Petri nets with rates, where a morphism from
\[   \xymatrix{ [0,\infty) & T \ar[l]_-r \ar@<-.5ex>[r]_-t \ar@<.5ex>[r]^-s & \N[S] }\]
 to
 \[   \xymatrix{ [0,\infty) & T' \ar[l]_-{r'} \ar@<-.5ex>[r]_-{t'} \ar@<.5ex>[r]^-{s'} & \N[S'] }\]
is a morphism of the underlying Petri nets whose map $g \maps T \to T'$ obeys
\[    r'(\tau') = \sum_{\{\tau \in T: g(\tau) = \tau'\}} r(\tau) \]
for all $\tau' \in T'$.   For example, if $S = \{A,B\}$ there is a morphism in $F(S)$ from this Petri net with rates:
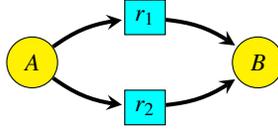
\[
\begin{tikzpicture}
	\begin{pgfonlayer}{nodelayer}
\node [style=species] (A) at (0, 0) {$A$};
\node [style=species] (B) at (3, 0) {$B$};
	   \node [style=transition] (tau1) at (1.5, 0.6) {$r_1$};
             \node [style=transition] (tau2) at (1.5,-0.6) {$r_2$};
	\end{pgfonlayer}
	\begin{pgfonlayer}{edgelayer}
		\draw [style=inarrow, bend left = 15, looseness=1] (A) to (tau1);
		\draw [style=inarrow, bend right = 15, looseness=1] (A) to (tau2);
	%	\draw [style=inarrow, bend right=15, looseness=1.00] (tau1) to (C);
		\draw [style=inarrow, bend left=15, looseness=1.00] (tau1) to (B);
            \draw [style=inarrow, bend right =15, looseness=1.00] (tau2) to (B); 
	\end{pgfonlayer}
\end{tikzpicture}
\]
to this one:
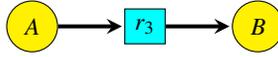
\[
\begin{tikzpicture}
	\begin{pgfonlayer}{nodelayer}
\node [style=species] (A) at (0, 0) {$A$};
\node [style=species] (B) at (3, 0) {$B$};
	   \node [style=transition] (tau1) at (1.5, 0) {$r_3$};
 %            \node [style=transition] (tau2) at (1.5,-0.6) {$r_2$};
	\end{pgfonlayer}
	\begin{pgfonlayer}{edgelayer}
		\draw [style=inarrow] (A) to (tau1);
	%	\draw [style=inarrow, bend right = 15, looseness=1] (A) to (tau2);
	%	\draw [style=inarrow, bend right=15, looseness=1.00] (tau1) to (C);
		\draw [style=inarrow] (tau1) to (B);
          %  \draw [style=inarrow, bend right =15, looseness=1.00] (tau2) to (B); 
	\end{pgfonlayer}
\end{tikzpicture}
\]
if and only if $r_3 = r_1 + r_2$.
   
Next, to construct a double category of open Petri nets with rates, we note that there is a symmetric lax monoidal pseudofunctor 
\[   F \maps (\Fin\Set, +) \to (\Cat, \times) \]
 such that for any finite set $S$:
\begin{itemize}
\item an object of $F(S)$ is a Petri nets with rates whose set of places is $S$,
\item a morphism is a morphism of Petri nets with rates that is the identity on the set of places.
\end{itemize}

By Theorem \ref{thm:decorated_cospan_2}, $F$ gives a symmetric monoidal double category $F \lCsp$, which we call the double category of \define{open Petri nets with rates}.  In this double category
\begin{itemize}
\item an object is a finite set,
\item a tight morphism is a function,
\item a loose morphism from a finite set $X$ to a finite set $Y$ is an \define{open Petri net with rates},
meaning a cospan of finite sets 
\[
\scalebox{1.2}{
\begin{tikzpicture}[scale=1.0]
\node (A) at (0,0) {$X$};
\node (B) at (1,0) {$S$};
\node (C) at (2,0) {$Y$};
\path[->,font=\scriptsize,>=angle 90]
(A) edge node[above,pos=0.3]{$i$} (B)
(C)edge node[above,pos=0.3]{$o$}(B);
\end{tikzpicture}
}
\]
decorated with a Petri net with rates
\[   \xymatrix{ [0,\infty) & T \ar[l]_-r \ar@<-.5ex>[r]_-t \ar@<.5ex>[r]^-s & \N[S], }\]
\item a 2-cell is a map of cospans of finite sets
\[
\scalebox{1.2}{
\begin{tikzpicture}[scale=1]
\node (E) at (3,0) {$X$};
\node (F) at (5,0) {$Y$};
\node (G) at (4,0) {$S$};
\node (E') at (3,-1) {$X'$};
\node (F') at (5,-1) {$Y'$.};
\node (G') at (4,-1) {$S'$};
\path[->,font=\scriptsize,>=angle 90]
(F) edge node[above]{$o$} (G)
(E) edge node[left]{$f$} (E')
(F) edge node[right]{$g$} (F')
(G) edge node[left]{$\alpha$} (G')
(E) edge node[above]{$i$} (G)
(E') edge node[below]{$i'$} (G')
(F') edge node[below]{$o'$} (G');
\end{tikzpicture}
}
\]
together with a morphism of decorations, namely a morphism of Petri nets with rates from
\[   \xymatrix{ [0,\infty) & T \ar[l]_-r \ar@<-.5ex>[r]_-t \ar@<.5ex>[r]^-s & \N[S] }\]
to
\[   \xymatrix{ [0,\infty) &{T'} \ar[l]_-{r'} \ar@<-.5ex>[r]_-{t'} \ar@<.5ex>[r]^-{s'} & \N[S'] }\]
where the map from $S$ to $S'$ is $\alpha$.
\end{itemize}

Now we are ready to define a semantics mapping open Petri nets to open dynamical systems.   Recall that in Section \ref{sec:dynamical} we used decorated cospans to construct a symmetric monoidal double category $D\lCsp$ of open dynamical systems.   This was built using a symmetric lax monoidal pseudofunctor
\[   D \maps (\Fin\Set, +) \to (\Cat, \times) \]
that maps any finite set $S$ to the discrete category on the set
\[  \left\{ v \maps \R^S \to \R^S | \; v \textrm{ is smooth} \right\}. \]
Our desired semantics should be a symmetric monoidal double functor
\[      \graysquare \maps F \lCsp \to D \lCsp \]
sending any open Petri net with rates to an open dynamical system.  This was already defined at the level of categories in \cite[Sec.\ 7]{BP}, where it was called `gray-boxing', since it obscures some but not all of the details of an open Petri net.   To boost this result to the double category level, we use the following general recipe for defining maps between decorated cospan double categories.  

\begin{thm}
\label{thm:decorated_cospan_functoriality}
Given categories $\A$ and $\A'$ with finite colimits, lax monoidal pseudofunctors $F \maps (\A,+) \to (\Cat,\times)$ and $F' \maps (\A',+) \to
(\Cat,\times)$, a finite colimit preserving functor $H \maps \A \to \A'$, a lax monoidal pseudofunctor $E \maps (\Cat,\times) \to (\Cat,\times)$ and
a monoidal natural transformation $\theta$ as in the following diagram:
\[
\scalebox{1.2}{
\begin{tikzpicture}[scale=1.5]
\node (A) at (0,0) {$\A$};
\node (B) at (1,0) {$\Cat$};
\node (C) at (0,-1) {$\A'$};
\node (D) at (1,-1) {$\Cat$};
\node (E) at (0.5,-0.5) {$\scriptstyle\Downarrow\theta$};
\path[->,font=\scriptsize,>=angle 90]
(A) edge node[above]{$F$} (B)
(A) edge node[left]{$H$} (C)
(B) edge node[right]{$E$} (D)
(C) edge node[below]{$F'$} (D);
\end{tikzpicture}
}
\]
we obtain a double functor $\Theta \maps F\lCsp \to F'\lCsp$.  If $F, F'$ and $E$ are symmetric, then $\Theta \maps F\lCsp \to F'\lCsp$ is a symmetric monoidal double functor.
\end{thm}

\begin{proof}
This is \cite[Thm.\ 2.5]{BCV}, which provides the details of how $\Theta$ is defined.
\end{proof}

Let us take the square in this theorem to be
\[
\scalebox{1.2}{
\begin{tikzpicture}[scale=1.5]
\node (A) at (0,0) {$\Fin\Set$};
\node (B) at (1,0) {$\Cat$};
\node (C) at (0,-1) {$\Fin\Set$};
\node (D) at (1,-1) {$\Cat$.};
\node (E) at (0.5,-0.5) {$\Downarrow \theta$};
\path[->,font=\scriptsize,>=angle 90]
(A) edge node[above]{$F$} (B)
(A) edge node[left]{$1$} (C)
(B) edge node[right]{$1$} (D)
(C) edge node[below]{$D$} (D);
\end{tikzpicture}
}
\] 
Here $\theta$ is a monoidal natural transformation given as follows.  For any finite set $S$, $\theta_S \maps F(S) \to D(S)$ maps any Petri net with rates 
\[   \xymatrix{ [0,\infty) & T \ar[l]_-r \ar@<-.5ex>[r]_-t \ar@<.5ex>[r]^-s & \N[S] }\]
to a smooth vector field on $\R^S$, say $v$.   This vector field is defined using the law of mass action, as explained in Section \ref{subsec:chemistry}.   Namely, for any $x \in \R^S$, we set
\[  
v(x) = \sum_{\tau \in T} r(\tau) \, ( t(\tau) - s(\tau) ) x^{s(\tau)} 
\]
where 
\[     x^{s(\tau)} = \prod_{i \in S} {c_i}^{s(\tau)_i}  \]
and we think of $t(\tau), s(\tau) \in \N[S]$ as vectors in $\R^S$.   That $\theta$ is monoidal follows from \cite[Thm.\ 18]{BP}.   Thus, it defines a symmetric monoidal double functor $\graysquare \maps F \lCsp \to D \lCsp$.

One of the simplest consequences of this result is that we can get an open dynamical system equation from an open Petri net with rates, following the ideas in Section \ref{subsec:open_dynamical_system_equation}.  For example, consider this open Petri net with rates:
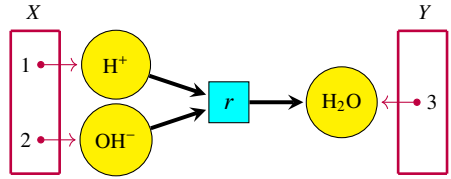
\[
\begin{tikzpicture}
	\begin{pgfonlayer}{nodelayer}
		\node [style=species] (A) at (-4, 0.5) {\scriptsize $\phantom{l}\text{H}^+\phantom{l}$};
		\node [style=species] (B) at (-4, -0.5) {\scriptsize $\text{OH}^-$};
		\node [style=species] (C) at (-1, 0) {\scriptsize $\text{H}_2\text{O}$};	
                 \node [style=transition] (a) at (-2.5, 0) {$\phantom{|}r\phantom{|}$}; 
		
		\node [style=empty] (X) at (-5.1, 1.2) {$X$};
		\node [style=none] (Xtr) at (-4.75, 0.95) {};
		\node [style=none] (Xbr) at (-4.75, -0.95) {};
		\node [style=none] (Xtl) at (-5.4, 0.95) {};
                 \node [style=none] (Xbl) at (-5.4, -0.95) {};
	
		\node [style=inputdot] (1) at (-5, 0.5) {};
		\node [style=empty] at (-5.2, 0.5) {$1$};
		\node [style=inputdot] (2) at (-5, -0.5) {};
		\node [style=empty] at (-5.2, -0.5) {$2$};

		\node [style=empty] (Y) at (0.1, 1.2) {$Y$};
		\node [style=none] (Ytr) at (.4, 0.95) {};
		\node [style=none] (Ytl) at (-.25, 0.95) {};
		\node [style=none] (Ybr) at (.4, -0.95) {};
		\node [style=none] (Ybl) at (-.25, -0.95) {};

		\node [style=inputdot] (3) at (0, 0) {};
		\node [style=empty] at (0.2, 0) {$3$};		
		
%		\node [style=empty] (Z) at (3, 1) {$Z$};
%		\node [style=none] (Ztr) at (3.25, 0.75) {};
%		\node [style=none] (Ztl) at (2.75, 0.75) {};
%		\node [style=none] (Zbl) at (2.75, -0.75) {};
%		\node [style=none] (Zbr) at (3.25, -0.75) {};
		
	\end{pgfonlayer}
	\begin{pgfonlayer}{edgelayer}
		\draw [style=inarrow] (A) to (a);
		\draw [style=inarrow] (B) to (a);
		\draw [style=inarrow] (a) to (C);
		\draw [style=inputarrow] (1) to (A);
		\draw [style=inputarrow] (2) to (B);
		\draw [style=inputarrow] (3) to (C);
		\draw [style=simple] (Xtl.center) to (Xtr.center);
		\draw [style=simple] (Xtr.center) to (Xbr.center);
		\draw [style=simple] (Xbr.center) to (Xbl.center);
		\draw [style=simple] (Xbl.center) to (Xtl.center);
		\draw [style=simple] (Ytl.center) to (Ytr.center);
		\draw [style=simple] (Ytr.center) to (Ybr.center);
		\draw [style=simple] (Ybr.center) to (Ybl.center);
		\draw [style=simple] (Ybl.center) to (Ytl.center);
	\end{pgfonlayer}
\end{tikzpicture}
\]
gives this open dynamical system equation
\[
\begin{array}{lcl} 
\displaystyle{ \frac{d}{dt} [\text{H}^+]}  &=& - r \, [\text{H}^+] \,  [\text{OH}^-] + I_1(t) \\ [8pt]  
\displaystyle{ \frac{d}{dt} [\text{OH}^-]}  &=& - r  \, [\text{H}^+] \,  [\text{OH}^-] + I_2(t) \\ [8pt]  
\displaystyle{ \frac{d}{dt} [\text{H}_2\text{O}]}  &=&  r \,  [\text{H}^+] \,  [\text{OH}^-] - O_3(t). \\ [8pt]  
\end{array}
\]

\subsection{Applications}
\label{sec:applications}

One might hope that open Petri nets with rates would help mathematical chemists prove new theorems.   So far, things have worked out quite differently.  Their main application so far has been to software for epidemiology!

Shortly after the formulation of decorated and structured cospans, the world was hit with a pandemic.   At this time Fairbanks and Patterson were developing Catlab \cite{Catlab}, a framework for applied and computational category theory, written in the Julia language.    Fairbanks had already implemented decorated cospans in Catlab---but in 2020, Patterson wrote the first version of structured cospans in this framework, specifically so that Fairbanks and Halter could implement open whole-grain Petri nets.    Soon the team applied these to recreate part of the UK's main COVID model \cite{PH}.    There is by now a well-developed software package to build and manipulate open Petri nets with rates in Catlab, called AlgebraicPetri \cite{AlgebraicPetri}.  This works together with a package for open dynamical systems developed by Libkind, called AlgebraicDynamics \cite{AlgebraicDynamics}.   AlgebraicDynamics is able to use either the variable sharing paradigm or the input-output paradigm.   

All this work attracted the attention of Osgood and Li, computational epidemiologists who helped run COVID modeling in Canada.   These experts pointed out that in public health modeling, `stock and flow diagrams' are more widely used than Petri nets with rates.   They are more general in some ways and less so in others.    On the one hand, stock and flow diagrams only allow transitions with either:
\begin{itemize}
\item one place as input and one as output (e.g.\ the transition of a patient from one state to another),
\item one place as input and none as output (e.g.\ death), or
\item no place as input and one as output (e.g\ birth).
\end{itemize}
On the other hand, they can describe dynamics more general than given by the law of mass action, and this is crucial in epidemiology.

Like open Petri nets with rates, open stock and flow diagrams can be described using decorated cospans and work well in CatLab.  Working with the Catlab team, Osgood and Li developed a software package called StockFlow to handle open stock and flow diagrams \cite{BLOP,StockFlow}.   Osgood's student Redekopp also developed a user-friendly web-based front end for StockFlow, called ModelCollab \cite{BLLOR,ModelCollab}.   For expository papers explaining all this software in more detail, see \cite{BLLOR,LMOP}.

This software has the following advantages over traditional software:

\begin{itemize}
\item \define{Compositionality.} Rather than merely a piece of code, each model is a crisply defined mathematical structure designed from the start to be easily combined with other models: for example, a structured or decorated cospan.   This lets models of specific subsystems be constructed individually by different domain experts,  and then composed to form larger models, supporting ongoing collaboration between these parties.   
\item \define{Functorial semantics.}  There is a clear distinction between a model and a specific way of extracting information from this model.    This is achieved by treating different ways of extracting information from models as different functors.
\item \define{Ease of stratification.}  One can `stratify' models---that is, create more detailed models by subdividing stocks in an existing model---without rewriting the whole model from scratch by hand.   This is achieved using pullbacks.
\end{itemize}

So far, the software mentioned above does not make explict the \emph{double} categorical nature of structured or decorated cospans, though it is implicit.   Patterson is working on a new software platform, CatColab \cite{CatColab}, based on `double categorical doctrines'.   However, this use of double categories is different from working with double categories of open systems.   On the other hand, the UK agency ARIA is running a project for safeguarded AI which may develop software using Libkind and Myers' work on double categories of open systems \cite{Libkind,LM,MyersBook}.   We say a bit about their work in Section \ref{subsec:hypergraph}.

Eventually people may prove theorems about the dynamics of open Petri nets with rates, but this has not happened yet.   One reason is that few have tried.   Another is that while there are many interesting theorems in mathematical chemisty \cite{Feinberg1987,Feinberg2019,HJ}, and challenging problems such as the Global Attractor Conjecture \cite{Craciun2015}, the Persistence Conjecture and the Permanence Conjecture \cite{CNP}, few involve Petri nets.   Far more involve `reaction networks'.  

Petri nets and reaction networks are often considered equivalent formalisms.  For example, this Petri net:
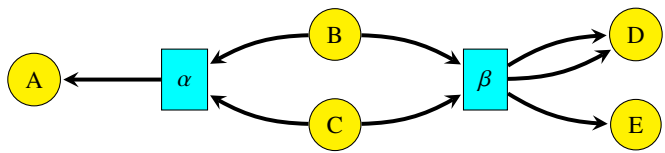
\[
\begin{tikzpicture}
	\begin{pgfonlayer}{nodelayer}
		\node [style=species] (B) at (0,0.6) {$\text{B}$};
		\node [style=species] (C) at (0,-0.6) {$\text{C}$};
		\node [style=transition] (alpha) at (-2,0) {$\phantom{\Big|}\alpha\phantom{\Big|}$};
		\node [style=species] (A) at (-4,0) {$\text{A}$};
		\node [style=transition] (beta) at (2,0) {$\phantom{\Big|}\beta\phantom{\Big|}$};
		\node [style=species] (D) at (4,0.6) {$\text{D}$};
		\node [style=species] (E) at (4,-0.6) {$\text{E}$};
%		\node [style=species] (peroxide) at (4,-1.4) {$\text{H}_2\text{O}_2$};
%		\node [style=transition] (beta) at (2,-1.4) {$\phantom{\Big|}r_2$\phantom{\Big|}};
%		\node [style=transition] (Something) at (6,0) {\tiny{Something}};
%		\node [style=species] (A) at (8,1) {O$\text{H}^{-}$};
%		\node [style=species] (B) at (8,-1) {$\text{H}_3 \text{O}^{+}$};
	\end{pgfonlayer}
	\begin{pgfonlayer}{edgelayer}
		\draw [style=inarrow, bend left =15, looseness=1.00] (B) to (beta);
		\draw [style=inarrow, bend right=15, looseness=1.00] (C) to (beta);
		\draw [style=inarrow,  bend left=15, looseness=1.00] (beta) to (D);
		\draw [style=inarrow,  bend right=15, looseness=1.00] (beta) to (D);
		\draw [style=inarrow, bend right =15, looseness=1.00] (beta) to (E);
		\draw [style=inarrow, bend right =15, looseness=1.00] (B) to (alpha);
		\draw [style=inarrow, bend left =15, looseness=1.00] (C) to (alpha);
		\draw [style=inarrow] (alpha) to (A);
	\end{pgfonlayer}
\end{tikzpicture}
\]
corresponds to this reaction network:
\[      A \longleftarrow \text{B} + \text{C} \longrightarrow 2 \text{D} + \text{E}. \]
The reaction network is a graph where each edge corresponds to a transition of the Petri net, while each vertex corresponds to a so-called `complex': a sum of species that is the source or target of some transition.  However, we compose open Petri nets by identifying species, while individual species are not a visible part of a reaction network.    This leaves us with two challenges:

\begin{problem}
Prove theorems about the qualitative dynamics of open Petri nets with rates, perhaps building on existing results about Petri nets with rates, for example as in \cite{CMPY,CTF}, where they are called `directed species-reaction graphs'.
\end{problem}

\begin{problem}
Develop a framework for open reaction networks suitable for generalizing existing results on reaction networks and invariants such as the `deficiency' of a reaction network \cite{Feinberg1987,Feinberg2019}.
\end{problem}

\section{The variable sharing paradigm}
\label{sec:variable_sharing}

As mentioned in Section \ref{sec:introduction}, there are several paradigms for open systems.   Here we are solely focused on the `variable sharing paradigm', where we compose open systems by gluing them together, or identifying variables.   While we have tried to explain how structured and decorated cospans let us work within this paradigm, we have not yet analyzed which features are distinctive of this paradigm.   We turn to this question now.

When he first discovered decorated cospans, Fong  was studying electrical circuits \cite{BF,Fong,FongThesis}.
He noticed that electrical circuits can be composed in very flexible ways.   A circuit can have any number of wires coming out of it, with no fundamental distinction between inputs and outputs.   We can compose circuits by arbitrarily connecting their wires.     Furthermore:
\begin{itemize}
\item We can take two wires coming out of a circuit and join them, getting a new circuit with one fewer
wire coming out.  
\[
  \xymatrix{
    \mult{.075\textwidth} 
    }
\]
\item We can add an extra wire to a circuit that doesn't actually connect with anything, getting a new circuit
with one more wire coming out.
\[
  \xymatrix{
    \unit{.075\textwidth} 
    }
\]
\item We can take any wire coming out of a circuit and split it, getting a new circuit with one more wire coming out.
\[
  \xymatrix{
    \comult{.075\textwidth} 
    }
\]
\item We can take a wire coming out of a circuit and cap it off, getting a new circuit with one fewer
wire coming out.
\[
  \xymatrix{
    \counit{.075\textwidth} 
    }
\]
\end{itemize}

These constructions obey all the rules of a `special commutative Frobenius monoid', meaning the associative law:
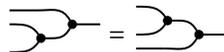
\[
  \xymatrix{
    \assocr{.1\textwidth} = \assocl{.1\textwidth} 
    }
\]
the left and right unit laws:
\[
  \xymatrix{
   \unitl{.1\textwidth} = \idone{.1\textwidth} =  \unitr{.1\textwidth} 
    }
\]
the commutative law:
\[
\xymatrix{
    \commute{.1\textwidth}  \raisebox{0.2 em}{=}  \mult{.07\textwidth}
  }
\]
 where $\swap{1em}$ denotes two wires crossing over each other, together with the `Frobenius equations':
\[
  \xymatrix{
    \frobs{.1\textwidth} \; \raisebox{0.2 em}{=} \; \frobx{.1\textwidth}  \; \raisebox{0.2 em}{=} \;  
    \frobz{.1\textwidth} 
    }
\]
and finally the `special' law:
\[
  \xymatrix{
    \spec{.1\textwidth} \,\raisebox{0.2 em}{=} \, \idone{.1\textwidth}
    }
\]
By composing some of these operations we can create a `cup' and a `cap':
\[
\raisebox{-0.7 em}{
 \scalebox{0.5}{
\begin{tikzpicture}
		\node [style=none] (0) at (1, -0) {};
		\node [style=none] (1) at (0.125, -0) {};
		\node [style=none] (2) at (-1, 0.5) {};
		\node [style=none] (3) at (-1, -0.5) {};
		\node [style=none] (4) at (1, -0) {};
	\draw[line width = 1.5pt] (2) to [in=0, out=0,looseness = 4] (3);
      \end{tikzpicture}
}
}
\hskip -0.5em
  := \;
 \raisebox{-0.5 em}{
 \scalebox{0.5}{
 \begin{tikzpicture}
	\begin{pgfonlayer}{nodelayer}
		\node [style=none] (0) at (1, -0) {};
		\node [style=circ] (1) at (0.125, -0) {};
		\node [style=none] (2) at (-1, 0.5) {};
		\node [style=none] (3) at (-1, -0.5) {};
		\node [style=circ] (4) at (1, 0) {};
	\end{pgfonlayer}
	\begin{pgfonlayer}{edgelayer}
		\draw[line width=2pt] (0.center) to (1.center);
		\draw[line width=2pt] [in=0, out=120, looseness=1.20] (1.center) to (2.center);
		\draw[line width=2pt] [in=0, out=-120, looseness=1.20] (1.center) to (3.center);
	\end{pgfonlayer}
      \end{tikzpicture}
      }
      }
      \quad 
 \raisebox{-0.5 em}{
 \scalebox{0.5}{
\begin{tikzpicture}
		\node [style=none] (0) at (-.75, -0) {};
		\node [style=none] (1) at (-0.125, -0) {};
		\node [style=none] (2) at (1, 0.5) {};
		\node [style=none] (3) at (1, -0.5) {};
		\node [style=none] (4) at (-1, -0) {};
		\node [style=none] (5) at (-0.5, -.5) {};
		\node [style=none] (6) at (-0.5, .5) {};
	\draw[line width = 1.5pt] (2) to [in=180, out=180,looseness = 4] (3);
\end{tikzpicture}
}
}
  := 
 \raisebox{-0.5 em}{
 \scalebox{0.5}{
 \begin{tikzpicture}
	\begin{pgfonlayer}{nodelayer}
		\node [style=none] (0) at (-1, -0) {};
		\node [style=circ] (1) at (-0.125, -0) {};
		\node [style=none] (2) at (1, 0.5) {};
		\node [style=none] (3) at (1, -0.5) {};
		\node [style=circ] (4) at (-1, -0) {};
	\end{pgfonlayer}
	\begin{pgfonlayer}{edgelayer}
		\draw[line width=2pt] (0.center) to (1.center);
		\draw[line width=2pt] [in=180, out=60, looseness=1.20] (1.center) to (2.center);
		\draw[line width=2pt] [in=180, out=-60, looseness=1.20] (1.center) to (3.center);
	\end{pgfonlayer}
      \end{tikzpicture}
      }
      }
\]
These allow us to bend wires around.  Moreover, thanks to the Frobenius monoid
laws, the cap and cup obey the so-called `zigzag identities':
\[
 \scalebox{0.4}{
\begin{tikzpicture}
	\begin{pgfonlayer}{nodelayer}
		\node [style=none] (0) at (-1.5, 0.5) {};
		\node [style=none] (1) at (-0.75, 0.5) {};
		\node [style=none] (2) at (0.25, -0) {};
		\node [style=none] (3) at (0.25, 1) {};
		\node [style=none] (4) at (1, -0.5) {};
		\node [style=none] (5) at (0, -0) {};
		\node [style=none] (6) at (1.75, -0.5) {};
		\node [style=none] (7) at (0, -1) {};
		\node [style=none] (8) at (1.75, 1) {};
		\node [style=none] (9) at (-1.5, -1) {};
	\end{pgfonlayer}
	\begin{pgfonlayer}{edgelayer}
		\draw[line width=2pt] [in=180, out=-60, looseness=1.20] (1) to (2.center);
		\draw[line width=2pt] [in=180, out=60, looseness=1.20] (1) to (3.center);
%		\draw[line width=2pt] (0.center) to (1);
%		\draw[line width=2pt] (6.center) to (4);
		\draw[line width=2pt] [in=0, out=120, looseness=1.20] (4) to (5.center);
		\draw[line width=2pt] [in=0, out=-120, looseness=1.20] (4) to (7.center);
		\draw[line width=2pt] (3.center) to (8.center);
		\draw[line width=2pt] (7.center) to (9.center);
	\end{pgfonlayer}
\end{tikzpicture}
}
\raisebox{0.8em}{
=
}
\raisebox{1em}{
 \scalebox{0.4}{ 
\begin{tikzpicture}
	\begin{pgfonlayer}{nodelayer}
		\node [style=none] (0) at (-1.2, -0) {};
		\node [style=none] (3) at (1.2, -0) {};
	\end{pgfonlayer}
	\begin{pgfonlayer}{edgelayer}
		\draw[line width=2pt] (0) to (3);
	\end{pgfonlayer}
\end{tikzpicture}
}
}
\raisebox{0.8em}{
=
}
 \scalebox{0.4}{
\begin{tikzpicture}
	\begin{pgfonlayer}{nodelayer}
		\node [style=none] (0) at (1.75, 0.5) {};
		\node [style=none] (1) at (1, 0.5) {};
		\node [style=none] (2) at (0, -0) {};
		\node [style=none] (3) at (0, 1) {};
		\node [style=none] (4) at (-0.75, -0.5) {};
		\node [style=none] (5) at (0.25, -0) {};
		\node [style=none] (6) at (-1.5, -0.5) {};
		\node [style=none] (7) at (0.25, -1) {};
		\node [style=none] (8) at (-1.5, 1) {};
		\node [style=none] (9) at (1.75, -1) {};
	\end{pgfonlayer}
	\begin{pgfonlayer}{edgelayer}
		\draw[line width=2pt] [in=0, out=-120, looseness=1.20] (1) to (2.center);
		\draw[line width=2pt] [in=0, out=120, looseness=1.20] (1) to (3.center);
%		\draw[line width=2pt] (0.center) to (1);
%		\draw[line width=2pt] (6.center) to (4);
		\draw[line width=2pt] [in=180, out=60, looseness=1.20] (4) to (5.center);
		\draw[line width=2pt] [in=180, out=-60, looseness=1.20] (4) to (7.center);
		\draw[line width=2pt] (3.center) to (8.center);
		\draw[line width=2pt] (7.center) to (9.center);
	\end{pgfonlayer}
\end{tikzpicture}
}
\]
Fong noticed that all these are general features of decorated cospan categories.   Every object in a decorated cospan category is a special commutative Frobenius monoid internal to that category, in a canonical way.   Furthermore, all these Frobenius structures fit together to make the category into a `hypergraph category'---a concept which, while apparently quite complicated, had already proved its importance by independently being discovered by multiple authors in different contexts \cite{Carboni,GH,Morton,Kissinger,RSW2}.   

\begin{defn}
\label{defn:hypergraph_category}
A \define{hypergraph category} is a symmetric monoidal category $(\C,\otimes)$ where each object $a \in \C$ is equipped with a \define{multiplication} $\mu_a \maps a \otimes a \to a$, \define{unit} $\eta_a \maps I \to a$, \define{comultiplication} $\delta_a \maps a \to a \otimes a$, and \define{counit} $\epsilon_a \maps a \to I$,
obeying the laws of a special commutative Frobenius monoid and satisfying
\[ \begin{array}{ccl}
     \mu_{a \otimes b} &=&  ( \mu_a \otimes \mu_b) (1_a \otimes \sigma_{b,a} \otimes 1_b) \\ 
       \eta_{a \otimes b} &=& \eta_a \otimes \eta_b  \\
      \delta_{a \otimes b} &=& (1 \otimes \sigma_{b,a} \otimes 1) (\delta_a \otimes \delta_b) \\
      \epsilon_{a \otimes b} &=& \epsilon_a \otimes \epsilon_b 
      \end{array}
\]
where $\sigma_{a,b} \maps a \otimes b \to b \otimes a$ is the symmetry, as well as
\[             \eta_I = 1_I = \epsilon_I \]
where $I$ is the unit for the tensor product.
\end{defn}

The morphisms in a hypergraph category are \emph{not} required to be Frobenius monoid homomorphisms!   Indeed, this fails in most examples, since a Frobenius monoid homomorphism is automatically an isomorphism.  

The rather curious concept of hypergraph category was greatly clarified when Fong and Spivak gave a slick equivalent definition using operads built from categories of cospans \cite{FS}.   Since they were working with categories rather than double categories, we can only apply their work to a double category $\lD$ by forming the category $\D$ whose morphisms are \emph{isomorphism classes} of loose morphisms in $\lD$.   To get around this, we shall try to recapitulate some of their work using double categories.  In the end this suggests two possible concepts of `hypergraph double category': one that categorifies their definition, and another that categorifies Definition \ref{defn:hypergraph_category}.   

More precisely, in Section \ref{subsec:cospans_of_interfaces} we show that any symmetric monoidal structured or decorated cospan double category $\lD$ has a kind of `exoskeleton' where the loose morphisms are \emph{trivially} structured or decorated cospans.   These can be seen as cospans in $\lD_0$, the category of objects and tight morphisms of $\lD$.    In Section \ref{subsec:outer_shell}, we show that this exoskeleton $\lCsp(\lD_0)$ in turn has an `outer shell' where the loose morphisms are cospans built up using only the objects of $\lD_0$ and the operations present in any category with finite colimits.   More precisely this outer shell is $\lCsp(\tilde{\lD}_0)$, where $\tilde{\lD}_0$ is the free category with finite colimits on the objects of $\lD$.   We show there are symmetric monoidal double functors
\[            \lCsp(\tilde{\lD}_0) \xrightarrow{j} \lCsp(\lD) \xrightarrow{\iota} \lD .\]
In Section \ref{subsec:hypergraph} we use the outer shell to motivate Libkind and Myers' definition of hypergraph double category \cite[Ex.\ 8.11]{LM}, which categorifies Fock and Spivak's slick definition of hypergraph category.
In Section \ref{subsec:frobenius} we show that any object in a symmetric monoidal structured or decorated cospan double category is a categorified version of a commutative Frobenius monoid.   This suggests a route to categorifying the old definition of hypergraph category, Definition \ref{defn:hypergraph_category}.

\subsection{The exoskeleton}
\label{subsec:cospans_of_interfaces}

One common feature of decorated and structured cospans is a certain relationship between `interfaces' and `open systems'.    Namely, in both formalisms, any cospan of interfaces can be seen as a degenerate case of an open system.     For example, a typical open Petri net might look like this:
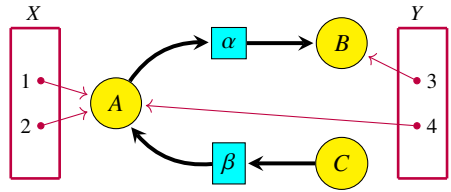
\[
\begin{tikzpicture}
	\begin{pgfonlayer}{nodelayer}

		\node [style = transition] (b) at (2.5, 0.8) {$\alpha$};
		\node [style = transition] (c) at (2.5, -0.8) {$\beta$};
		\node [style = species] (A) at (1, 0) {$A$};
		\node [style = species] (B) at (4,0.8) {$B$};
		\node [style = species] (C) at (4,-0.8) {$C$};

		\node [style=empty] (Y) at (-0.1, 1.2) {$X$};
		\node [style=none] (Ytr) at (.25, 1) {};
		\node [style=none] (Ytl) at (-.4, 1) {};
		\node [style=none] (Ybr) at (.25, -1) {};
		\node [style=none] (Ybl) at (-.4, -1) {};

		\node [style=inputdot] (1) at (0, 0.3) {};
		\node [style=empty] at (-0.2, 0.3) {$1$};
		\node [style=inputdot] (2) at (0, -0.3) {};
		\node [style=empty] at (-0.2, -0.3) {$2$};		
		
		\node [style=empty] (Z) at (5, 1.2) {$Y$};
		\node [style=none] (Ztr) at (4.75, 1) {};
		\node [style=none] (Ztl) at (5.4, 1) {};
		\node [style=none] (Zbl) at (5.4, -1) {};
		\node [style=none] (Zbr) at (4.75, -1) {};

		\node [style=inputdot] (3) at (5, 0.3) {};
		\node [style=empty] at (5.2, 0.3) {$3$};	
		\node [style=inputdot] (4) at (5, -0.3) {};
		\node [style=empty] at (5.2, -0.3) {$4$};	
		
	\end{pgfonlayer}
	\begin{pgfonlayer}{edgelayer}
%		\draw [style=inarrow] (A) to (a);
%		\draw [style=inarrow] (B) to (a);
%		\draw [style=inarrow] (a) to (C);
%		\draw [style=inarrow] (a) to (D);
		\draw [style=inarrow, bend left=30, looseness=1.00] (A) to (b);
		\draw [style=inarrow] (b) to (B);
		\draw [style=inarrow, bend left=30, looseness=1.00] (c) to (A);
		\draw [style=inarrow] (C) to (c);
%		\draw [style=inputarrow] (1) to (A);
%		\draw [style=inputarrow] (2) to (B);
%		\draw [style=inputarrow] (3) to (B);
%		\draw [style=inputarrow] (4) to (C);
%		\draw [style=inputarrow] (5) to (D);
		\draw [style=inputarrow] (1) to (A);
		\draw [style=inputarrow] (2) to (A);
		\draw [style=inputarrow] (3) to (B);
		\draw [style=inputarrow] (4) to (A);
		
%		\draw [style=simple] (Xtl.center) to (Xtr.center);
%		\draw [style=simple] (Xtr.center) to (Xbr.center);
%		\draw [style=simple] (Xbr.center) to (Xbl.center);
%		\draw [style=simple] (Xbl.center) to (Xtl.center);
		\draw [style=simple] (Ytl.center) to (Ytr.center);
		\draw [style=simple] (Ytr.center) to (Ybr.center);
		\draw [style=simple] (Ybr.center) to (Ybl.center);
		\draw [style=simple] (Ybl.center) to (Ytl.center);
		\draw [style=simple] (Ztl.center) to (Ztr.center);
		\draw [style=simple] (Ztr.center) to (Zbr.center);
		\draw [style=simple] (Zbr.center) to (Zbl.center);
		\draw [style=simple] (Zbl.center) to (Ztl.center);
	\end{pgfonlayer}
\end{tikzpicture}
\]
But sometimes an open Petri net has no transitions:
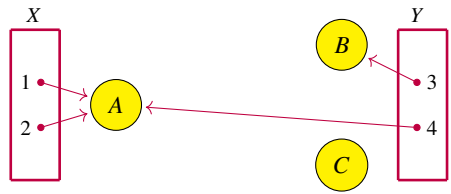
\[
\begin{tikzpicture}
	\begin{pgfonlayer}{nodelayer}

%		\node [style = transition] (b) at (2.5, 0.8) {$\alpha$};
%		\node [style = transition] (c) at (2.5, -0.8) {$\beta$};
		\node [style = species] (A) at (1, 0) {$A$};
		\node [style = species] (B) at (4,0.8) {$B$};
		\node [style = species] (C) at (4,-0.8) {$C$};

		\node [style=empty] (Y) at (-0.1, 1.2) {$X$};
		\node [style=none] (Ytr) at (.25, 1) {};
		\node [style=none] (Ytl) at (-.4, 1) {};
		\node [style=none] (Ybr) at (.25, -1) {};
		\node [style=none] (Ybl) at (-.4, -1) {};

		\node [style=inputdot] (1) at (0, 0.3) {};
		\node [style=empty] at (-0.2, 0.3) {$1$};
		\node [style=inputdot] (2) at (0, -0.3) {};
		\node [style=empty] at (-0.2, -0.3) {$2$};		
		
		\node [style=empty] (Z) at (5, 1.2) {$Y$};
		\node [style=none] (Ztr) at (4.75, 1) {};
		\node [style=none] (Ztl) at (5.4, 1) {};
		\node [style=none] (Zbl) at (5.4, -1) {};
		\node [style=none] (Zbr) at (4.75, -1) {};

		\node [style=inputdot] (3) at (5, 0.3) {};
		\node [style=empty] at (5.2, 0.3) {$3$};	
		\node [style=inputdot] (4) at (5, -0.3) {};
		\node [style=empty] at (5.2, -0.3) {$4$};	
		
	\end{pgfonlayer}
	\begin{pgfonlayer}{edgelayer}
%		\draw [style=inarrow] (A) to (a);
%		\draw [style=inarrow] (B) to (a);
%		\draw [style=inarrow] (a) to (C);
%		\draw [style=inarrow] (a) to (D);
%		\draw [style=inarrow, bend left=30, looseness=1.00] (A) to (b);
%		\draw [style=inarrow] (b) to (B);
%		\draw [style=inarrow, bend left=30, looseness=1.00] (c) to (A);
%		\draw [style=inarrow] (C) to (c);
%		\draw [style=inputarrow] (1) to (A);
%		\draw [style=inputarrow] (2) to (B);
%		\draw [style=inputarrow] (3) to (B);
%		\draw [style=inputarrow] (4) to (C);
%		\draw [style=inputarrow] (5) to (D);
		\draw [style=inputarrow] (1) to (A);
		\draw [style=inputarrow] (2) to (A);
		\draw [style=inputarrow] (3) to (B);
		\draw [style=inputarrow] (4) to (A);
		
%		\draw [style=simple] (Xtl.center) to (Xtr.center);
%		\draw [style=simple] (Xtr.center) to (Xbr.center);
%		\draw [style=simple] (Xbr.center) to (Xbl.center);
%		\draw [style=simple] (Xbl.center) to (Xtl.center);
		\draw [style=simple] (Ytl.center) to (Ytr.center);
		\draw [style=simple] (Ytr.center) to (Ybr.center);
		\draw [style=simple] (Ybr.center) to (Ybl.center);
		\draw [style=simple] (Ybl.center) to (Ytl.center);
		\draw [style=simple] (Ztl.center) to (Ztr.center);
		\draw [style=simple] (Ztr.center) to (Zbr.center);
		\draw [style=simple] (Zbr.center) to (Zbl.center);
		\draw [style=simple] (Zbl.center) to (Ztl.center);
	\end{pgfonlayer}
\end{tikzpicture}
\]
A Petri net with no transitions amounts to the same thing as a cospan of finite sets:
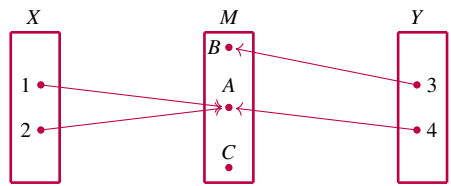
\[
\begin{tikzpicture}
	\begin{pgfonlayer}{nodelayer}

%		\node [style = transition] (b) at (2.5, 0.8) {$\alpha$};
%		\node [style = transition] (c) at (2.5, -0.8) {$\beta$};
	         \node [style=inputdot] (A) at (2.5, 0) {};
		\node [style = empty] at (2.5, 0.3) {$A$};
		\node [style=inputdot] (B) at (2.5, 0.8) {};
		\node [style = empty] at (2.3,0.8) {$B$};
		\node [style=inputdot] (C) at (2.5, -0.8) {};
		\node [style = empty] at (2.5,-0.6) {$C$};
		
		\node [style=empty] (M) at (2.5, 1.2) {$M$};
		\node [style=none] (Xtr) at (2.825, 1) {};
		\node [style=none] (Xtl) at (2.175, 1) {};
		\node [style=none] (Xbl) at (2.175, -1) {};
		\node [style=none] (Xbr) at (2.825, -1) {};

		\node [style=empty] (Y) at (-0.1, 1.2) {$X$};
		\node [style=none] (Ytr) at (.25, 1) {};
		\node [style=none] (Ytl) at (-.4, 1) {};
		\node [style=none] (Ybr) at (.25, -1) {};
		\node [style=none] (Ybl) at (-.4, -1) {};

		\node [style=inputdot] (1) at (0, 0.3) {};
		\node [style=empty] at (-0.2, 0.3) {$1$};
		\node [style=inputdot] (2) at (0, -0.3) {};
		\node [style=empty] at (-0.2, -0.3) {$2$};		
		
		\node [style=empty] (Z) at (5, 1.2) {$Y$};
		\node [style=none] (Ztr) at (4.75, 1) {};
		\node [style=none] (Ztl) at (5.4, 1) {};
		\node [style=none] (Zbl) at (5.4, -1) {};
		\node [style=none] (Zbr) at (4.75, -1) {};

		\node [style=inputdot] (3) at (5, 0.3) {};
		\node [style=empty] at (5.2, 0.3) {$3$};	
		\node [style=inputdot] (4) at (5, -0.3) {};
		\node [style=empty] at (5.2, -0.3) {$4$};	
		
	\end{pgfonlayer}
	\begin{pgfonlayer}{edgelayer}
%		\draw [style=inarrow] (A) to (a);
%		\draw [style=inarrow] (B) to (a);
%		\draw [style=inarrow] (a) to (C);
%		\draw [style=inarrow] (a) to (D);
%		\draw [style=inarrow, bend left=30, looseness=1.00] (A) to (b);
%		\draw [style=inarrow] (b) to (B);
%		\draw [style=inarrow, bend left=30, looseness=1.00] (c) to (A);
%		\draw [style=inarrow] (C) to (c);
%		\draw [style=inputarrow] (1) to (A);
%		\draw [style=inputarrow] (2) to (B);
%		\draw [style=inputarrow] (3) to (B);
%		\draw [style=inputarrow] (4) to (C);
%		\draw [style=inputarrow] (5) to (D);
		\draw [style=inputarrow] (1) to (A);
		\draw [style=inputarrow] (2) to (A);
		\draw [style=inputarrow] (3) to (B);
		\draw [style=inputarrow] (4) to (A);
		
		\draw [style=simple] (Xtl.center) to (Xtr.center);
		\draw [style=simple] (Xtr.center) to (Xbr.center);
		\draw [style=simple] (Xbr.center) to (Xbl.center);
		\draw [style=simple] (Xbl.center) to (Xtl.center);
		\draw [style=simple] (Ytl.center) to (Ytr.center);
		\draw [style=simple] (Ytr.center) to (Ybr.center);
		\draw [style=simple] (Ybr.center) to (Ybl.center);
		\draw [style=simple] (Ybl.center) to (Ytl.center);
		\draw [style=simple] (Ztl.center) to (Ztr.center);
		\draw [style=simple] (Ztr.center) to (Zbr.center);
		\draw [style=simple] (Zbr.center) to (Zbl.center);
		\draw [style=simple] (Zbl.center) to (Ztl.center);
	\end{pgfonlayer}
\end{tikzpicture}
\]
In this example finite sets are playing the role of `interfaces', while open Petri nets are our `open systems'.  So, we are seeing that a cospan of interfaces is a degenerate open system.  The only significant use of this degenerate open system is to glue other open systems together, by composition.  But as we shall see, this is quite important.

We can formalize this fact as follows.   First, since structured and decorated cospan categories behave so similarly, let us lump them together in the following ungainly definition:

\begin{defn}
We define $\lD$ to be a \define{double category of structured/decorated cospans} if it is either
\begin{enumerate}
\item the double category ${}_L \lCsp$ of \emph{structured} cospans arising from a functor $L \maps \A \to \X$ meeting the hypotheses of Theorem \ref{thm:structured_cospan_1}, or
\item  the double category $F \lCsp$ of \emph{decorated} cospans arising from a lax monoidal pseudofunctor $F \maps (\A, +) \to (\Cat, \times)$ meeting the hypotheses of Theorem \ref{thm:decorated_cospan_1}.
\end{enumerate}
In either of these cases, we say $\lD$ is a \define{symmetric monoidal double category of structured/decorated cospans} if 
\begin{enumerate}
\item the hypotheses of Theorem \ref{thm:structured_cospan_2} hold, so $\lD = {}_L \lCsp$ is symmetric monoidal, or
\item the hypotheses of Theorem \ref{thm:decorated_cospan_2} hold,  so $\lD = F\lCsp$ is symmetric monoidal.
\end{enumerate}
\end{defn}

Second, let us note a fact about cospan double categories:

\begin{lem}
\label{lem:cospans}
 For any category $\A$ with pushouts there is a double category $\lCsp(\A)$ in which:
\begin{itemize}
\item objects are objects of $\A$,
\item tight morphisms are morphisms in $\A$, composed in the same way,
\item loose morphisms are cospans in $\A$, composed via pushouts,
\item 2-cells are maps of spans in $\A$, composed via pushouts.
\end{itemize}
When $\A$ has finite colimits, $\lCsp(\A)$ naturally has the structure of a cocartesian double
category, and thus becomes symmetric monoidal using finite coproducts in $\A$.
\end{lem}

\begin{proof}
This is a special case of Theorems \ref{thm:structured_cospan_1} and \ref{thm:structured_cospan_2}, taking $L \maps \A \to \X$ to be the identity functor, but most of these facts were known earlier \cite{Niefield,Shulman2008}.
\end{proof}

Third, for any double category $\lD$, define its \define{tight category} $\lD_0$ to be its category of objects and tight morphisms.  When $\lD$ is a structured/decorated double category, cospans in $\lD_0$ are what we are calling cospans of interfaces.  

With these preliminaries out of the way, we can formalize the idea that a cospan of interfaces is a degenerate sort of open system:

\begin{thm}
\label{thm:structured/decorated_cospan_interfaces}
Let $\lD$ be a double category of structured/decorated cospans.  Then there is a double functor 
\[    \iota \maps \lCsp(\lD_0) \to \lD , \]
unique up to isomorphism, that restricts to the identity functor on $\lD_0$.   If $\lD$ is symmetric monoidal, then $\iota$ can be given the structure of a symmetric monoidal double functor.
\end{thm}

\begin{proof}
We divide the proof into the three lemmas, which occupy the rest of this section.  Lemma \ref{lem:structured_cospan_interfaces} proves existence for structured cospans, Lemma \ref{lem:decorated_cospan_interfaces} proves existence for decorated cospans, and Lemma \ref{lem:characterizing_lCsp(A)} proves uniqueness up to isomorphism in both cases.
\end{proof}

\begin{lem}
\label{lem:structured_cospan_interfaces}
Let $L \maps \A \to \X$ be a functor, $\X$ a category with pushouts, and ${}_L\lCsp$ the double category of $L$-structured cospans.   Then there is a double functor $\iota \maps \lCsp(\A) \to {}_L\lCsp$ such that
\begin{itemize}
\item  $\iota$ acts as the identity on objects and tight morphisms,
\item $\iota$ acts as $L$ on loose morphisms and 2-cells: it maps any loose morphism
\[
\scalebox{1.2}{
\begin{tikzpicture}[scale=1.0]
\node (A) at (0,0) {$a$};
\node (B) at (1,0) {$m$};
\node (C) at (2,0) {$b$};
\path[->,font=\scriptsize,>=angle 90]
(A) edge node[above,pos=0.5]{$i$} (B)
(C)edge node[above,pos=0.4]{$o$}(B);
\end{tikzpicture}
}
\]
in $\lCsp(\A)$ to the structured cospan
\[
\scalebox{1.2}{
\begin{tikzpicture}[scale=1.0]
\node (A) at (0,0) {$L(a)$};
\node (B) at (1.5,0) {$L(m)$};
\node (C) at (3,0) {$L(b)$,};
\path[->,font=\scriptsize,>=angle 90]
(A) edge node[above,pos=0.5]{$L(i)$} (B)
(C) edge node[above,pos=0.5]{$L(o)$}(B);
\end{tikzpicture}
}
\]
and it maps any 2-cell
\[
\scalebox{1.2}{
\begin{tikzpicture}[scale=1.5]
\node (A) at (0,0) {$a$};
\node (B) at (1,0) {$m$};
\node (C) at (2,0) {$b$};
\node (A') at (0,-1) {$a'$};
\node (B') at (1,-1) {$m'$};
\node (C') at (2,-1) {$b'$};
\path[->,font=\scriptsize,>=angle 90]
(A) edge node[above]{$i$} (B)
(C) edge node[above]{$o$} (B)
(A') edge node[below]{$i'$} (B')
(C') edge node[below]{$o'$} (B')
(A) edge node [left]{$f$} (A')
(B) edge node [left]{$\alpha$} (B')
(C) edge node [right]{$g$} (C');
\end{tikzpicture}
}
\]
in $\lCsp(\A)$ to the map of structured cospans
\[
\scalebox{1.2}{
\begin{tikzpicture}[scale=1.5]
\node (A) at (0,0) {$L(a)$};
\node (B) at (1.2,0) {$L(m)$};
\node (C) at (2.4,0) {$L(b)$};
\node (A') at (0,-1) {$L(a')$};
\node (B') at (1.2,-1) {$L(m')$};
\node (C') at (2.4,-1) {$L(b')$.};
\path[->,font=\scriptsize,>=angle 90]
(A) edge node[above]{$L(i)$} (B)
(C) edge node[above]{$L(o)$} (B)
(A') edge node[below]{$L(i')$} (B')
(C') edge node[below]{$L(o')$} (B')
(A) edge node [left]{$L(f)$} (A')
(B) edge node [left]{$L(\alpha)$} (B')
(C) edge node [right]{$L(g)$} (C');
\end{tikzpicture}
}
\]
\end{itemize}
If $\A$ and $\X$ also have finite coproducts and $L$ preserves them, then $\iota \maps \lCsp(\A) \to {}_L\lCsp$ can be given the structure of a symmetric monoidal double functor.
\end{lem}

\begin{proof}
This follows from Theorems \ref{thm:structured_cospan_functoriality_1} and \ref{thm:structured_cospan_functoriality_2} on maps between structured cospan categories, applied to this square in $\Rex$:
\[
\scalebox{1.2}{
\begin{tikzpicture}[scale=1.5]
\node (A) at (0,0) {$\A$};
\node (B) at (1,0) {$\A$};
\node (C) at (0,-1) {$\A$};
\node (D) at (1,-1) {$\X$.};
\node (E) at (0.5,-0.5) {$\Downarrow 1$};
\path[->,font=\scriptsize,>=angle 90]
(A) edge node[above]{$1$} (B)
(B) edge node [right]{$F$} (D)
(C) edge node [below] {$F$} (D)
(A)edge node[left]{$1$}(C);
\end{tikzpicture}
}
\qedhere
\] 
\end{proof}

In the decorated cospan case we use the fact that given a lax monoidal pseudofunctor $F \maps (\A,+) \to (\Cat,\times)$, any object $m \in \A$ has a blandest possible decoration, the \define{empty decoration} $v_m \in F(m)$.   This is the object of $F(m)$ given by the composite
\[
\scalebox{1.2}{
\begin{tikzpicture}[scale=1.5]
\node (A) at (0,0) {$1$};
\node (B) at (1,0) {$F(0)$};
\node (C) at (2,0) {$F(m)$};
\path[->,font=\scriptsize,>=angle 90]
(A) edge node[above]{$\phi_0$} (B)
(B) edge node[above]{$F(!)$}(C);
\end{tikzpicture}
}
\]
where $\phi_0$ is the unitor for the monoidal structure on $F$ and $!$ is the unique morphism from $0$ to $m$.   The empty decoration is functorial: in fact, to any morphism $\alpha \maps m \to m'$ in $\A$ we can assign a decoration morphism $F(\alpha)(v_m) \to v_{m'}$ which is simply the identity, since $F(\alpha)(v_m) = v_{m'}$.

\begin{lem}
\label{lem:decorated_cospan_interfaces}
Let $\A$ be a category with finite colimits, $F \maps (\A,+) \to (\Cat,\times)$ a lax monoidal pseudofunctor, and $F\lCsp$ the double category of $F$-decorated cospans.  Then there is a double functor $i \maps \lCsp(\A) \to F\lCsp$ such that
\begin{itemize}
\item  $\iota$ acts as the identity on objects and tight morphisms,
\item $\iota$ equips any loose morphism with the empty decoration, and any 2-cell with the identity morphism of decorations: it maps any loose morphism
\[
\scalebox{1.2}{
\begin{tikzpicture}[scale=1.0]
\node (A) at (0,0) {$a$};
\node (B) at (1,0) {$m$};
\node (C) at (2,0) {$b$};
\path[->,font=\scriptsize,>=angle 90]
(A) edge node[above,pos=0.5]{$i$} (B)
(C)edge node[above,pos=0.4]{$o$}(B);
\end{tikzpicture}
}
\]
in $\lCsp(\A)$ to the decorated cospan
\[
\scalebox{1.2}{
\begin{tikzpicture}[scale=1.0]
\node (A) at (0,0) {$a$};
\node (B) at (1,0) {$m$};
\node (C) at (2,0) {$b$};
\node (D) at (4,0) {$v_m \in F(m)$};
\path[->,font=\scriptsize,>=angle 90]
(A) edge node[above,pos=0.5]{$i$} (B)
(C) edge node[above,pos=0.5]{$o$}(B);
\end{tikzpicture}
}
\]
and it maps any 2-cell
\[
\scalebox{1.2}{
\begin{tikzpicture}[scale=1.2]
\node (A) at (0,0) {$a$};
\node (B) at (1,0) {$m$};
\node (C) at (2,0) {$b$};
\node (A') at (0,-1) {$a'$};
\node (B') at (1,-1) {$m'$};
\node (C') at (2,-1) {$b'$};
\path[->,font=\scriptsize,>=angle 90]
(A) edge node[above]{$i$} (B)
(C) edge node[above]{$o$} (B)
(A') edge node[below]{$i'$} (B')
(C') edge node[below]{$o'$} (B')
(A) edge node [left]{$f$} (A')
(B) edge node [left]{$\alpha$} (B')
(C) edge node [right]{$g$} (C');
\end{tikzpicture}
}
\]
in $\lCsp(\A)$ to the map of decorated cospans
\[
\scalebox{1.2}{
\begin{tikzpicture}[scale=1.2]
\node (A) at (0,0) {$a$};
\node (B) at (1,0) {$m$};
\node (C) at (2,0) {$b$};
\node (D) at (3,0) {$v_m \in F(m)$};
\node (A') at (0,-1) {$a'$};
\node (B') at (1,-1) {$m'$};
\node (C') at (2,-1) {$b'$};
\node (D) at (3,-1) {$v_{m'} \in F(m')$};
\path[->,font=\scriptsize,>=angle 90]
(A) edge node[above]{$i$} (B)
(C) edge node[above]{$o$} (B)
(A') edge node[below]{$i'$} (B')
(C') edge node[below]{$o'$} (B')
(A) edge node [left]{$f$} (A')
(B) edge node [left]{$\alpha$} (B')
(C) edge node [right]{$g$} (C');
\end{tikzpicture}
}
\]
with the decoration morphism $F(\alpha)(v_m) \to v_{m'}$ given by the identity.
\end{itemize}
If $F$ is a symmetric lax monoidal pseudofunctor, then $\iota \maps \lCsp(\A) \to F\lCsp$ can be given the structure of a symmetric monoidal double functor. 
\end{lem}

\begin{proof}
This follows from Theorem \ref{thm:decorated_cospan_functoriality}.
\end{proof}

To prove the uniqueness up to isomorphism in Theorem \ref{thm:structured/decorated_cospan_interfaces}, we can use a nice characterization of $\lCsp(\A)$ due to Dawson, Par\'e and Pronk \cite{DPP2010}.  This says roughly that for any category $\A$ with pushouts, $\lCsp(\A)$ is the free fibrant double category having $\A$ as its tight category.

To state this result more precisely, recall that a \define{lax} double functor between double categories, say $F \maps \lC \to \lD$, is like an ordinary (i.e., pseudo) double functor except that the laxator $F(f) \circ F(g) \Rightarrow F(f \circ g)$ and unitor $1 \Rightarrow F(1)$ for composition of loose morphisms are not required to be invertible.   A lax double functor is \define{normal} if its unitor is invertible.   Among the lax double functors, the normal ones are precisely those that preserve companions and conjoints \cite[Prop.\ 3.8]{DPP2010}.

\begin{lem}
\label{lem:characterizing_lCsp(A)}
Let $\A$ be a category with pushouts and $\lD$ a fibrant double category.  Then composing with the inclusion $\A \to \lCsp(\A)$ gives an equivalence between the category of normal lax double functors 
\[       \lCsp(\A) \to \lD \]
and the category of functors 
\[  \A \to \lD_0 . \]
\end{lem}

\begin{proof}
This is \cite[Thm.\ 3.15]{DPP2010}, dualized to apply to cospans. 
\end{proof}

Using this lemma and setting $\A = \lD_0$, we see that under the hypotheses of Theorem \ref{thm:structured/decorated_cospan_interfaces}, the double functor $\iota \maps \lCsp(\lD_0) \to \lD$ is characterized uniquely up to isomorphism by the property that it restricts to the identity on $\lD_0$.    This completes the proof of Theorem \ref{thm:structured/decorated_cospan_interfaces}.

The perceptive reader will notice that most of this proof could have been quickly dispatched using Lemma \ref{lem:characterizing_lCsp(A)}.   However, this lemma does not cover the symmetric monoidal aspects, and our statement of Theorem \ref{thm:structured/decorated_cospan_interfaces} does not claim any uniqueness for the symmetric monoidal structure of $\iota$.  It may be possible to adapt the lemma to deal with this.

\begin{problem}
Prove a result similar to Lemma \ref{lem:characterizing_lCsp(A)} that applies when $\A$ has finite colimits, characterizing $\lCsp(\A)$ as a symmetric monoidal double category in terms of a left universal property.  For example, perhaps $\lCsp(\A)$ is the free fibrant symmetric monoidal double category having $\A$ as its tight category.
\end{problem}

However, we wanted to give explicit descriptions of the double functor $\iota \maps \lCsp(\lD_0) \to \lD$ for structured and decorated cospan categories, to make it crystal clear how cospans of interfaces can be seen as open systems in either of these formalisms.

\subsection{The outer shell}
\label{subsec:outer_shell}

We have shown that any symmetric monoidal structured/decorated cospan category $\lD$ has a kind of `exoskeleton' involving only cospans of interfaces, namely the double category $\lCsp(\lD_0)$.    However, we can trim this down further by replacing $\lD_0$ with the free category with finite colimits on the set of objects of $\lD_0$.    This gives what could be called the `outer shell' of $\lD$: the double category containing only those cospans that can be defined using the objects of $\lD$ and finite colimits.  Examples include the following:
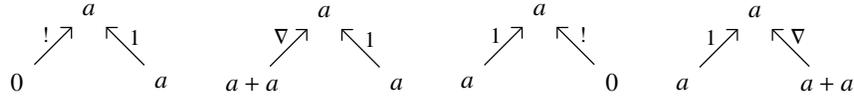
\[
  \xymatrixrowsep{-15pt}
  \xymatrixcolsep{2pt}
  \xymatrix{
  \scalebox{1.2}{
\begin{tikzpicture}[scale=0.8]
\node (A) at (0,0) {$0$};
\node (B) at (1,1) {$a$};
\node (C) at (2,0) {$a$};
\path[->,font=\scriptsize,>=angle 90]
(A) edge node[above,pos=0.3]{$!$} (B)
(C)edge node[above,pos=0.25]{$1$}(B);
\end{tikzpicture}
} 
&
\scalebox{1.2}{
\begin{tikzpicture}[scale=0.8]
\node (A) at (0,0) {$a + a$};
\node (B) at (1,1) {$a$};
\node (C) at (2,0) {$a$};
\path[->,font=\scriptsize,>=angle 90]
(A) edge node[above,pos=0.3]{$\nabla$} (B)
(C)edge node[above,pos=0.25]{$1$}(B);
\end{tikzpicture}
}
&
\scalebox{1.2}{
\begin{tikzpicture}[scale=0.8]
\node (A) at (0,0) {$a$};
\node (B) at (1,1) {$a$};
\node (C) at (2,0) {$0$};
\path[->,font=\scriptsize,>=angle 90]
(A) edge node[above,pos=0.3]{$1$} (B)
(C)edge node[above,pos=0.3]{$!$}(B);
\end{tikzpicture}
}
&
\scalebox{1.2}{
\begin{tikzpicture}[scale=0.8]
\node (A) at (0,0) {$a$};
\node (B) at (1,1) {$a$};
\node (C) at (2,0) {$a+a$};
\path[->,font=\scriptsize,>=angle 90]
(A) edge node[above,pos=0.3]{$1$} (B)
(C)edge node[above,pos=0.3]{$\nabla$}(B);
\end{tikzpicture}
}
}
\]
Here $a$ is any object of $\lD$, $!$ is the unique morphism from the initial object, and $\nabla$ is the 
codiagonal.   If we draw these four cospans as string diagrams:
\[
 \xymatrixrowsep{-15pt}
  \xymatrixcolsep{5em}
  \xymatrix{
    \unit{.075\textwidth}  & \mult{.075\textwidth} &  \counit{.075\textwidth}  & \comult{.075\textwidth} 
  }
\]
we see they echo the four basic wiring patterns listed at the start of Section \ref{sec:variable_sharing}.
This is no coincidence.   Fong and Spivak \cite{FS} discovered that these wiring patterns, and the laws governing them, arise whenever one considers cospans in a category with finite colimits.   To study such cospans in their
purest form, they introduced cospans in the free category with finite colimits on an arbitrary set.   However, they only studied the \emph{category} of such cospans.   Here we study 
the double category of such cospans.

The free category with finite colimits on one object is $\Fin\Set$.  More generally, for any finite set $X$, the free category with finite colimits on $X$ is $\Fin\Set^X$.  But this fails when $X$ is infinite.  Instead, we need to use
the category of finite sets equipped with a map to $X$, which we denote as $\Fin\Set \downarrow X$, even though $X$ itself need not be finite.   The morphisms in $\Fin\Set \downarrow X$ are as usual in a slice category: that is, a morphism from $A \xrightarrow{p} X$ to $B \xrightarrow{q} X$ is a commutative triangle
\[
\scalebox{1.2}{
\begin{tikzpicture}
\node (A) at (1,1) {$A$};
\node (B) at (2,0) {$X$};
\node (C) at (3,1) {$B$};
\path[->,font=\scriptsize,>=angle 90]
(A) edge node[above,pos = 0.5]{$f$} (C)
(A) edge node[left,pos = 0.5]{$p$} (B)
(C) edge node[right, pos = 0.5]{$q$} (B);
\end{tikzpicture}
}
\]
 
 \begin{lem}
 \label{lem:free_category_with_finite_colimits}
The free category with finite colimits on a set $X$ is $\Fin\Set \downarrow X$.   More precisely, suppose $\A$ is a category with finite colimits.  Then any function sending elements of $X$ to objects of $\A$ extends to a finite-colimit-preserving functor from $\Fin\Set \downarrow X$ to $\A$.  Moreover, any two such extensions are naturally isomorphic in a unique way. 
\end{lem}

\begin{proof}
We use a general result \cite[Thm.\ 5.35]{Kelly} about the free category with colimits of a given sort on a small category $\C$, and specialize this to the case where $\C = \Disc(X)$, the discrete category on $X$.   This general result finds the desired free category inside the presheaf category $\widehat{\C} =  \Set^{\C^\op}$.   In particular, it says that the free category with finite colimits on $\C$ is the full subcategory $\widetilde{\C}$ of $\widehat{\C}$ whose objects are finite colimits of representables.   In the case $\C = \Disc(X)$ there is one representable for each element of $X$, and the finite colimits of these are simply finite coproducts.  Thus, an object of $\widetilde{\Disc}(X)$ is a presheaf $F$ on $\Disc(X)$ that assigns to each $x \in X$ a finite set $F(x)$ such that the disjoint union
$\sum_{x \in X} F(x) $ is finite.   The map $p_F \maps \sum_{x \in X} F(x) \to X$ sending all elements of $F(x)$ to $x \in  X$ can be seen as an object of $\Fin\Set \downarrow X$.   By reversing this construction, any object of $\Fin\Set \downarrow X$ gives an object of $\widetilde{\Disc}(X)$. Since morphisms between finite colimits of representables correspond precisely to morphisms in $\Fin\Set \downarrow X$, $\widetilde{\Disc}(X)$ is equivalent to $\Fin\Set \downarrow X$.
\end{proof}

Now, suppose $\A$ is a small category with finite colimits.   Let 
\[          \tilde{\A} = \Fin\Set \downarrow \Ob(\A)  \]
where $\Ob(\A)$ is the set of objects of $\A$.   By Lemma \ref{lem:free_category_with_finite_colimits}, $\tilde{\A}$ is the free category with colimits on the objects of $\A$.  There is a natural inclusion $\Ob(\A) \hookrightarrow \Ob(\tilde{\A})$, and thus an essentially unique  functor 
\[      \pi \maps \tilde{\A} \to \A \]
sending objects of $\A$ to themselves and preserving finite colimits.  
As one would expect, $\pi$ induces a symmetric monoidal double functor 
\[      j \maps \lCsp(\tilde{A}) \to \lCsp(\A) . \]
Indeed, this is a special case of Theorem \ref{thm:structured_cospan_functoriality_2}.  

We can summarize the story so far as follows.  Let $\lD$ be a small symmetric monoidal structured/decorated cospan category.  Then there are symmetric monoidal double functors 
\[            \lCsp(\tilde{\lD}_0) \xrightarrow{j} \lCsp(\lD) \xrightarrow{\iota} \lD .\]
We call $\lCsp(\lD)$ the `exoskeleton' of $\lD$, and $\lCsp(\tilde{\lD}_0)$ the `outer shell'.    Loose morphisms in the exoskeleton are cospans in the tight subcategory of $\lD$, while loose morphisms in the outer shell are cospans that can be defined using only the objects of $\lD$ and general features of categories with finite colimits.
In a sense, we obtain the outer shell by stripping the exoskeleton of most of its personality, which comes from its morphisms.

We expect that this story can be carried through for a larger class of double categories, with structured and decorated cospans being just examples.  Libkind and Myers' work on variable sharing theories suggests a way to generalize.  We turn to this next.

\iffalse One can even try to remove the symmetric monoidal aspects, though they are important in practice.

\begin{problem}
Define double functors
\[            \lCsp(\tilde{\lD}_0) \xrightarrow{j} \lCsp(\lD) \xrightarrow{\iota} \lD \]
for a more general class of double categories $\lD$ in a way that reduces to the constructions 
already given when $\lD$ is a small symmetric monoidal structured/decorated cospan category.
For example, can this setup be generalized to the case where $\lD$ is a cofibrant double category
for which $\lD_0$ has finite colimits?
\end{problem}
\fi

\subsection{Hypergraph double categories}
\label{subsec:hypergraph}

We have just seen that any symmetric monoidal structured/decorated cospan double category comes equipped with maps from two simpler double categories: its exoskeleton and its outer shell.   In Section \ref{subsec:frobenius} we dig into the rich structure that still remains in the outer shell.  Libkind and Myers take a complementary approach, developing a general framework to unify structured and decorated cospans, and also study open systems in the input-output paradigm \cite{LM,MyersBook}.    Here we introduce their line of thinking with the example of structured cospans.

In Section \ref{sec:introduction} we asked why an open system should have just two interfaces, and answered that this was just a matter of convenience.  Libkind and Myers go further and consider open systems with just \emph{one} interface.   Suppose $\lD$ is a symmetric monoidal structured cospan category.   It turns out that we can recover all of $\lD$ from cospans with trivial left interface:
\[
\scalebox{1.2}{
\begin{tikzpicture}[scale=1.0]
\node (A) at (0,0) {$0$};
\node (B) at (1,1) {$x$};
\node (C) at (2,0) {$L(b)$};
\path[->,font=\scriptsize,>=angle 90]
(A) edge node[above,pos=0.3]{$!$} (B)
(C)edge node[above,pos=0.3]{$o$}(B);
\end{tikzpicture}
}
\]
together with certain operations we can do on these: namely, tensoring them and composing them on the right with cospans coming from the outer shell.

First, we can take an arbitrary structured cospan
\[
\scalebox{1.2}{
\begin{tikzpicture}[scale=1.0]
\node (A) at (0,0) {$L(a)$};
\node (B) at (1,1) {$x$};
\node (C) at (2,0) {$L(b)$};
\path[->,font=\scriptsize,>=angle 90]
(A) edge node[above,pos=0.3]{$i$} (B)
(C)edge node[above,pos=0.3]{$o$}(B);
\end{tikzpicture}
}
\]
and `fold' it to get a structured cospan with trivial left interface:
\[
\scalebox{1.2}{
\begin{tikzpicture}[scale=1.0]
\node (A) at (0,0) {$0$};
\node (B) at (1,1) {$x$};
\node (C) at (2,0) {$L(a+b).$};
\path[->,font=\scriptsize,>=angle 90]
(A) edge node[above,pos=0.3]{$!$} (B)
(C)edge node[right,pos=0.7]{$\;\langle i, o\rangle$}(B);
\end{tikzpicture}
}
\]
Here $\langle i, o \rangle$ is the copairing of $i$ and $o$ composed with the isomorphism $L(a+b) \cong L(a) + L(b)$.    

Next, suppose we want to compose two structured cospans:
\begin{equation}
\label{eq:cospan_1}
\scalebox{1.2}{
\begin{tikzpicture}[baseline=(current  bounding  box.center)]
\node (A) at (0,0) {$L(a)$};
\node (B) at (1,1) {$x$};
\node (C) at (2,0) {$L(b)$};
\node (A') at (3,0) {$L(b)$};
\node (B') at (4,1) {$y$};
\node (C') at (5,0) {$L(c)$};
\path[->,font=\scriptsize,>=angle 90]
(A) edge node[above,pos=0.3]{$i$} (B)
(C)edge node[above,pos=0.3]{$o$}(B)
(A') edge node[above,pos=0.3]{$i'$} (B')
(C')edge node[above,pos=0.3]{$o'$}(B');
\end{tikzpicture}
}
\end{equation}
using only their folded versions:
\[
\scalebox{1.2}{
\begin{tikzpicture}[scale=1.0]
\node (A) at (0,0) {$0$};
\node (B) at (1,1) {$x$};
\node (C) at (2,0) {$L(a+b)$};
\node (A') at (3,0) {$0$};
\node (B') at (4,1) {$y$};
\node (C') at (5,0) {$L(b+c)$.};
\path[->,font=\scriptsize,>=angle 90]
(A) edge node[above,pos=0.3]{$!$} (B)
(C)edge node[right,pos=0.7]{\, $\langle i, o \rangle$}(B)
(A') edge node[above,pos=0.3]{$!$} (B')
(C')edge node[right,pos=0.7]{\, $\langle i', o' \rangle$}(B');
\end{tikzpicture}
}
\]
To do this we first tensor their folded versions, getting this:
\begin{equation}
\label{eq:cospan_2}
\scalebox{1.2}{
\begin{tikzpicture}[baseline=(current  bounding  box.center)]
\node (A) at (0,0) {$0$};
\node (B) at (1,1) {$x+y$};
\node (C) at (2,0) {$L(a+b+b+c).$};
\path[->,font=\scriptsize,>=angle 90]
(A) edge node[above,pos=0.3]{$!$} (B)
(C)edge node[right,pos=0.7]{\;\, $\langle i, o\rangle + \langle i', o' \rangle$}(B);
\end{tikzpicture}
}
\end{equation}
We then compose this with the following cospan:
\begin{equation}
\label{eq:cospan_3}
\scalebox{1.2}{
\begin{tikzpicture}[baseline=(current  bounding  box.center)]
\node (A) at (0,0) {$L(a+b+b+c)$};
\node (B) at (1,1) {$L(a+b+c)$};
\node (C) at (2,0) {$L(a+c)$};
\path[->,font=\scriptsize,>=angle 90]
(A) edge node[above,pos=0.3]{$$} (B)
(C)edge node[right,pos=0.7]{\;\, $$}(B);
\end{tikzpicture}
}
\end{equation}
whose left leg is built using the codiagonal $\nabla \maps b + b \to b$, and whose right leg is built using the unique map $! \maps 0 \to b$.    

One can show that the result of composing cospans \eqref{eq:cospan_2} and \eqref{eq:cospan_3} is
\[
\scalebox{1.2}{
\begin{tikzpicture}[scale=1.0]
\node (A) at (0,0) {$0$};
\node (B) at (1,1) {$x+_{L(b)} y$};
\node (C) at (2,0) {$L(a+c).$};
\path[->,font=\scriptsize,>=angle 90]
(A) edge node[above,pos=0.3]{$!$} (B)
(C) edge node[right,pos=0.7]{\;\, $\langle i'', o''\rangle$}(B);
\end{tikzpicture}
}
\]
This is the folded version of the cospan
\[
\scalebox{1.2}{
\begin{tikzpicture}[scale=1.0]
\node (A) at (0,0) {$L(a)$};
\node (B) at (1,1) {$x+_{L(b)} y$};
\node (C) at (2,0) {$L(c)$};
\path[->,font=\scriptsize,>=angle 90]
(A) edge node[above,pos=0.3]{$i''$} (B)
(C) edge node[right,pos=0.7]{$o''$}(B);
\end{tikzpicture}
}
\]
that is the composite of the cospans in Equation (\ref{eq:cospan_1}).   To prove this, we do a computation showing that $i''$ and $o''$ can also be defined using the pushout diagram that defines the desired composite:
\[
\scalebox{1.2}{
\begin{tikzpicture}[scale=1.0]
\node (A) at (0,0) {$L(a)$};
\node (B) at (1,1) {$x$};
\node (C) at (2,0) {$L(b)$};
\node (D) at (3,1) {$y$};
\node (E) at (4,0) {$L(c)$};
\node (F) at (2,2) {$x+_{L(b)}y$};
\path[->,font=\scriptsize,>=angle 90]
(A) edge node[above,pos=0.3]{$i$}(B)
(C) edge node[above,pos=0.3]{$o$}(B)
(C) edge node[above,pos=0.3]{$i'$}(D)
(E) edge node[above,pos=0.3]{$o'$}(D)
(B) edge node[above,pos=0.3]{$$}(F)
(D) edge node[above,pos=0.3]{$$}(F)
(A) edge [out=90, in=180, looseness=0.7] node[above,pos=0.5]{$i''$} (F)
(E) edge [in=0, out=90, looseness=1.0] node[right,pos=0.5]{$o''$}(F);
\end{tikzpicture}
}
\]

In short, composing structured cospans in folded form requires only tensoring them and then composing the result with the cospan in Equation \eqref{eq:cospan_3}.  But this cospan comes from the outer shell.  To see this, note that if we take the following cospan in the outer shell $\tilde{\lD}_0$:
\[
\scalebox{1.2}{
\begin{tikzpicture}[baseline=(current  bounding  box.center)]
\node (A) at (0,0) {$a+b+b+c$};
\node (B) at (1,1) {$a+b+c$};
\node (C) at (2,0) {$a+c$};
\path[->,font=\scriptsize,>=angle 90]
(A) edge node[left,pos=0.5]{$1 + \nabla + 1$\;\,} (B)
(C) edge node[right,pos=0.5]{$1 + ! + 1 $}(B);
\end{tikzpicture}
}
\]
and map it into $\lD$ using the double functor discussed in Section \ref{subsec:outer_shell}:
\[            \lCsp(\tilde{\lD}_0) \xrightarrow{j} \lCsp(\lD) \xrightarrow{\iota} \lD ,\]
we get the cospan in \eqref{eq:cospan_3}.   The double functor $j$ converts formal colimits to colimits in the tight category of $\lD$; then $\iota$ applies $L$ to everything, as in Lemma \ref{lem:structured_cospan_interfaces}.  

To generalize this sort of observation, Libkind and Myers define a `symmetric monoidal loose right module' of a symmetric monoidal double category \cite[Sec.\ 4]{LM}.   Structured cospans with trivial left interface, and maps between these, should form a symmetric monoidal loose right module of the outer shell of $\lD$.   The argument above suggests that the whole symmetric monoidal double category $\lD$ can be recovered from this module.  The same is probably also true for decorated cospans.  However, I have not seen proofs of these claims.

\begin{problem}
Read Libkind and Myers \cite{LM} and show that if $\lD$ is a symmetric monoidal stuctured/decorated cospan double category, then loose morphisms with trivial left interface, and maps between these, give a symmetric monoidal loose right module of the outer shell $\lCsp(\tilde{\lD}_0)$.   Then show that $\lD$ can be recovered from this symmetric monoidal loose right module.
\end{problem}

If the first part of this challenge can be met, every symmetric monoidal structured or decorated cospan double category will be a `hypergraph double category' according to the following definition, which categorifies Fong and Spivak's definition of hypergraph category \cite{FS}.   

\begin{tentativedefn}
A \define{hypergraph double category} is a symmetric monoidal loose right module of any symmetric monoidal double category of the form $\lCsp(\X)$, where $\X$ is the free category with finite colimits on some set.
\end{tentativedefn}

In \cite[Ex.\ 8.11]{LM}, Libkind and Myers proposed a more general definition of hypergraph double category, allowing $\X$ to be the free category with colimits on any small category.    

\subsection{Frobenius structures}
\label{subsec:frobenius}

Since any structured or decorated cospan double category comes equipped with a map from a cospan double category, some features of ordinary cospans are automatically inherited by structured and decorated cospans.    
For example, when $\lD$ is a symmetric monoidal structured/decorated cospan double category, every object $a \in \lD$ is equipped with loose morphisms
\[
  \xymatrixrowsep{1pt}
  \xymatrixcolsep{30pt}
  \xymatrix{
    \mult{.075\textwidth} & \unit{.075\textwidth} &
    \comult{.075\textwidth} & \; \counit{.075\textwidth} 
    \\
    \mu\colon a+a \slashedrightarrow a & \eta\colon 0 \slashedrightarrow a &
    \delta\colon a \slashedrightarrow a+a & \epsilon\colon a \slashedrightarrow 0
  }
\]
that obey the laws of a special commutative Frobenius monoid up to isomorphism.  Thus, the object comes with chosen invertible 2-cells
\[
  \xymatrixrowsep{1pt}
  \xymatrixcolsep{25pt}
  \xymatrix{
    \assocr{.1\textwidth} \; \raisebox{0.2 em}{$\cong$} \;  \assocl{.1\textwidth} & 
    \unitr{.1\textwidth} \; \raisebox{0.2 em}{$\cong$} \; \idone{.1\textwidth} \; \raisebox{0.2 em}{$\cong$} \;  \unitl{.1\textwidth} 
  }
\]
\[
  \xymatrixrowsep{1pt}
  \xymatrixcolsep{25pt}
  \xymatrix{
   \commute{.1\textwidth} \; \raisebox{0.2 em}{$\cong$} \;   \mult{.07\textwidth} &
    \frobs{.1\textwidth} \; \raisebox{0.2 em}{$\cong$} \;  \frobx{.1\textwidth} 
    }
\]
and their left-right reflected versions, and these 2-cells obey a bevy of coherence laws.   All these coherence laws arise from an analysis of cospan double categories. 

Before delving into this, let us begin with some generalities.   First, any double category $\lD$ has a \define{loose bicategory}, denoted $\bD$, in which:
\begin{itemize}
\item objects are objects of $\lD$,
\item morphisms are loose morphisms of $\lD$,
\item 2-morphisms are 2-cells of $\lD$ for which the
source and target tight morphisms are identities,
\item composition of morphisms is given by composition of loose morphisms in $\lD$,
\item vertical and horizontal composition of 2-morphisms are given by tight and loose
composition of 2-cells in $\lD$, respectively.
\end{itemize}
Second, any bicategory $\bD$ has a \define{decategorification}, a category $\D$ in which:
\begin{itemize}
\item objects are objects of $\bD$,
\item morphisms are isomorphism classes of morphisms of $\bD$.
\end{itemize}
Hansen and Shulman have shown that if $\lD$ is a \emph{fibrant} symmetric monoidal double category, $\bD$ naturally becomes a symmetric monoidal bicategory \cite{HS,Shulman2010}.  It then follows easily that $\D$ becomes a symmetric monoidal category.   

Now let $\A$ be a category with finite colimits.  Merely by virtue of $\A$ having finite coproducts, every object $a \in \A$ becomes a commutative monoid internal to $(\A, +)$.   The unit for this monoid is the unique morphism $! \maps 0 \to a$ in $\A$, while the multiplication is the codiagonal $\nabla \maps a \to a + a$.    These morphisms give cospans in $\A$
\[
\scalebox{1.2}{
\begin{tikzpicture}[scale=0.8]
\node (A) at (0,0) {$0$};
\node (B) at (1,1) {$a$};
\node (C) at (2,0) {$a$};
\path[->,font=\scriptsize,>=angle 90]
(A) edge node[above,pos=0.3]{$!$} (B)
(C)edge node[above,pos=0.25]{$1$}(B);
\end{tikzpicture}
\qquad
\begin{tikzpicture}[scale=0.8]
\node (A) at (0,0) {$a + a$};
\node (B) at (1,1) {$a$};
\node (C) at (2,0) {$a$};
\path[->,font=\scriptsize,>=angle 90]
(A) edge node[above,pos=0.3]{$\nabla$} (B)
(C)edge node[above,pos=0.25]{$1$}(B);
\end{tikzpicture}
}
\]
that we call the \define{unit} $\eta \maps 0 \slashedrightarrow a$ and  \define{multiplication} $\mu \maps a+a \slashedrightarrow a$, respectively.    In terms of string diagrams, these look like `adding an extra wire that doesn't connect with anything' and `joining wires', respectively:
\[
  \xymatrixrowsep{1pt}
  \xymatrixcolsep{30pt}
  \xymatrix{
    \unit{.075\textwidth}  & \mult{.075\textwidth} 
    \\
    \eta\colon 0 \slashedrightarrow a  &  \mu\colon a+a \slashedrightarrow a
  }
\]
 Since $!$ and $\nabla$ obey the commutative monoid laws, but pushouts are defined only up to isomorphism, we expect that $\mu$ and $\eta$ obey the commutative monoid laws only up to coherent isomorphism.    They should thus make $a$ into a `symmetric pseudomonoid', a concept that can be defined in any symmetric monoidal bicategory \cite[Def.\ 17]{DS}. 
 
Because $\A$ has finite colimits, we can form the symmetric monoidal double category $\lCsp(\A)$ as in Lemma \ref{lem:cospans}.   Since this is fibrant, Hansen and Shulman's work implies that the loose bicategory $\mathbf{Csp}(\A)$ is symmetric monoidal.   The loose morphisms $\mu$ and $\eta$ live in this symmetric monoidal bicategory.   And indeed, they make $a$ into a symmetric pseudomonoid in $\mathbf{Csp}(\A)$.  In more detail:

\begin{thm}  Let $\A$ be a category with finite colimits, and let $a \in \A$.   The cospans $\eta\colon 0 \slashedrightarrow a$ and $\mu\colon a+a \slashedrightarrow a$ defined as above make $a$ into symmetric pseudomonoid in $\mathbf{Csp}(\A)$.   In more detail:
\begin{itemize}
\item The multiplication $\mu$ obeys the associative law up to an isomorphism of cospans, the \define{associator} $\alpha \maps \mu(\mu + 1_a) \simRightarrow \mu(1_a + \mu)$.
\[
  \xymatrix{
    \assocr{.1\textwidth}\; \xRightarrow{\alpha} \; \assocl{.1\textwidth}
    }
\]
This associator can be chosen to obey the pentagon identity.
\item The unit $\eta$ obeys the left and right unit laws up to isomorphisms called the \define{left unitor}  $\lambda \maps \mu(\eta + 1_a) \simRightarrow 1_a$ and \define{right unitor} $\rho \maps \mu(1_a + \eta) \simRightarrow 1_a$.  
\[
  \xymatrixrowsep{1pt}
  \xymatrixcolsep{25pt}
  \xymatrix{
    \unitr{.1\textwidth}  \;  \raisebox{0.2 em}{$\xRightarrow{\lambda}$} \;  \idone{.1\textwidth} 
    &  \unitl{.1\textwidth}  \;  \raisebox{0.2 em}{$\xRightarrow{\rho}$} \;   \idone{.1\textwidth} 
  }
\]
Together with the associator, these can be chosen to obey the triangle identity.
\item The multiplication $\mu$ obeys the commutative law up to an isomorphism called the \define{symmetry} $\sigma \maps \mu s \simRightarrow \mu$, where $s \maps a + a \slashedrightarrow a + a$ is the symmetry in $\mathbf{Csp}(\A)$.  
\[
  \xymatrix{
   \commute{.1\textwidth} \;  \raisebox{0.2 em}{$\xRightarrow{\sigma}$} \;  \mult{.07\textwidth}
    }
\]
Together with the associator, this can be chosen to obey the hexagon identities.
\end{itemize}
\end{thm}

\begin{proofsketch}
The associator, unitors and symmetry can all be defined using the universal property of pushouts, and the uniqueness clause in this universal property implies that the pentagon, triangle and hexagon diagrams commute.
\end{proofsketch}

To fit this into a larger context, note that what we have done here is to turn two tight morphisms $!$ and $\nabla$ in $\lCsp(\A)$ into their companions, $\eta$ and $\mu$.   The process of turning tight morphisms into their companions can be shown to define a symmetric monoidal pseudofunctor 
\[                   \text{comp} \maps \A \to \textbf{Csp}(\A) .  \]
Thus it sends commutative monoids in $\A$ to symmetric pseudomonoids in $\textbf{Csp}(\A)$ \cite[Prop.\ 16]{DS}.   

Similarly, if we take the conjoints of $! \maps 0 \to a$ and $\nabla \maps a \to a+a$ we obtain cospans
\[
\scalebox{1.2}{
\begin{tikzpicture}[scale=0.8]
\node (A) at (0,0) {$a$};
\node (B) at (1,1) {$a$};
\node (C) at (2,0) {$a+a$};
\path[->,font=\scriptsize,>=angle 90]
(A) edge node[above,pos=0.3]{$1$} (B)
(C)edge node[above,pos=0.3]{$\nabla$}(B);
\end{tikzpicture}
\qquad
\begin{tikzpicture}[scale=0.8]
\node (A) at (0,0) {$a$};
\node (B) at (1,1) {$a$};
\node (C) at (2,0) {$0$};
\path[->,font=\scriptsize,>=angle 90]
(A) edge node[above,pos=0.3]{$1$} (B)
(C)edge node[above,pos=0.3]{$!$}(B);
\end{tikzpicture}
}
\]
which we call the \define{comultiplication} $\delta \maps a \slashedrightarrow a+a$ and \define{counit} $\epsilon \maps a \slashedrightarrow 0$.   In terms of string diagrams, these look like `splitting a wire' and `capping off a wire':
\[
  \xymatrixrowsep{1pt}
  \xymatrixcolsep{30pt}
  \xymatrix{
    \mult{.075\textwidth} & \unit{.075\textwidth} 
    \\
    \mu\colon a+a \slashedrightarrow a & \eta\colon 0 \slashedrightarrow a 
  }
\]
Since taking conjoints is a symmetric monoidal pseudofunctor
\[                   \text{conj} \maps \A^{\text{op}}  \to \textbf{Csp}(\A) \]
it sends $a$, which is a commutative comonoid in  $(\A^{\text{op}}, +)$, to a symmetric pseudocomonoid in $\textbf{Csp}(\A)$, with $\delta$ as comultiplication and $\epsilon$ as counit.

How, and why, do the pseudomonoid and pseudocomonoid structures on any object $a \in \mathbf{Csp}(\A)$ interact?   They do so because the multiplication $\mu \maps a + a \slashedrightarrow a$ followed by the counit $ \epsilon \maps a \slashedrightarrow 0$ gives an important cospan called the \define{cup}:
\[
\scalebox{1.2}{
\begin{tikzpicture}[scale=0.8]
\node (A) at (0,0) {$a+a$};
\node (B) at (1,1) {$a$};
\node (C) at (2,0) {$0$};
\path[->,font=\scriptsize,>=angle 90]
(A) edge node[above,pos=0.3]{$\nabla$} (B)
(C)edge node[above,pos=0.3]{$!$}(B);
\end{tikzpicture}
}
\]
which we can draw as this string diagram:
\[
  \xymatrixrowsep{1pt}
  \xymatrixcolsep{30pt}
  \xymatrix{
 \cuptwo{.075\textwidth} 
    \\
    \cup \colon a+a \slashedrightarrow 0.
  }
\]
Similarly, we can define the \define{cap} to be the unit $\iota \maps 0 \slashedrightarrow a$ followed by the comultiplication $\delta \maps a \slashedrightarrow a + a$:
\[
\scalebox{1.2}{
\begin{tikzpicture}[scale=0.8]
\node (A) at (0,0) {$0$};
\node (B) at (1,1) {$a$};
\node (C) at (2,0) {$a+a$};
\path[->,font=\scriptsize,>=angle 90]
(A) edge node[above,pos=0.3]{$!$} (B)
(C)edge node[above,pos=0.3]{$\nabla$}(B);
\end{tikzpicture}
}
\]
which we can draw as follows:
\[
  \xymatrixrowsep{1pt}
  \xymatrixcolsep{30pt}
  \xymatrix{
 \captwo{.075\textwidth} 
    \\
    \cap \colon 0 \slashedrightarrow a+a 
  }
\]
In terms of electrical circuits, these give the ability to bend wires.  Mathematically they make $a$ into its own dual in $\lCsp(\A)$.  

For any object $a \in \lCsp(\A)$, the cap and cup obey the usual zigzag identities up to isomorphisms that in turn obey coherence laws called the swallowtail equations \cite[Cor.\ 5.8]{Stay}.   But in fact, as long as the zigzag identities hold up to isomorphism, the isomorphisms can always be tweaked to make them obey the swallowtail equations \cite[Thm.\ 2.7]{Pstragowski}.    Thus, in general, we can define a \define{Frobenius} pseudomonoid in a monoidal bicategory to be a pseudomonoid for which the cap and cup obey the zigzag identities up to isomorphism.   For an equivalent alternative definition, see \cite{Street}.

A \define{symmetric} Frobenius pseudomonoid is one whose underlying pseudomonoid is symmetric.  One can show this implies that its underlying pseudocomonoid is also symmetric in a dual sense, involving an isomorphism
\[
   \xymatrixrowsep{1pt}
  \xymatrixcolsep{30pt}
  \xymatrix{
   \cocommute{.1\textwidth} \;  \raisebox{0.2 em}{$\simRightarrow$} \;  \comult{.07\textwidth} \\
 \;  \; s \delta \simRightarrow \delta.
    }
\]
So far, we have seen that whenever $\A$ is a category with finite colimits, every object in $\lCsp(\A)$ is a symmetric Frobenius pseudomonoid in $\mathbf{Csp}(\A)$.   

But there is more!   The multiplication in this Frobenius pseudomonoid is left biadjoint to the comultiplication, with the counit $\varepsilon$ of the biadjunction being an isomorphism:
\[
   \xymatrixrowsep{1pt}
  \xymatrixcolsep{30pt}
  \xymatrix{
    \spec{.1\textwidth} \,\raisebox{0.2 em}{$\simRightarrow$} \, \idone{.1\textwidth}
    \\
 \; \varepsilon \maps  \mu \delta \simRightarrow 1_a.
    }
\]
We call a Frobenius pseudomonoid with this property \define{special}, extending the existing terminology where a Frobenius monoid is called special if
\[
  \xymatrix{
    \spec{.1\textwidth} \,\raisebox{0.2 em}{=} \, \idone{.1\textwidth}.
    }
\]
We do not know if there are additional coherence laws that the counit $\varepsilon$ obeys, which should be incorporated in the definition of `special'.

Following tradition we have spoken about pseudomonoids in a monoidal bicategory $\mathbf{D}$, which happens to be the loose bicategory of a monoidal double category $\lD$.   However, in this situation we might as well call these pseudomonoids in $\lD$.   Thus, we have shown:

\begin{thm}
Let $\A$ be a category with finite colimits.   Then the above structures make any object $a \in \A$ into a special symmetric Frobenius pseudomonoid in $\lCsp(\A)$.
\end{thm}

What we have seen for cospans, now follows for structured or decorated cospans:
 
\begin{thm}
Let $\lD$ be a symmetric monoidal structured/decorated cospan double category.  Then any object $a \in \lD$ becomes a special symmetric Frobenius pseudomonoid in $\lD$.
\end{thm}

\begin{proof}
Given a category $\A$ with finite colimits, we have seen how any object $a \in \A$ gives a special symmetric Frobenius pseudomonoid $a \in \lCsp(\A)$.   This holds in particular when $\A = \lD_0$ where $\lD$ is a symmetric monoidal structured/decorated cospan double category.    Applying the symmetric monoidal pseudofunctor $\iota \maps \Csp(\lD_0) \to \lD$, we obtain a special symmetric Frobenius pseudomonoid structure on the object $a$ regarded as an object of $\lD$.
\end{proof}

At the end of Section \ref{subsec:hypergraph}, following ideas of Libkind and Myers \cite{LM}, we proposed a definition of `hypergraph double category' aimed at isolating the most important common features of symmetric monoidal structured and decorated cospan double categories.   This definition categorified Fong and Spivak's slick definition of hypergraph category \cite{FS}.

We are now in a position to pursue a different definition of hypergraph double category, which categorifies the `old-fashioned' definition of a hypergraph category, Definition \ref{defn:hypergraph_category}.   Naively, we could define a \define{hypergraph double category} to be a symmetric monoidal double category $\lD$ where each object is equipped with the structure of a special symmetric Frobenius pseudomonoid, with these structures  compatible with the monoidal structure of $\lD$ in the sense that they obey all the equations in Definition \ref{defn:hypergraph_category}.   However, it may be overly strict to demand equations here.  Furthermore, besides having a multiplication, unit, comultiplication and counit, a symmetric Frobenius pseudomonoid is equipped with extra structure, namely various 2-isomorphisms.   Presumably these, too, must be compatible with the monoidal structure of $\lD$.

Thus, there are some coherence issues to sort out, of the kind only higher category theorists truly enjoy.
We leave this as a challenge:

\begin{problem}  
Find the correct definition of  `hypergraph double category' along the following lines: it is a symmetric monoidal double category $\lD$ where each object has the structure of a special symmetric Frobenius pseudomonoid and these structures are compatible with the symmetric monoidal structure of $\lD$.   To test the correctness of a candidate definition, show that every symmetric monoidal structured/decorated cospan category is a hypergraph double category according to this definition.   If possible, also show that this definition is equivalent to the one proposed at the end of Section \ref{subsec:hypergraph}.
\end{problem}

To conclude, let us recall why this issue is important.  We have some intuition of what open systems are in the variable sharing paradigm.  Structured and decorated cospans gives us \emph{examples}.  But what is the right \emph{general concept} of a symmetric monoidal double category of open systems in the variable sharing paradigm?  It may be a hypergraph double category.   But what, exactly, is that?  We have outlined two approaches to defining this concept.  Ideally they will lead to equivalent definitions.

%\section{Conclusions}\label{sec:conclusions}

\subsection*{Acknowledgements} 

The work described here is due to many people.  I would especially like to acknowledge my collaborators in work on decorated and structured cospans and the mathematics of open systems, including Brandon Coya, Kenny Courser, Jarson Erbele, Brendan Fong, Fabrizio Genovese, Xiaoyan Li, Jade Master, Joe Moeller, Nathaniel Osgood, Sophie Libkind, Evan Patterson, Blake Pollard, Mike Shulman, Mike Stay, and Christina Vasikakopoulou.   I thank Nathanael Arkor, Kevin Carlson and others on the Category Theory Community Server for helpful conversations about double categories, and David Jaz Myers for persistently pushing forward the theory of open systems.

\subsection*{Disacknowledgment}

The published version of this article is licensed under a \href{https://creativecommons.org/licenses/by/4.0/}{Creative Commons Attribution 4.0 International License}, which permits use, sharing, adaptation, distribution and reproduction in any medium or format, as long as one gives appropriate credit to the original author(s) and the source, provides a link to the Creative Commons licence, and indicates if changes were made.  

The version here is identical to the published article except for details of formatting.  This is clearly permitted by the license.   However, because the publisher, Springer Nature, is part of an oligopoly that feels all-powerful, they insisted at the last minute that I include here an arbitrary statement of their choosing, even though it contradicts my rights as listed above and I do not accept it.   Here is that statement:

\begin{quote}
This preprint has not undergone peer review (when applicable) or any post-submission improvements or corrections. The Version of Record of this article is published in Applied Categorical Structures, and is available online at \href{https://doi.org/10.1007/s10485-026-09885-9}{https://doi.org/10.1007/s10485-026-09885-9}.
\end{quote}

This is exactly the kind of thing that has made me boycott Springer.   I made an exception for the Bob Par\'e 80th birthday celebration, which I now regret.    Anyone wishing to keep their paper on an open-access repository should avoid Springer.

\end{document}